\documentclass[11pt,a4paper,reqno]{amsart}
\usepackage{geometry}
\usepackage[T1]{fontenc}
\usepackage[utf8]{inputenc}

\usepackage{amsfonts}
\usepackage{amsmath}
\usepackage{amsthm}
\usepackage{amssymb}
\usepackage{mathrsfs}
\usepackage[numbers]{natbib}
\usepackage[fit]{truncate}
\usepackage{orcidlink}
\usepackage{subcaption}

\usepackage{mathtools} 
\usepackage{enumitem}  
\usepackage{lmodern}
\usepackage{microtype}
\usepackage{pgfplots}
\pgfplotsset{compat=1.18}
\usepackage{booktabs}

\usepackage{hyperref}
\definecolor{linkblue}{RGB}{0,71,171}
\hypersetup{
	colorlinks = true,
	linkcolor  = linkblue,
	citecolor  = linkblue,
	urlcolor   = linkblue
}

\mathtoolsset{showonlyrefs} 
\numberwithin{equation}{section}
\allowdisplaybreaks

\theoremstyle{plain}
\newtheorem{theorem}{Theorem}[section]
\newtheorem{corollary}[theorem]{Corollary}
\newtheorem{lemma}[theorem]{Lemma}

\newtheorem{proposition}[theorem]{Proposition}

\theoremstyle{definition}
\newtheorem{definition}[theorem]{Definition}
\newtheorem{example}[theorem]{Example}

\newtheorem{remark}[theorem]{Remark}

\DeclareMathOperator{\Esp}{E}
\DeclareMathOperator{\Prb}{P}

\DeclareMathOperator{\dist}{dist}

\newcommand{\mc}[0]{\mathcal}

\newcommand{\mbf}[0]{\mathbf}

\newcommand{\ol}[0]{\overline}

\newenvironment{enuroman}{\begin{enumerate}[label=(\roman*)]}{\end{enumerate}}

\newcommand{\rd}{\mathrm{d}}
\newcommand{\vd}{\,\mathrm{d}}

\newcommand{\indic}[1]{\mathbf{1}_{#1}}
\newcommand{\indicB}[1]{\mathbf{1}_{\{#1\}}}
\newcommand{\loct}[3]{L^{#2}_{#3}(#1)}

\newcommand{\braces}[1]{ \left({#1}\right) } 
\newcommand{\sqbraces}[1]{ \left[{#1}\right] } 
 
\newcommand{\xnorm}[2]{ \left\|{#1}\right\|_{#2}} 
\newcommand{\abs}[1]{ \left|{#1}\right|} 
 
\newcommand{\qv}[1]{ \left\langle {#1} \right\rangle }

\newcommand{\hemail}[1]{\href{mailto:{#1}}{\normalfont \texttt{#1}}}

\newenvironment{cdate}{
	\noindent
	\textbf{\small This version:}
	\small
}{\par}

\newenvironment{keys}{
	\noindent
	\textbf{\small Keywords:}
	\small
}{\par}

\newenvironment{msc}{
	\noindent
	\textbf{\small MSC 2020 codes:}
	\small
}{\par}

\begin{document}
	
	\begin{abstract}
Let $A$ be a positive continuous additive functional of a strong Markov process and $\alpha$ an independent finite-state Markov chain. For bounded $h$, we study
\[
J_\varepsilon(t)=\int_0^t\bigl(h(\alpha_{s/\varepsilon})-\pi(h)\bigr)\,dA_s .
\]
If $\sup_y \mathrm E_y A_t\le C_T t^\vartheta$, then for every bounded random time determined by the base process, not necessarily a stopping time,
\[
\mathrm E|J_\varepsilon(\tau)|^2=O(\varepsilon^\vartheta),\qquad
\mathrm E\sup_{t\le T}|J_\varepsilon(t)|^2
=O\!\left(\varepsilon^\vartheta(1+|\log\varepsilon|)^2\right).
\]
For reflected Brownian boundary local time, the exponent $1/2$ is sharp at fixed time.

We apply this principle to reflected Brownian motion with rapidly switching Wentzell boundary dynamics. The homogenized boundary condition is again of Wentzell type, with stickiness and killing coefficients given by invariant averages. For bounded Lipschitz observables, the quenched Feynman--Kac error has every polynomial rate below $1/8$, uniformly over deterministic times and in expected uniform norm. For smooth Neumann-compatible observables with proportional killing and stickiness, the rate improves to $O(\varepsilon^{1/4})$, up to a logarithmic factor in the uniform-in-time estimate.

We also construct the switching process, derive quenched and annealed backward evolutions, and transfer the estimates by duality to the forward measure evolution. For the forward evolution, we obtain quantitative convergence in bounded-Lipschitz distance and the sharper $1/4$ rate for smooth observables, including surviving mass. Numerical experiments illustrate the finite-scale convergence.
	\end{abstract}
	
	\title[Homogenization of Markov-modulated additive functionals]{Quantitative averaging of Markov-modulated additive functionals and Wentzell boundary homogenization}
	
	\author{Alexis Anagnostakis}
	\address{Alexis Anagnostakis -- Centro de Modelamiento Matemático (CNRS IRL2807), Universidad de Chile, Santiago, Chile}
	\email{\hemail{aanagnostakis@cmm.uchile.cl}}
	\urladdr{\orcidlinkf{0009-0000-3306-9766}}
	
	\author{Fausto Colantoni}
	\address{Fausto Colantoni -- Institute of Mathematical Finance, Ulm University, Ulm, Germany}
	\email{\hemail{fausto.colantoni@uni-ulm.de}}
	\urladdr{\orcidlinkf{0000-0001-5123-5916}}
	
	\maketitle
	
	\begin{quote}
		\begin{cdate}
			\today
		\end{cdate}
		
		\begin{msc}
			35B27; 35K20; 35R60; 60J27; 60J55; 92C20.
		\end{msc}
		
		\begin{keys}
			Markov-modulated additive functionals; quantitative averaging; fast switching; positive continuous additive functionals; boundary local time; Wentzell boundary conditions; reflected Brownian motion.
		\end{keys}
	\end{quote}
	
	\setcounter{tocdepth}{2}
	\tableofcontents
	
	\section{Introduction}
	
	A rapidly changing boundary is scarcely observed in its microscopic detail.
	One sees instead an effective boundary condition. 
	Robin conditions are the classical example: they interpolate between reflection and absorption and arise in a
	number of boundary-homogenization problems
	\cite{bensoussan1978homogenization,tanaka1984homogenization}. Probabilistically, a Robin condition describes reflected diffusion killed at a rate proportional to its boundary local time.
	
	The same sort of averaging can be observed under temporal switching. 
	Fast switching between Dirichlet and Neumann states can produce an effective Robin condition
	\cite{LawleyKeener}; Markov-switched parabolic equations and switching
	absorbing--reflecting boundaries have been studied in
	\cite{LMR,BressloffLawley}. The time-dependent Robin problem of
	\cite{colantoni2026_time_dependent_Robin} is particularly close to the present
	setting. Spatial alternation of Dirichlet and Neumann pieces gives a different,
	but related, class of homogenization problems \cite{FiloLuckhaus}.
	
	Robin conditions account for killing at the boundary, but not for a positive
	reversible residence time there. Wentzell\footnote{The name is transliterated from Russian in several ways in the literature, including \emph{Ventcel'}, \emph{Ventsel'}, and \emph{Venttsel'}. We use the conventional spelling \emph{Wentzell} throughout.} 
	conditions do. Their probabilistic
	counterparts are sticky diffusions: after reaching the boundary, the process
	may spend positive physical time there before returning to the interior. 
	This mechanism is classical in one dimension
	\cite{feller1952The_parabolic,ito1996Diffusion_processes,
		mandl1968Analytical_treatment}; multidimensional constructions and analytic
	treatments may be found in
	\cite{grothaus2017stochastic,dovidio2026_fractional_boundary,
		wentzell1960_on_boundary_conditions,FGGR}.
	
	\bigskip
	
	In this work, we consider the Wentzell condition without tangential diffusion. For a bounded
	domain $D\subset\mathbb R^d$ and a function $u$ on $\overline D$, the condition reads
	\begin{equation}
		\label{eq:intro-wentzell-condition}
		\frac{\rho}{2}\Delta u
		+
		\partial_{\mathbf n}u
		+
		\kappa u
		=
		0,
		\qquad
		\text{on }\partial D,
	\end{equation}
	where $\mathbf n$ denotes the outward unit normal and $\Delta u|_{\partial D}$ is the boundary trace of the Laplacian computed in the interior. Since the interior dynamics satisfy
	\[
	\partial_tu=\frac12\Delta u,
	\]
	condition~\eqref{eq:intro-wentzell-condition} is equivalently the dynamic boundary condition
	\[
	\rho\,\partial_tu
	+
	\partial_{\mathbf n}u
	+
	\kappa u
	=
	0,
	\qquad
	\text{on }\partial D.
	\]
	The coefficient $\rho$ is the stickiness parameter and measures boundary residence; $\kappa$ is the killing rate and measures elastic killing through boundary local time.
	
	We ask what remains of these two mechanisms when they both fluctuate rapidly. Let
	$D$ be connected with $C^2$ boundary, and consider the family
	\begin{equation}
		\label{eq:intro-switching-wentzell}
		\frac{\rho_i}{2}\Delta u
		+
		\partial_{\mathbf n}u
		+
		\kappa_i u
		=
		0,
		\qquad
		\text{on }\partial D,
		\qquad
		i\in S:=\{1,\ldots,N\}.
	\end{equation}
	In regime $i$, the parameters are $\rho_i$ and $\kappa_i$. The regime is an
	irreducible continuous-time Markov chain
	$\alpha=(\alpha_t)_{t\ge0}$ on $S$, with invariant law
	$\pi=(\pi_i)_{i\in S}$. The chain runs in the \emph{operational time} of the
	underlying reflected Brownian motion, not in physical time. We accelerate it by
	\[
	\alpha_{\cdot/\varepsilon},
	\qquad
	\varepsilon\to0.
	\]
	
	The averaging argument is formulated for a general positive continuous
	additive functional and is therefore not specific to boundary local time.
	Let $Y$ be a strong Markov process, $A$ a positive continuous additive functional of $Y$,
	and $\alpha$ an independent irreducible finite-state chain. 
	For bounded
	$h:S\to\mathbb R$, the quantity to control is
	\begin{equation}
		\label{eq:intro-general-AF}
		J_\varepsilon(t)
		:=
		\int_0^t
		\bigl(h(\alpha_{s/\varepsilon})-\pi(h)\bigr)\,dA_s.
	\end{equation}
	Assume the short-time growth estimate
	\begin{equation}
		\label{eq:AF_condition}
		\sup_y\Esp_y A_t\le C_Tt^\vartheta,\qquad\vartheta\in(0,1].
	\end{equation}
	Then, for every bounded random time $\tau$ measurable with respect to the base process,
	\begin{equation}
		\label{eq:AF_results}
		\Esp|J_\varepsilon(\tau)|^2\lesssim\varepsilon^\vartheta,
		\qquad
		\Esp\sup_{t\le T}|J_\varepsilon(t)|^2
		\lesssim\varepsilon^\vartheta(1+|\log\varepsilon|)^2.
	\end{equation}
	Example~\ref{ex:abstract-benchmarks} illustrates the two characteristic
	regimes: ordinary occupation gives \(\vartheta=1\), while Brownian local time
	gives \(\vartheta=1/2\). In the latter case, the terminal-time variance is
	computed explicitly, showing that the exponent \(1/2\) is attained.
	
	Quantitative averaging for additive and occupation functionals has a
	substantial literature in ordinary-time and discrete-time settings.
	For singularly perturbed Markov chains, Yin and Zhang
	\cite{yinzhang2007_singular} obtain strong diffusion approximations
	with convergence rates for centered and scaled functional occupation
	processes.
	Related random-evolution methods yield averaging and diffusion
	approximations for systems of additive-functional type under Markov
	or semi-Markov switching; see, for example,
	\cite{korolyuklimnios2004_average,limniosswishchuk2013_discrete}.
	The latter work also provides convergence-rate estimates formulated
	in terms of expectations, including applications to discrete-time
	additive functionals.
	
	The cited works primarily concern occupation integrals with respect to
	Lebesgue time, or additive sums generated by the switching process. 
	The present setting is
	different in that the fast environment is integrated against an independent
	positive continuous additive functional,
	\[
	\int_0^t h(\alpha_{s/\varepsilon})\,dA_s,
	\]
	whose Stieltjes measure may be singular with respect to Lebesgue measure.
	The short-time growth of \(A\) then determines the averaging rate.
	Our result isolates this mechanism at the level of general PCAFs and gives
	quantitative \(L^2\) estimates both at bounded random times measurable with
	respect to the base process and uniformly on compact time intervals. 
	
	For boundary local time in a bounded $C^2$ domain, \eqref{eq:AF_condition} is satisfied for $\vartheta=1/2$. Hence the
	squared $L^2$ error at a random time is of order $\varepsilon^{1/2}$, and the
	maximal estimate gives every power below $1/2$. 
	
	Returning to the switching boundary condition~\eqref{eq:intro-switching-wentzell}, the rapidly switching model converges to the deterministic Wentzell problem
	\begin{equation}
		\label{eq:intro-homogenized-wentzell}
		\frac{\overline\rho}{2}\Delta u
		+
		\partial_{\mathbf n}u
		+
		\overline\kappa u
		=
		0,
		\qquad
		\text{on }\partial D,
	\end{equation}
	where
	\begin{equation}
		\label{eq:intro-averaged-coefficients}
		\overline\rho
		:=
		\sum_{i\in S}\rho_i\pi_i,
		\qquad
		\overline\kappa
		:=
		\sum_{i\in S}\kappa_i\pi_i.
	\end{equation}
	The limit is again of Wentzell type, and no correction term appears: the
	effective coefficients are simply the ergodic averages of the microscopic ones.
	For every bounded Lipschitz observable and every bounded physical-time interval, the expected quenched error is $O(\varepsilon^\beta)$ for all
	$0<\beta<1/8$, uniformly over deterministic observation times. 
	The same range of exponents holds when the supremum in time is taken before expectation. 
	The logarithmic factor in~\eqref{eq:AF_results} does not alter the final polynomial threshold. 
	
	The main difficulty is to transfer the averaging estimates to the
	time-changed process. The switching coefficients first enter through
	additive functionals of the boundary local time, which determine the random
	clock and the killing functional. 
	Their fluctuations perturb the inverse clock. 
	Through the inverse clock, they also affect the time at which the reflected Brownian motion and its local time are evaluated.
	We control these effects by coupling the
	switching model and its homogenized limit through the same reflected Brownian
	motion, estimating the clock and killing errors by the averaging theorem,
	and then using the H\"older regularity of the reflected path and its local
	time to pass these estimates through the inverse time change.

The backward estimates also yield quantitative homogenization of the
corresponding forward measure evolutions. For a finite nonnegative initial
measure, we prove convergence of the quenched forward measures to the
homogenized evolution in bounded-Lipschitz distance, with every polynomial
rate below \(1/8\), both uniformly over deterministic times and in expected
uniform norm. Under the proportionality and compatibility assumptions above,
the improved \(1/4\)-scale transfers to smooth forward observables as well.
In particular, taking \(f\equiv1\) gives quantitative convergence of the total
surviving mass, or equivalently of the survival probability when the initial
measure is a probability distribution.

We conclude with numerical experiments for the switching Wentzell model
on an interval. We compare annealed survival curves with the homogenized
PDE solution and estimate the quenched marginal and uniform-in-time
errors on a finite observation grid. For the survival observable,
proportional stickiness and killing coefficients yield a sharper
estimate, which provides the relevant theoretical comparison.
A second quenched experiment uses a generic bounded Lipschitz observable
and non-proportional microscopic coefficients, providing a numerical
comparison with the general bounds
\(O(\varepsilon^\beta)\), \(0<\beta<1/8\).
The experiments examine pre-asymptotic convergence over the simulated
range of switching scales. The parameter values are motivated by a
biological application and are chosen to reflect characteristic scales
of glutamate clearance near an excitatory synapse.
	
	\bigskip
	
	The construction extends the fast-switching Robin framework of
	\cite{colantoni2026_time_dependent_Robin}. 
	Setting every \(\rho_i=0\)
	recovers the Robin case. 
	In this case, no sticky time change is present and the
	switching and killing are indexed directly by physical time. When
	\(\rho_i>0\), the chain is instead indexed by the operational time of the
	underlying reflected Brownian motion, and both the boundary coefficients and
	the killing functional are evaluated through the inverse clock. Consequently,
	comparison at a deterministic physical time is reduced to an averaging
	estimate at the random time \(\gamma_0(t)\). This time-inversion step is the
	main additional difficulty in the Wentzell setting. Unlike
	\cite{colantoni2026_time_dependent_Robin}, we also obtain quantitative error
	bounds between the fast-switching model and its homogenized limit and study
	the sharpness of the underlying averaging estimates.

	\subsection{Paper outline}
	Section~\ref{sec:abstract-averaging} contains the abstract averaging theorem
	and its basic examples. Section~\ref{sec:construction} constructs the
	Wentzell process with Markovian switching. Section~\ref{sec:homogenization}
	proves the quantitative homogenization theorem and then discusses sharpness
	and improved rates for smoother observables. 
	Section~\ref{sec:quenched} identifies the quenched and annealed backward
	evolutions. Section~\ref{sec:forward-homogenization} develops the quenched
	forward formulation and derives quantitative forward homogenization by
	duality. Section~\ref{sec:numerical} contains the numerical experiments.
	
	The appendices contain the H\"older regularity estimates, the verification
	of the boundary-local-time assumptions, the proofs of the forward
	formulations, and the numerical approximation scheme for the reflected
	Brownian motion, boundary local time, random clock, and killing functional.
	
	\subsection{Notation and conventions}
	\label{ssec:notation}
	
	Unless stated otherwise, \(D\subset\mathbb R^d\) is a bounded, connected, and open domain with \(C^2\) boundary. 
	We write \(\overline D\) for its closure, \(\sigma\) for the surface measure on \(\partial D\), and \(\mathbf n\) for the outward unit normal vector. Thus
	\[
	\partial_{\mathbf n}f
	:=
	\nabla f\cdot\mathbf n
	\]
	for the corresponding normal derivative.
	
	The symbol \(W^+=(W_t^+)_{t\ge0}\) denotes reflected Brownian motion in
	\(\overline D\), with generator
	\[
	\frac12\Delta
	\]
	in the interior of \(D\). Its boundary local time is denoted by
	\(L=(L_t)_{t\ge0}\). We use the Skorokhod-regulator normalization
	\[
	dW_t^+
	=
	dB_t-\mathbf n(W_t^+)\,dL_t,
	\]
	where \(B\) is a standard \(d\)-dimensional Brownian motion. With this normalization, the occupation-density representation reads
	\[
	L_t
	=
	\lim_{\delta\downarrow0}
	\frac{1}{2\delta}
	\int_0^t
	\mathbf 1_{\{
		\operatorname{dist}(W_s^+,\partial D)<\delta
		\}}
	\,ds.
	\]
	This convention fixes the factor \(1/2\) used throughout the paper. 
	Figure~\ref{fig:Brownian_on_disk} illustrates the geometry and shows a
	sample path of \(W^+\) in a two-dimensional disk.
	
	\begin{figure}
		[htbp]
		\centering 
		\includegraphics[width=0.4\textwidth]{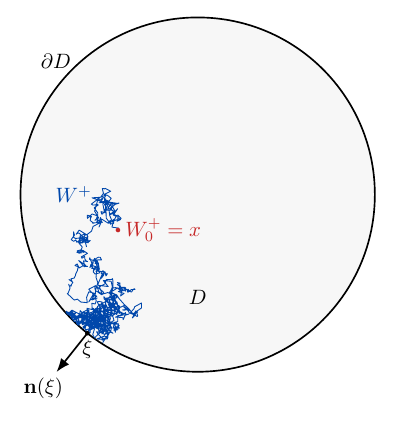}
		\caption{
			A sample path of reflected Brownian motion \(W^+\) in the unit disk,
			together with the boundary \(\partial D\) and an outward unit normal
			\(\mathbf n(\xi)\).
		}
		\label{fig:Brownian_on_disk}
	\end{figure}
	
	Throughout the paper, $S:=\{1,\ldots,N\} $
	is a finite state space and
	\(\alpha=(\alpha_t)_{t\ge0}\) is an irreducible continuous-time Markov chain
	on \(S\), with intensity matrix $Q=(q_{ij})_{i,j\in S} $.
	Unless explicitly stated otherwise, \(\alpha\) is independent of \(W^+\).
	Its unique invariant probability is denoted by $\pi=(\pi_i)_{i\in S} $.
	
	The functions
	\[
	\rho:S\to[0,\infty),
	\qquad
	\kappa:S\to[0,\infty)
	\]
	denote, respectively, the regime-dependent stickiness and elastic-killing
	coefficients. We use the abbreviations
	\[
	\rho_i:=\rho(i),
	\qquad
	\kappa_i:=\kappa(i),
	\]
	and their ergodic averages are
	\[
	\overline\rho
	:=
	\sum_{i\in S}\pi_i\rho_i,
	\qquad
	\overline\kappa
	:=
	\sum_{i\in S}\pi_i\kappa_i.
	\]
	
	We write $\Prb_x^W $, $\Esp_x^W $ for probability and expectation with respect to the reflected Brownian motion
	started from \(x\in\overline D\). When there is no ambiguity, we simply write $\Prb_x:=\Prb_x^W $, $\Esp_x:=\Esp_x^W $.
	We write $\Prb_i^\alpha $, $\Esp_i^\alpha $ 
	for probability and expectation with respect to the Markov chain started from
	\(i\in S\). Since \(\pi\) is its invariant probability, we also set
	\[
	\Prb_\pi^\alpha
	:=
	\sum_{i\in S}\pi_i\Prb_i^\alpha,
	\qquad
	\Esp_\pi^\alpha
	:=
	\sum_{i\in S}\pi_i\Esp_i^\alpha.
	\]
	Finally, since \(W^+\) and \(\alpha\) are taken independent, we set
	\[
	\Prb_{x,i}
	:=
	\Prb_x^W\otimes\Prb_i^\alpha,
	\qquad
	\Prb_{x,\pi}
	:=
	\Prb_x^W\otimes\Prb_\pi^\alpha,
	\]
	and write $\Esp_{x,i} $, $\Esp_{x,\pi} $
	for the corresponding joint expectations.
	
	For a measurable space \(F\), \(B_b(F)\) denotes the space of bounded Borel
	functions on \(F\). We write \(C(F)\) for the continuous functions when
	\(F\) is compact, and \(\|\cdot\|_\infty\) for the uniform norm.
	
	\section{Quantitative averaging for Markov-modulated additive functionals}
	\label{sec:abstract-averaging}
	
	Let \(Y=(Y_t)_{t\ge0}\) be a conservative continuous-time strong Markov process
	on a Polish space \(E\), realized on a canonical path space with shifts
	\((\theta_t)_{t\ge0}\) and its usual augmented filtration
	\((\mathcal F_t)_{t\ge0}\). We shall use the following standard notion.
	
	\begin{definition}
		A nonnegative $(\mathcal F_t)_{t\ge0}$-adapted process
		$A=(A_t)_{t\ge0}$ is called a \emph{positive continuous additive functional}
		(PCAF) of $Y$ if $A_0=0$ and, under every admissible starting law, its paths
		are almost surely continuous and nondecreasing, and satisfy
		\begin{equation}
			\label{eq:PCAF-additivity}
			A_{t+s}
			=
			A_t+A_s\circ\theta_t,
			\qquad s,t\ge0.
		\end{equation}
	\end{definition}
	
	Throughout this section, \(A\) is a strict PCAF, almost surely finite on compact intervals under every starting law considered.
	We choose a version for which \eqref{eq:PCAF-additivity}, including its
	stopping-time form, hold under every starting law that we use. If the original
	PCAF is defined only off a properly exceptional set, we work on the reduced
	state space. For reflected Brownian motion below, no such restriction is needed.
	
	Let $\alpha$ be an independent irreducible finite-state continuous-time Markov
	chain, with generator $Q$ and invariant law $\pi$. Under
	$\Prb_{y,i}=\Prb_y^Y\otimes\Prb_i^\alpha$, the two initial states are
	$Y_0=y$ and $\alpha_0=i$. We use the same coordinate processes for all
	$\varepsilon$. For bounded $h:S\to\mathbb R$, put
	$\widetilde h=h-\pi(h)$ and define
	\begin{equation}\label{eq:abstract-J}
		J_\varepsilon(t)=\int_0^t\widetilde h(\alpha_{s/\varepsilon})\,dA_s.
	\end{equation}
	These pathwise Stieltjes integrals are measurable and continuous in \(t\);
	indeed,
	\[
	|J_\varepsilon(t)-J_\varepsilon(s)|
	\le
	\|\widetilde h\|_\infty |A_t-A_s|.
	\]
	
	The only quantitative assumption is the following short-time bound: for some
	$\vartheta\in(0,1]$ and each $T>0$,
	\begin{equation}\label{eq:abstract-A-first}
		M(t):=\sup_{y\in E}\Esp_y A_t\le C_Tt^\vartheta,
		\qquad 0<t\le T.
	\end{equation}
	
	\begin{lemma}
		\label{lem:AF-conditioning}
		Let \(F\) be a nonnegative measurable functional of the future base path such that
		\[
		K:=\sup_{z\in E}\Esp_z F<\infty.
		\]
		For deterministic $0\le a\le b$,
		\begin{equation}\label{eq:AF-conditioning}
			\Esp_y\int_a^b F\circ\theta_s\,dA_s
			\le K\Esp_y(A_b-A_a).
		\end{equation}
		Moreover, for every integer $p\ge1$,
		\begin{equation}\label{eq:abstract-A-moment}
			\sup_y\Esp_y A_t^p\le p!M(t)^p\le p!C_T^pt^{p\vartheta}.
		\end{equation}
		For every finite real $q\ge1$ this implies
		\begin{equation}\label{eq:abstract-A-increments}
			\sup_y\Esp_y(A_v-A_u)^q\le C_{T,q}(v-u)^{q\vartheta},
			\qquad 0\le u<v\le T.
		\end{equation}
	\end{lemma}
	
	\begin{proof}
		Set $\tau_\ell=\inf\{s:A_s>\ell\}$, with $\inf\emptyset=\infty$.
		Then $\tau_\ell$ is a stopping time, and the Stieltjes change of variables gives  (see \cite[Proposition~0.4.10]{revuz1999Continuous_martingales})
		\[
		\int_a^b F\circ\theta_s\,dA_s
		=\int_0^\infty\mathbf1_{\{a\le\tau_\ell<b\}}
		F\circ\theta_{\tau_\ell}\,d\ell.
		\]
		Endpoint and plateau levels have zero Lebesgue measure. Conditioning at
		$\tau_\ell$ and using the strong Markov property gives the bound
		$K\Esp_y\int_0^\infty\mathbf1_{\{a\le\tau_\ell<b\}}d\ell$, which is
		\eqref{eq:AF-conditioning}. 
		The interchange of expectation and integration is justified by localization in time and level and then applying monotone convergence.
		
		For integer $p\ge2$, continuity and monotonicity give
		\[
		A_t^p=p\int_0^t(A_t-A_s)^{p-1}\,dA_s.
		\]
		Since $A_t-A_s=A_{t-s}\circ\theta_s\le A_t\circ\theta_s$, induction and \eqref{eq:AF-conditioning} give
		\[
		\Esp_y A_t^p\le p(p-1)!M(t)^{p-1}\Esp_y A_t\le p!M(t)^p.
		\]
		All integrals are nonnegative, so Tonelli and truncation justify the induction
		without any prior moment assumption. Applying the same argument after time
		$u$ gives the increment estimate for integer powers; Lyapunov's inequality
		with $p=\lceil q\rceil$ gives the real exponents. 
		This is the standard ordered-integral (Kac moment) argument for positive continuous additive functionals; see \cite[Section~3.1]{fitzsimmons1999kac}.
		The same Kac moment machinery also yields the two-point local-time moment identity used in Example~\ref{ex:abstract-benchmarks}.
	\end{proof}
	
	We now present our main abstract quantitative averaging theorem whose rate is governed by the short-time growth exponent \(\vartheta\).
	
	\begin{theorem}[Quantitative averaging]
		\label{thm:abstract-averaging}
		Assume \eqref{eq:abstract-A-first}. The following estimates hold uniformly in the initial states $y\in E$ and $i\in S$, for $0<\varepsilon\le1$.
		\begin{enuroman}
			\item\label{item:abstract-random-time}
			For every $[0,T]$-valued random time $\tau$ measurable with respect to the base path,
			\begin{equation}\label{eq:abstract-random-time}
				\Esp_{y,i}|J_\varepsilon(\tau)|^2\le C_{T,h}\varepsilon^\vartheta.
			\end{equation}
			No stopping-time property is required. Once the base path is fixed,
			both $dA$ and the upper limit $\tau$ are deterministic from the point of
			view of the chain.
			\item\label{item:abstract-max}
			The uniform estimate is
			\begin{equation}\label{eq:abstract-max}
				\Esp_{y,i}\sup_{0\le t\le T}|J_\varepsilon(t)|^2
				\le C_{T,h}\varepsilon^\vartheta(1+|\log\varepsilon|)^2.
			\end{equation}
			In particular, $J_\varepsilon\to0$ in $L^2$ for the uniform norm on $C([0,T])$ under every starting law, uniformly in the starting states.
		\end{enuroman}
	\end{theorem}
	
	\begin{proof}
		\emph{Step 1: mixing and discounted potentials.} 
		Since \(S\) is finite and the chain is irreducible, the transition
		semigroup \(P_t^\alpha=e^{tQ}\) is exponentially ergodic: there exist
		constants \(C,c>0\) such that
		\[
		\sup_{i,j\in S}
		\left|P_{ij}^\alpha(t)-\pi_j\right|
		\le Ce^{-ct},
		\qquad t\ge0;
		\]
		see, e.g., \cite[Section~6.6]{anderson1991_book}. For bounded \(h\),
		with
		\[
		\widetilde h:=h-\pi(h),
		\qquad
		\pi(\widetilde h)=0,
		\]
		we have
		\[
		\|P_t^\alpha\widetilde h\|_\infty
		\le C_h e^{-ct},
		\qquad t\ge0,
		\]
		where \(C_h:=|S|C\|\widetilde h\|_\infty\). Therefore, for \(r\le s\),
		the Markov property gives
		\begin{align}\label{eq:abstract-covariance}
			\left|\Esp_i^\alpha[\widetilde h(\alpha_{r/\varepsilon})
			\widetilde h(\alpha_{s/\varepsilon})]\right|
			&=\left|\Esp_i^\alpha[\widetilde h(\alpha_{r/\varepsilon})
			(P^\alpha_{(s-r)/\varepsilon}\widetilde h)(\alpha_{r/\varepsilon})]\right|\notag\\
			&\le C_he^{-c(s-r)/\varepsilon}.
		\end{align}
		No stationarity is used here. Integrating by parts against $\Esp_yA_s$ gives
		\begin{align}\label{eq:abstract-weighted-potential}
			U_\varepsilon(T):=\sup_y\Esp_y\int_0^T e^{-cs/\varepsilon}\,dA_s
			&\le e^{-cT/\varepsilon}M(T)
			+\frac c\varepsilon\int_0^T e^{-cs/\varepsilon}M(s)\,ds\notag\\
			&\le C_T\varepsilon^\vartheta.
		\end{align}
		
		\medskip\noindent\emph{Step 2: interval and random-time estimates.}
		Lemma~\ref{lem:AF-conditioning} gives the integrability needed for Fubini.
		Condition on the base path. Since \(\tau\) is measurable with respect to
		that path, both \(A\) and \(\tau\) are then fixed. Since \(A\) is continuous,
		its Stieltjes measure \(dA\) is atomless, so the diagonal has zero
		\(dA\otimes dA\)-measure.
		Therefore,
		\[
		\Esp_{y,i}|J_\varepsilon(\tau)|^2
		\le 2C_h\Esp_y\int_0^T\int_s^T
		e^{-c(r-s)/\varepsilon}\,dA_r\,dA_s
		\le 2C_h U_\varepsilon(T)\Esp_y A_T,
		\]
		where \eqref{eq:AF-conditioning} is applied to
		$F=\int_0^T e^{-cu/\varepsilon}\,dA_u$ after enlarging the remaining
		future horizon to $T$. 
		Together with
		\eqref{eq:abstract-weighted-potential} and
		our assumption \eqref{eq:abstract-A-first}, this proves
		\eqref{eq:abstract-random-time}.
		
		The same calculation on a deterministic interval \([a,b]\) yields
		\begin{equation}\label{eq:abstract-interval}
			\Esp_{y,i}|J_\varepsilon(b)-J_\varepsilon(a)|^2
			\le
			C_{T,h}\varepsilon^\vartheta
			\Esp_y(A_b-A_a),
			\qquad 0\le a\le b\le T.
		\end{equation}
		The dependence on the interval through the increment \(A_b-A_a\) will be
		important below. Indeed, if
		\[
		0=t_0<t_1<\cdots<t_N=T
		\]
		is any deterministic partition of \([0,T]\), then
		\[
		\sum_{k=0}^{N-1}
		\Esp_y\!\left(A_{t_{k+1}}-A_{t_k}\right)
		=
		\Esp_y A_T.
		\]
		Thus, summing \eqref{eq:abstract-interval} over all intervals at one dyadic
		level introduces no factor proportional to the number of intervals.
		
		\medskip\noindent\emph{Step 3: dyadic maximal bound.}
		Let \(m\in\mathbb N\), and let
		\[
		G_m=\{kT2^{-m}:0\le k\le2^m\}
		\]
		be the dyadic grid of mesh size \(T2^{-m}\) on \([0,T]\), and for
		\(0\le j\le m\) and \(0\le k<2^j\), define
		\[
		D_{j,k}
		:=
		J_\varepsilon((k+1)T2^{-j})
		-
		J_\varepsilon(kT2^{-j}),
		\]
		the increment of \(J_\varepsilon\) over the \(k\)-th dyadic interval at
		level \(j\).
		Every interval \([0,t]\), with \(t\in G_m\), admits a dyadic decomposition
		into disjoint intervals, containing at most one interval from each level
		\(0\le j\le m\). Hence the corresponding telescoping decomposition of
		\(J_\varepsilon(t)\) yields
		\[
		\max_{t\in G_m}|J_\varepsilon(t)|
		\le\sum_{j=0}^m\max_{k<2^j}|D_{j,k}|.
		\]
		At each level the intervals partition $[0,T]$, and \eqref{eq:abstract-interval} implies
		\[
		\Esp_{y,i}\max_{k<2^j}|D_{j,k}|^2
		\le\sum_{k<2^j}\Esp_{y,i}|D_{j,k}|^2
		\le C_{T,h}\varepsilon^\vartheta\Esp_y A_T.
		\]
		Therefore, by \eqref{eq:abstract-A-first} and Minkowski's inequality,
		\begin{equation}\label{eq:abstract-grid-points}
			\left(\Esp_{y,i}\max_{t\in G_m}|J_\varepsilon(t)|^2\right)^{1/2}
			\le C_{T,h}(m+1)\varepsilon^{\vartheta/2}.
		\end{equation}
		
		Choose an integer $p>\max\{2,1/\vartheta\}$ and set $\delta:=T2^{-m}$. Between grid points,
		\[
		|J_\varepsilon(t)-J_\varepsilon(k\delta)|
		\le\|\widetilde h\|_\infty(A_{(k+1)\delta}-A_{k\delta}).
		\]
		Using \eqref{eq:abstract-A-increments} and concavity of $x^{2/p}$,
		\[
		\Esp_y\max_{k<2^m}(A_{(k+1)\delta}-A_{k\delta})^2
		\le\left(\sum_{k<2^m}\Esp_y(A_{(k+1)\delta}-A_{k\delta})^p\right)^{2/p}
		\le C_{T,p}\delta^{2\vartheta-2/p}.
		\]
		Consequently,
		\begin{equation}\label{eq:abstract-balance}
			\Esp_{y,i}\sup_{t\le T}|J_\varepsilon(t)|^2
			\le C_{T,h,p}\big((m+1)^2\varepsilon^\vartheta
			+\delta^{2\vartheta-2/p}\big).
		\end{equation}
		For small $\varepsilon$, choose the dyadic mesh so that
		$\delta\asymp\varepsilon^{\vartheta/(2\vartheta-2/p)}$. Then
		$m= \log_2(T/\delta) =O(1+|\log\varepsilon|)$ and the interpolation term is
		$O(\varepsilon^\vartheta)$. The remaining values of $\varepsilon\in(0,1]$
		are absorbed into the constant using
		$|J_\varepsilon|\le\|\widetilde h\|_\infty A_T$.
		This proves the uniform-in-time bound, and completes the proof. 
	\end{proof}
	
	\begin{corollary}\label{cor:abstract-all-moments}
		Under \eqref{eq:abstract-A-first}, for every $0<\zeta<\vartheta$, 
		\begin{equation}\label{eq:abstract-all-moments-max}
			\sup_{y,i}\Esp_{y,i}\sup_{t\le T}|J_\varepsilon(t)|^2
			\le C_{T,h,\zeta}\varepsilon^\zeta.
		\end{equation}
		Here, the constant depends on $\zeta $ and the limit exponent $\vartheta $ is not attained.
	\end{corollary}
	
	\begin{proof}
		Since $\varepsilon^{\vartheta-\zeta}(1+|\log\varepsilon|)^2$ is bounded on
		$(0,1]$, the claim follows. The constant depends on $\zeta$; this argument
		does not give the limit case $\zeta=\vartheta$.
	\end{proof}
	
	\begin{remark}
The integrand may also depend on the base process, provided
the switching chain remains independent of it.
Indeed, let \(g:E\times S\to\mathbb R\) be bounded and measurable, and set
\[
\overline g(y):=\sum_{i\in S}\pi_i g(y,i).
\]
Then the same estimates apply to
\[
J_\varepsilon^g(t)
=
\int_0^t
\bigl[g(Y_s,\alpha_{s/\varepsilon})-\overline g(Y_s)\bigr]\,dA_s.
\]
Indeed, conditional on the base path, the functions
\[
h_s(i):=g(Y_s,i)-\overline g(Y_s)
\]
are deterministic, centered with respect to \(\pi\), and uniformly bounded by
\(2\|g\|_\infty\). Exponential ergodicity of the chain therefore gives, uniformly
in \(r\le s\),
\[
\left|
\Esp_i^\alpha
\left[
h_r(\alpha_{r/\varepsilon})
h_s(\alpha_{s/\varepsilon})
\right]
\right|
\le
C_g e^{-c(s-r)/\varepsilon}.
\]
The proof of Theorem~\ref{thm:abstract-averaging} then applies without change.
Consequently, \eqref{eq:abstract-random-time} and
\eqref{eq:abstract-max} remain valid, with constants depending on
\(\|g\|_\infty\).
	\end{remark}

	\begin{example}[Ordinary occupation and interior local time]
		\label{ex:abstract-benchmarks}
		For bounded nonnegative $V$, $A_t=\int_0^t V(Y_s)\,ds$ satisfies \eqref{eq:abstract-A-first} with $\vartheta=1$. A singular example, without a reflecting boundary, is the symmetric local time $\ell_t^0$ at zero of  standard one-dimensional Brownian motion $B$. 
		With $p_s(x,y)$ its heat kernel,
		\[
		\sup_x\Esp_x\ell_t^0
		=\sup_x\int_0^t p_s(x,0)\,ds
		=\sqrt{2t/\pi},
		\]
		so the bound~\eqref{eq:abstract-A-first} holds for $\vartheta=1/2$. 
		The identity
		\[
		\Esp_x\ell_t^0=\int_0^t p_s(x,0)\,ds
		\]
		follows from the occupation-density formula
		\cite[Corollary~VI.1.9]
		{revuz1999Continuous_martingales}.
		
		The exponent $\vartheta=1/2$ is attained in the following simple benchmark.
		Let \(B_0=0\) and let \(H\) be an independent sign chain on
		\(S=\{-1,1\}\) with generator and invariant law
		\[
		Q=
		\begin{pmatrix}
			-a & a\\
			a & -a
		\end{pmatrix},
		\qquad
		\pi=
		\left(\frac12,\frac12\right).
		\]
		
		Assume that \(H\) is started from \(\pi\). Then \(H\) is stationary and centered in the sense that \(\Esp(H_t)=0\) for all \(t\ge 0\), and its covariance formula reads
		\begin{equation}
			\label{eq:covariance_formula}
			\Esp[H_sH_r]=e^{-2a|s-r|}.
		\end{equation}
		
		Conditioning on the Brownian path and using the independence of \(H\) and
		\(B\), together with~\eqref{eq:covariance_formula}, gives
		\begin{align}
			\Esp\left|\int_0^T H_{s/\varepsilon}\,d\ell_s^0\right|^2
			&= \Esp\left[\int_0^T \int_0^T H_{r/\varepsilon}\,H_{s/\varepsilon}\,d\ell_r^0\,d\ell_s^0\right]
			\\
			&= \Esp\left[
			\int_0^T\int_0^T
			e^{-2a|s-r|/\varepsilon}\,
			d\ell_r^0\,d\ell_s^0
			\right].
		\end{align}
		For every nonnegative Borel function \(g\), applying the space--time
		occupation-times formula
		\cite[Exercise~VI.1.15]{revuz1999Continuous_martingales}
		to \((s,a)\mapsto g(s)\eta_\delta(a)\), where
		\[
		\eta_\delta(a):=\frac{1}{2\delta}\mathbf 1_{\{|a|<\delta\}},
		\]
		then taking expectations and letting \(\delta\downarrow0\), gives
		\[
		\Esp_x\left[\int_0^t g(s)\,d\ell_s^0\right]
		=
		\int_0^t g(s)p_s(x,0)\,ds.
		\]
		Using this identity together with Tonelli's theorem, the additivity of
		local time, and the Markov property, for every nonnegative Borel function
		\(\varphi\) on \(\{(r,s):0<r<s<T\}\),
		\[
		\begin{aligned}
			&\Esp_0\left[
			\int_0^T\int_0^s
			\varphi(r,s)\,d\ell_r^0\,d\ell_s^0
			\right]
			\\
			&\qquad=
			\Esp_0\left[
			\int_0^T
			\left(
			\int_r^T \varphi(r,s)\,d\ell_s^0
			\right)
			d\ell_r^0
			\right]
			\\
			&\qquad=
			\int_0^T
			p_r(0,0)
			\left(
			\int_0^{T-r}
			\varphi(r,r+u)\,
			p_u(0,0)\,du
			\right)dr
			\\
			&\qquad=
			\int_0^T\int_0^s
			\varphi(r,s)\,
			p_r(0,0)p_{s-r}(0,0)\,dr\,ds.
		\end{aligned}
		\]
		
		Since \(\ell^0\) is continuous, its Stieltjes measure \(d\ell^0\) is
		atomless, and therefore \(d\ell^0\otimes d\ell^0\) gives no mass to the
		diagonal. Taking
		\[
		\varphi(r,s)=e^{-2a(s-r)/\varepsilon}
		\]
		and using symmetry, we obtain
		\begin{align}
			\Esp\left|\int_0^T H_{s/\varepsilon}\,d\ell_s^0\right|^2
			&=
			2\int_{0<r<s<T}
			e^{-2a(s-r)/\varepsilon}
			p_r(0,0)p_{s-r}(0,0)\,dr\,ds
			\\
			&=
			2\int_0^T e^{-2au/\varepsilon}p_u(0,0)
			\int_0^{T-u}p_s(0,0)\,ds\,du
			\\
			&=
			\frac2\pi\int_0^T e^{-2au/\varepsilon}
			\sqrt{\frac{T-u}{u}}\,du
			\sim
			\sqrt{\frac{2T}{\pi a}}\,\varepsilon^{1/2}.
			\label{eq:interior-localtime-variance}
		\end{align}
		For the asymptotic, substitute $u=\varepsilon v$ and apply dominated
		convergence with the integrable bound
		\(\sqrt T\,e^{-2av}v^{-1/2}\). 
		
		By contrast, using the covariance formula~\eqref{eq:covariance_formula} with \(\lambda=2a/\varepsilon\),
		we get 
		\begin{align}
			\Esp\left|\int_0^T H_{s/\varepsilon}\,ds\right|^2
			&=
			\int_0^T\int_0^T e^{-\lambda|s-r|}\,dr\,ds \\
			&=
			2\int_0^T \frac{1-e^{-\lambda s}}{\lambda}\,ds \\
			&=
			2\left(
			\frac{T}{\lambda}
			-
			\frac{1-e^{-\lambda T}}{\lambda^2}
			\right)
			\sim \frac{T}{a}\varepsilon.
			\label{eq:interior-occupationtime-variance}
		\end{align}
		The asymptotics~\eqref{eq:interior-localtime-variance} and
		\eqref{eq:interior-occupationtime-variance} show that ordinary occupation
		and local-time occupation fluctuate on different scales. In particular,
		within this class the exponent \(1/2\) in the terminal-time second-moment
		estimate for local-time occupation cannot be improved. 
	\end{example}
	
	\begin{remark}\label{rem:abstract-boundary-local-time}
		Unlike the one-dimensional interior-local-time benchmark of
		Example~\ref{ex:abstract-benchmarks}, we consider here the boundary
		local time of reflected Brownian motion in a \(C^2\) domain \(D\).
		Appendix~\ref{app:averaging} shows that \(\vartheta=1/2\) in this case.
		Consequently, the random-time squared \(L^2\) error is
		\(O(\varepsilon^{1/2})\), while the maximal error is
		\(O(\varepsilon^{1/2}(1+|\log\varepsilon|)^2)\), and hence
		\(O(\varepsilon^\zeta)\) for every \(0<\zeta<1/2\).
		
		By Lemma~\ref{lem:AF-conditioning}, the first-moment estimate for the
		boundary local time also yields, for every finite \(q\ge1\),
		\[
		\sup_{x\in\overline D}\Esp_x[L_t^q]
		\le C_{T,q}t^{q/2},
		\qquad 0\le t\le T,
		\]
		and
		\[
		\sup_{x\in\overline D}
		\Esp_x[(L_v-L_u)^q]
		\le C_{T,q}(v-u)^{q/2},
		\qquad 0\le u<v\le T.
		\]
		Thus no additional moment assumption on the boundary local time is
		required beyond the short-time first-moment estimate.
	\end{remark}
	
	\section{Wentzell boundary dynamics with Markovian switching}
	\label{sec:construction}
	
	\subsection{The Brownian motion with sticky reflection and elastic killing} 
	
The Wentzell process without switching is classical
\cite{wentzell1960_on_boundary_conditions,taira,FGGR}.
We recall only the representation in which reflected Brownian motion is
first slowed down by its boundary local time and then killed through the
same local time. This representation is taken from
\cite{dovidio2026_fractional_boundary}.

In the boundary relations below, the Laplacian is understood
distributionally. If \(f\in H^1(D)\) and \(\Delta f\in L^2(D)\), the
outward normal derivative is understood weakly through
\[
\langle \partial_{\mathbf n}f,\operatorname{Tr}v\rangle
=
\int_D \nabla f\cdot\nabla v\,dx
+
\int_D (\Delta f)v\,dx,
\qquad v\in H^1(D).
\]
This defines an element of \(H^{-1/2}(\partial D)\). 
For \(f\in C^2(\overline D)\), it agrees with the classical normal derivative.

For analytic realizations of the Neumann and Wentzell--Robin Laplacians
on spaces of continuous functions, see
\cite[Theorem~3.3]{arendt2003laplacian},
\cite[Theorem~1.1]{engel2003wentzell}, and
\cite{warma2006_robin_wentzell}; see also
\cite{arendt2012sectorial,favini2008wentzell}.
For \(\rho>0\), the boundary relation
\[
\frac{\rho}{2}\Delta f
+\partial_{\mathbf n}f
+\kappa f
=0
\]
is equivalently
\[
\Delta f
+\frac{2}{\rho}\partial_{\mathbf n}f
+\frac{2\kappa}{\rho}f
=0,
\]
which is of the generalized Wentzell form treated in
\cite[Theorem~1.1]{engel2003wentzell}.
For the probabilistic interpretation in terms of sticky-reflected
diffusions, see, e.g.,
\cite{grothaus2017stochastic,bormann2026functional}.

We shall not need an exact characterization of the full generator domains.
Instead, the boundary relations displayed below specify the smooth test
classes on which the corresponding generators act as
\(\frac12\Delta\), with the normalization induced by our choice of
reflected Brownian motion and boundary local time. The processes themselves
will be used through the time-change and Feynman--Kac representations below.

\begin{definition}
	\label{def:SR_EK}
	We use the following terminology and associated smooth boundary classes.
	\begin{enuroman}
		\item The instantaneously reflected Brownian motion corresponds to
		\[
		\mathcal L^0 f=\frac12\Delta f,
		\qquad
		\partial_{\mathbf n}f=0
		\quad\text{on }\partial D.
		\]
		
		\item The homogeneous sticky reflected Brownian motion with parameter
		\(\rho\) corresponds to
		\[
		\mathcal L^{\mathrm{sticky}}f=\frac12\Delta f,
		\qquad
		\frac{\rho}{2}\Delta f+\partial_{\mathbf n}f=0
		\quad\text{on }\partial D.
		\]
		
		\item The homogeneous sticky reflected Brownian motion with elastic
		killing parameters \((\rho,\kappa)\) corresponds to
		\[
		\mathcal L^{\mathrm{el-sticky}}f=\frac12\Delta f,
		\qquad
		\frac{\rho}{2}\Delta f
		+\partial_{\mathbf n}f
		+\kappa f=0
		\quad\text{on }\partial D.
		\]
	\end{enuroman}
\end{definition}
	
	The parameter $\rho$ controls the delay created at the boundary, while
	$\kappa$ controls killing there. Multidimensional sticky reflected Brownian
	motion can be constructed by Dirichlet-form methods
	\cite{grothaus2017stochastic}; for the one-dimensional case see
	\cite{bass2014A_stochastic,engelbert2014Stochastic_differential}.
	
	Let $W^+$, $W^{(\rho)}$, and $W^{(\rho,\kappa)}$ denote the three
	processes in Definition~\ref{def:SR_EK}; the superscripts indicate the
	dependence on the boundary parameters.
	Write $\loct{W^+}{}{t}$ and $\loct{W^{(\rho)}}{}{t}$ for the corresponding
	boundary local times, and let $\gamma$ be the right inverse of
	\[
	A(t)=t+\rho\loct{W^+}{}{t},\quad t\ge 0.
	\]
	
	\begin{proposition}
		[\cite{dovidio2026_fractional_boundary}, Theorem~4.6]
		\label{prop:time_change0}
		Let $f\in C(\overline D)$ and define $u$ by
		\begin{equation}
			u(t,x) := \Esp_x \sqbraces{f(W^{(\rho,\kappa)}_t)}, \quad t \ge 0, \; x\in D.
		\end{equation}
		Then $u$ admits the representation
		\begin{equation}
			\label{eq:semigroup}
			u(t,x) = \Esp_{x} \sqbraces{f(W^{+}_{\gamma(t)}) e^{-\kappa \loct{W^+}{}{\gamma(t)}}}, \quad t \ge 0, \; x\in D.
		\end{equation}
	\end{proposition}
	
	This is a special case of the general time-change theory for symmetric Markov
	processes \cite{chen-fukushima}.
	
	\begin{remark}
		If $\kappa=0$, \eqref{eq:semigroup} reduces to
		\begin{equation}
			\Esp_x \sqbraces{f(W^{(\rho)}_t)} = \Esp_{x} \sqbraces{f(W^{+}_{\gamma(t)}) }, 
		\end{equation}
		so $W^{(\rho)}$ and $W^{+}_{\gamma(\cdot)}$ have the same law.
		More generally, for $\kappa\in[0,\infty)$, the semigroups of $W^{(\rho,\kappa)}$, $W^{(\rho)}$, and $W^+$ are linked by the identity
		\begin{equation}
			\Esp_x \sqbraces{f(W^{(\rho,\kappa)}_t)} = \Esp_{x} \sqbraces{f(W^{(\rho)}_{t}) e^{-\kappa \loct{W^{(\rho)}}{}{t}}}
			= \Esp_{x} \sqbraces{f(W^{+}_{\gamma(t)}) e^{-\kappa \loct{W^+}{}{\gamma(t)}}} , \quad t \ge 0, \; x\in D.
		\end{equation}
	\end{remark}
	
	\subsection{Constructing the switching process}
	\label{subsec:operational-time-construction}
	
	We now let the boundary regime switch. The key convention is where the
	chain is run: it is indexed by the operational time of the reflected
	Brownian motion, not the physical time.
	This same clock is used in the homogenization argument.
	
	Set the augmented state space
	\[
	E:=\overline D\times S.
	\]
	On $E$ consider the product process
	\[
	Z_t:=(W_t^+,\alpha_t),\qquad t\ge0.
	\]
	Define the clock: the continuous increasing additive functional
	\begin{equation}
		\label{eq:operational-clock}
		A(t)
		:=
		t+\int_0^t \rho(\alpha_s)\,dL_s,
		\qquad t\ge0,
	\end{equation}
	and its true right-inverse $t\mapsto \gamma(t) $, characterized by
	\begin{equation}
		\label{eq:operational-clock-inverse}
		\gamma(t)
		:=
		\inf\{s\ge0:A(s)>t\},
		\qquad t\ge0.
	\end{equation}
	Because $A(t)-A(s)\ge t-s$, the clock is strictly increasing and therefore its right-inverse $\gamma$ is a true inverse.
	Also, $\gamma $ is $1$-Lipschitz: $\gamma(t) - \gamma(s) \le t-s $.
	
	The time-changed regime and spatial components of $Z $ are
	\begin{equation}
		\label{eq:joint-time-changed-process}
		I_t:=\alpha_{\gamma(t)},
		\qquad
		X_t:=W^+_{\gamma(t)},
	\end{equation}
	and the elastic killing functional is
	\begin{equation}
		\label{eq:joint-killing-functional}
		K(t)
		:=
		\int_0^{\gamma(t)}\kappa(\alpha_s)\,dL_s,
		\qquad
		M_t:=e^{-K(t)}.
	\end{equation}
	To realize the killing explicitly, we enlarge the probability space by an
	independent $\mathcal E\sim\operatorname{Exp}(1)$ random variable, and set
	\[
	\zeta:=\inf\{t\ge0:K(t)>\mathcal E\}.
	\]
	The time $\zeta $ is the killing time: at $\zeta$ the process is sent to the cemetery point $\dagger$.
	
	\begin{proposition}
		\label{prop:operational-time-process}
		The killed process
		\[
		\mathcal Y_t
		:=
		\begin{cases}
			(X_t, I_t),&t<\zeta,\\
			\dagger,&t\ge\zeta,
		\end{cases}
		\]
		is a c\`adl\`ag strong Markov process on
		\[
		E_\dagger
		:=
		(\overline D \times S)\cup\{\dagger\}.
		\]
		For every \(\Phi\in B_b(E)\), its family of sub-Markov transition operators is characterized by
		\begin{equation}
			\label{eq:construction-semigroup}
			(\mathcal S_t\Phi)(x,i)
			=
			\Esp_{x,i}\!\left[
			\Phi(X_t,I_t)\mathbf 1_{\{t<\zeta\}}
			\right]
			=
			\Esp_{x,i}\!\left[
			\Phi\!\left(W^+_{\gamma(t)},\alpha_{\gamma(t)}\right)
			e^{-K(t)}
			\right], 
		\end{equation}
		$t\ge0 $, $(i,x)\in E $.
	\end{proposition}
	
	\begin{proof}
		The independent processes \(W^+\) and \(\alpha\) are strong Markov, hence
		their product
		\[
		Z_t=(W_t^+,\alpha_t)
		\]
		is a c\`adl\`ag strong Markov process on
		\(E=\overline D\times S\).
		Its path regularity follows from the continuity of \(W^+\) and the
		c\`adl\`ag paths of \(\alpha\).
		
		\emph{Step 1: Applying the time-change.} The clock
		\[
		A(t)=t+\int_0^t\rho(\alpha_s)\,dL_s
		\]
		is a continuous additive functional of \(Z\), by the additivity of the
		boundary local time. Moreover, for \(0\le s\le t\),
		\[
		A(t)-A(s)
		=
		t-s+\int_s^t\rho(\alpha_u)\,dL_u
		\ge t-s.
		\]
		Hence \(A\) is strictly increasing and tends to infinity, and its inverse
		\(\gamma\) is finite, continuous, nondecreasing, and \(1\)-Lipschitz:
		\[
		A(\gamma(t))=t,
		\qquad
		|\gamma(t)-\gamma(s)|\le |t-s|.
		\]
		
		The inverse clock satisfies the usual composition rule
		\[
		\gamma(t+s)
		=
		\gamma(t)+\gamma(s)\circ\theta_{\gamma(t)},
		\]
		as well as the corresponding stopping-time version; see
		\cite[Chapter~V, Proposition~2.3]{BlumenthalGetor68}.
		
		By the standard time-change theorem for continuous additive functionals,
		the process
		\[
		\widehat Z_t:=Z_{\gamma(t)}
		=
		\bigl(W^+_{\gamma(t)},\alpha_{\gamma(t)}\bigr)
		=
		(X_t,I_t)
		\]
		is strong Markov with respect to the time-changed filtration
		\[
		\widehat{\mathcal F}_t=\mathcal F_{\gamma(t)};
		\]
		see \cite[Chapter~V, Proposition~2.3 and Exercise~2.11]{BlumenthalGetor68} or \cite[Theorem~10.10]{dynkin1965Markov_processes}.
		
		\emph{Step 2: Applying the killing functional.} Since \(\gamma\) is continuous and nondecreasing, \(X_t=W^+_{\gamma(t)}\)
		is continuous and \(I_t=\alpha_{\gamma(t)}\) is c\`adl\`ag. Hence
		\(\widehat Z=(X,I)\) is c\`adl\`ag.
		
		Define the operational-time killing functional
		\[
		B(t):=\int_0^t\kappa(\alpha_s)\,dL_s.
		\]
		Then \(B\) is a continuous additive functional of \(Z\), and
		\[
		K(t)=B(\gamma(t)).
		\]
		By the inverse-clock composition rule,
		\[
		K(t+s)=K(t)+K(s)\circ\widehat\theta_t,
		\]
		with the corresponding stopping-time version. Consequently,
		\[
		M_t:=e^{-K(t)}
		\]
		is a continuous strong multiplicative functional of \(\widehat Z\); see
		\cite[Chapter~III, Definition~3.11]{BlumenthalGetor68}.
		
		The standard multiplicative-functional construction therefore yields the
		killed subprocess
		\[
		\mathcal Y_t=
		\begin{cases}
			(X_t,I_t),&t<\zeta,\\
			\dagger,&t\ge\zeta,
		\end{cases}
		\]
		which is a c\`adl\`ag strong Markov process; see
		\cite[Chapter~III, Theorem~3.3 and Proposition~3.12]{BlumenthalGetor68}.

		\emph{Step 3: Identification of transition law.} 
		Since \(K(t)\) is measurable with respect to the path of \(Z\), independence
		of \(\mathcal E\sim\operatorname{Exp}(1)\) gives
		\[
		\Prb(t<\zeta\mid \sigma(Z_s:s\ge0))
		=
		\Prb(\mathcal E\ge K(t)\mid \sigma(Z_s:s\ge0))
		=
		e^{-K(t)}.
		\]
		Hence, for every bounded Borel function \(\Phi\),
		\begin{align}
			\Esp_{x,i}\!\left[
			\Phi(X_t,I_t)\mathbf 1_{\{t<\zeta\}}
			\right]
			&=
			\Esp_{x,i}\!\left[
			\Phi(X_t,I_t)e^{-K(t)}
			\right]
			\\
			&=
			\Esp_{x,i}\!\left[
			\Phi\!\left(
			W^+_{\gamma(t)},\alpha_{\gamma(t)}
			\right)
			e^{-K(t)}
			\right].
		\end{align}
		This is \eqref{eq:construction-semigroup}.
	\end{proof}
	
	\begin{remark}
		The construction also admits a segment-by-segment interpretation.
		While the regime is \(i\), the reflected motion evolves with clock
		\(s+\rho_iL_s\), until either the Markov chain leaves state \(i\) or
		elastic killing occurs through the local-time hazard
		\(\kappa_i\,dL_s\). At a switching event the spatial position is
		unchanged and the evolution continues in the new regime. 
		This is consistent with the classical piecing-out viewpoint of
		Ikeda--Nagasawa--Watanabe
		\cite{ikeda1996_construction_piecingout}.
	\end{remark}

	\section{Quantitative Wentzell homogenization}
	\label{sec:homogenization}
	
	We now turn to quantitative homogenization of Wentzell boundary conditions in the fast-switching regime.
	To achieve this, we specialize the abstract theorem by taking the additive functional
	\(A\) to be the boundary local time \(L\). 
	We proceed in two stages.
	First, we homogenize the clock, which yields the purely sticky case
	of~\eqref{eq:intro-switching-wentzell}, with \(\kappa_i=0\) for all
	\(i\in S\). 
	We then homogenize the killing functional and transfer the resulting
	estimates to the quenched Feynman--Kac observable, obtaining the general
	switching Wentzell case.
	
Throughout this section, \(x\in\overline D\) denotes the starting point of the
reflected Brownian motion, while the environment starts in equilibrium:
\[
\alpha_0\sim\pi.
\]
We write
\[
\Prb^\alpha:=\Prb_\pi^\alpha,
\qquad
\Esp^\alpha:=\Esp_\pi^\alpha,
\]
and, when needed,
\[
\Prb_{x,\pi}:=\Prb_x^W\otimes\Prb_\pi^\alpha,
\qquad
\Esp_{x,\pi}
\]
for the corresponding joint law and expectation.
	
	Under the standing assumptions of Section~\ref{ssec:notation}, fix
	\(T>0\). For \(\varepsilon\in(0,1]\), set
	\[
	\alpha_s^\varepsilon:=\alpha_{s/\varepsilon}.
	\]
	Recall that
	\begin{equation}
		\label{eq:effective-rho-kappa}
		\overline\rho:=\sum_{i\in S}\rho(i)\pi_i,
		\qquad
		\overline\kappa:=\sum_{i\in S}\kappa(i)\pi_i.
	\end{equation}
	
	We define
	\[
	A_\varepsilon(t)
	:=
	t + \int_0^t\rho(\alpha_s^\varepsilon)\,dL_s,
	\qquad
	\gamma_\varepsilon:=A_\varepsilon^{-1},
	\]
	and
	\[
	A_0(t) := t+\overline\rho L_t,
	\qquad
	\gamma_0:=A_0^{-1}.
	\]
	For \(t\geq0\), set
	\[
	\widetilde X_t^\varepsilon
	:=
	W^+_{\gamma_\varepsilon(t)},
	\qquad
	\widetilde X_t^0
	:=
	W^+_{\gamma_0(t)}.
	\]
	Since the switching and homogenized processes are constructed from the
	same reflected Brownian reference process \(W^+\), we refer to this joint
	representation as the \emph{common-reference coupling}.
	
	Let \(f:\overline D\to\mathbb R\) be bounded and Lipschitz. Define
	\begin{equation}
		\label{eq:killing-functionals}
		K_\varepsilon(t)
		:=
		\int_0^{\gamma_\varepsilon(t)}
		\kappa(\alpha_{s/\varepsilon})\,dL_s,
		\qquad
		K_0(t)
		:=
		\overline\kappa L_{\gamma_0(t)}.
	\end{equation}
	Theorem~\ref{thm:abstract-averaging} shows that \(A_\varepsilon\) and
	\(K_\varepsilon\) average to \(A_0\) and \(K_0\), respectively, as $\varepsilon \downarrow 0 $.
	
	For a fixed realization of the environment \(\alpha\), define the quenched
	Feynman--Kac operator
	\begin{equation}
		\label{eq:quenched-semigroup}
		u^\varepsilon_\alpha(t,x)
		:=
		\Esp_x^W\left[
		f\left(W^+_{\gamma_\varepsilon(t)}\right)
		e^{-K_\varepsilon(t)}
		\right],
	\end{equation}
	where \(\Esp_x^W\) denotes expectation only with respect to \(W^+\), with the
	path of \(\alpha\) frozen. Define also the deterministic limit
	\begin{equation}
		\label{eq:limit-semigroup}
		u^0(t,x)
		:=
		\Esp_x^W\left[
		f\left(W^+_{\gamma_0(t)}\right)
		e^{-K_0(t)}
		\right].
	\end{equation}
	
	The following estimate controls the difference between quenched and homogenized
	observables in expectation over the environment.
	The proof is the object of Section~\ref{sec:proof_Wentzell}. 
	
	\begin{theorem}[Quantitative sticky-elastic homogenization]
		\label{thm:SE_homogenization_1}
		For every bounded Lipschitz function
		\(f:\overline D\to\mathbb R\) and every time horizon \(T>0\), the following estimates hold.
		\begin{enuroman}
			\item \label{item:marginal}
			For every \(0<\beta<1/8\), there exists
			\(C=C(f,T,\beta)>0\), independent of \(\varepsilon\), such that
			\begin{equation}
				\label{eq:semigroup-rate-nonmax}
				\sup_{x\in\overline D} \sup_{0\le t\le T}
				\Esp^\alpha\left[
				|u^\varepsilon_\alpha(t,x)-u^0(t,x)|
				\right]
				\le
				C\varepsilon^\beta.
			\end{equation}
			
			\item \label{item:sup}
			For every \(0<\beta<1/8\), there exists
			\(C=C(f,T,\beta)>0\), independent of \(\varepsilon\), such that
			\begin{equation}
				\label{eq:semigroup-rate}
				\sup_{x\in\overline D} \Esp^\alpha\left[
				\sup_{0\le t\le T}
				|u^\varepsilon_\alpha(t,x)-u^0(t,x)|
				\right]
				\le
				C\varepsilon^\beta.
			\end{equation}
		\end{enuroman}
		
		Consequently, for every \(x\in\overline D\), Markov's inequality yields
		\[
		\sup_{0\le t\le T}
		|u^\varepsilon_\alpha(t,x)-u^0(t,x)|
		\longrightarrow 0
		\qquad\text{in }\Prb^\alpha\text{-probability}.
		\]
		Hence, for every fixed \(t\in[0,T]\),
		\[
		u^\varepsilon_\alpha(t,x)
		\longrightarrow
		u^0(t,x)
		\qquad\text{in }\Prb^\alpha\text{-probability}.
		\]
	\end{theorem}
	
	The proof relies on the following estimates.
	\begin{itemize}
		\item The first consists of the local-time averaging estimates of
		Proposition~\ref{prop:LTA}. Namely,
		\begin{equation}
			\label{eq:local-time-averaging-assumption-nonmax}
			\sup_{x\in\overline D}
			\Esp_{x,\pi} \left[
			\left|
			\int_0^\tau h(\alpha_{s/\varepsilon})\,dL_s
			\right|^2
			\right]
			\le
			C_T(h)\varepsilon^{1/2},
		\end{equation}
		and that, for every $0<\zeta<1/2$,
		\begin{equation}
			\label{eq:local-time-averaging-assumption-max}
			\sup_{x\in\overline D}
			\Esp_{x,\pi} \left[
			\sup_{t\in [0,T]}\left|
			\int_0^t h(\alpha_{s/\varepsilon})\,dL_s
			\right|^2
			\right]
			\le
			C_{T,h,\zeta}\varepsilon^{\zeta}.
		\end{equation}
		These are proved in Appendix~\ref{app:averaging} by first showing that the boundary local time \(L\) satisfies the short-time estimate~\eqref{eq:abstract-A-first} for $\vartheta = 1/2 $, and then applying Theorem~\ref{thm:abstract-averaging} to \(L\).
		
		\item The second is the usual H\"older control of the reflected path and its local
		time. For $\theta\in(0,1/2)$ set
		\begin{equation}
			\label{eq:holder-norms0}
			[W^+]_{\theta,T}
			:=
			\sup_{0\le r<s\le T}
			\frac{|W^+_s-W^+_r|}{|s-r|^\theta},
			\qquad
			[L]_{\theta,T}
			:=
			\sup_{0\le r<s\le T}
			\frac{|L_s-L_r|}{|s-r|^\theta}.
		\end{equation}
		Proposition~\ref{prop:Holder} shows that, for every \(q\ge1\),
		\begin{equation}
			\label{eq:Kolmogorov_regularity}
			\sup_{x\in\overline D}
			\Esp_x^W \left[
			[W^+]_{\theta,T}^q+[L]_{\theta,T}^q
			\right]<\infty.
		\end{equation}
	\end{itemize}
	
	\subsection{The purely sticky case}
	
	\begin{proposition}
		\label{thm:sticky}
		For every \(p\in[1,2]\) and every \(T>0\), the following estimates hold.
		\begin{enuroman}
			\item \label{item:sticky_marginal}
			For every \(0<\beta<1/8\), there exists
			\(C=C(p,T,\beta)>0\), independent of \(\varepsilon\), such that
			\begin{equation}
				\label{eq:fixed-time-process-rate}
				\sup_{x\in\overline D} \sup_{0\le t\le T}
				\Esp_{x,\pi}\left[
				|\widetilde X^\varepsilon_t-\widetilde X^0_t|^p
				\right]^{1/p}
				\le
				C\varepsilon^{\beta}.
			\end{equation}
			
			\item \label{item:sticky_uniform}
			For every \(0<\beta<1/8\), there exists
			\(C=C(p,T,\beta)>0\), independent of \(\varepsilon\), such that
			\begin{equation}
				\label{eq:uniform-process-rate}
				\sup_{x\in\overline D} \Esp_{x,\pi}\left[
				\sup_{t\in[0,T]}
				|\widetilde X^\varepsilon_t-\widetilde X^0_t|^p
				\right]^{1/p}
				\le
				C\varepsilon^\beta.
			\end{equation}
		\end{enuroman}
		
		In particular, both estimates imply the corresponding convergence
		in probability by Markov's inequality.
	\end{proposition}
	
\begin{proof}
	Throughout the proof, \(C\) may change from line to line but is independent
	of \(\varepsilon\) and uniform in \(x\in\overline D\).
	
	\medskip\noindent\emph{Step 1: control of the clocks.}
	Set
	\begin{equation}
		\label{eq:Delta_epsilon}
		\Delta_\varepsilon(t)
		:=
		|\gamma_\varepsilon(t)-\gamma_0(t)|.
	\end{equation}
	Since \(A_\varepsilon(r)\ge r\) and \(A_0(r)\ge r\),
	\[
	\gamma_\varepsilon(t),\gamma_0(t)\le t\le T.
	\]
	Moreover,
	\[
	A_\varepsilon(b)-A_\varepsilon(a)
	=
	(b-a)+\int_a^b\rho(\alpha_{s/\varepsilon})\,dL_s
	\ge b-a,
	\]
	so \(\gamma_\varepsilon=A_\varepsilon^{-1}\) is \(1\)-Lipschitz.
	
	Taking \(y=\gamma_0(t)\) yields \(A_0(y)=t\) and
	\begin{equation}
		\label{eq:fixed-time-inverse-bound}
		\begin{aligned}
			\Delta_\varepsilon(t)
			&=
			|\gamma_\varepsilon(t)-\gamma_\varepsilon(A_\varepsilon(y))|\\
			&\le
			|A_0(y)-A_\varepsilon(y)|\\
			&=
			\left|
			\int_0^{\gamma_0(t)}
			\bigl(\rho(\alpha_{s/\varepsilon})-\overline\rho\bigr)\,dL_s
			\right|.
		\end{aligned}
	\end{equation}
	Since \(\gamma_0(t)\) depends only on the reflected path and is bounded by
	\(T\), the random-time averaging estimate
	\eqref{eq:local-time-averaging-assumption-nonmax}, which is uniform in the
	starting point, gives
	\begin{equation}
		\label{eq:fixed-time-gamma-L2-rate}
		\sup_{x\in\overline D}
		\sup_{0\le t\le T}
		\Esp_{x,\pi}\left[\Delta_\varepsilon(t)^2\right]
		\le
		C\varepsilon^{1/2}.
	\end{equation}
	Hence, by Lyapunov's inequality, for every \(a\in(0,2]\),
	\begin{equation}
		\label{eq:fixed-time-gamma-La-rate}
		\sup_{x\in\overline D}
		\sup_{0\le t\le T}
		\Esp_{x,\pi}\left[\Delta_\varepsilon(t)^a\right]
		\le
		C\varepsilon^{a/4}.
	\end{equation}
	
	\medskip
	\noindent\emph{Step 2: fixed-time spatial error.}
	Fix \(\theta\in(0,1/2)\). The H\"older seminorm gives
	\[
	\left|
	W^+_{\gamma_\varepsilon(t)}
	-
	W^+_{\gamma_0(t)}
	\right|
	\le
	[W^+]_{\theta,T}
	|\gamma_\varepsilon(t)-\gamma_0(t)|^\theta
	=
	[W^+]_{\theta,T}\Delta_\varepsilon(t)^\theta.
	\]
	Choose \(q>1\) so close to \(1\) that \(p\theta q\le2\), and let \(q'\) be its
	conjugate exponent. H\"older's inequality gives
	\begin{align}
		\Esp_{x,\pi}\left[
		\left|
		W^+_{\gamma_\varepsilon(t)}
		-
		W^+_{\gamma_0(t)}
		\right|^p
		\right]
		&\le
		\Esp_{x,\pi}\left[
		[W^+]_{\theta,T}^p
		\Delta_\varepsilon(t)^{p\theta}
		\right]
		\\
		&\le
		\Esp_x^W\left[[W^+]_{\theta,T}^{pq'}\right]^{1/q'}
		\Esp_{x,\pi}\left[
		\Delta_\varepsilon(t)^{p\theta q}
		\right]^{1/q}.
	\end{align}
	Taking the power \(1/p\), we obtain
	\begin{equation}
		\label{eq:fixed-time-space-error-rate}
		\Esp_{x,\pi}\left[
		\left|
		W^+_{\gamma_\varepsilon(t)}
		-
		W^+_{\gamma_0(t)}
		\right|^p
		\right]^{1/p}
		\le
		\Esp_x^W\left[[W^+]_{\theta,T}^{pq'}\right]^{1/(pq')}
		\Esp_{x,\pi}\left[
		\Delta_\varepsilon(t)^{p\theta q}
		\right]^{1/(pq)}.
	\end{equation}
	By Proposition~\ref{prop:Holder},
	\[
	\sup_{x\in\overline D}
	\Esp_x^W\left[[W^+]_{\theta,T}^{pq'}\right]
	<\infty,
	\]
	while \eqref{eq:fixed-time-gamma-La-rate}, with
	\(a=p\theta q\), gives
	\[
	\sup_{x\in\overline D}
	\sup_{0\le t\le T}
	\Esp_{x,\pi}\left[
	\Delta_\varepsilon(t)^{p\theta q}
	\right]^{1/(pq)}
	\le
	C\varepsilon^{\theta/4}.
	\]
	Therefore,
	\[
	\sup_{x\in\overline D}
	\sup_{0\le t\le T}
	\Esp_{x,\pi}\left[
	\left|
	W^+_{\gamma_\varepsilon(t)}
	-
	W^+_{\gamma_0(t)}
	\right|^p
	\right]^{1/p}
	\le
	C\varepsilon^{\theta/4}.
	\]
	Letting \(\theta\) approach \(1/2\), any \(0<\beta<1/8\) can be chosen with
	\[
	\beta<\frac{\theta}{4}.
	\]
	This proves \eqref{eq:fixed-time-process-rate}.
	
	\medskip
	\noindent\emph{Step 3: uniform-in-time error.}
	Fix \(0<\beta<1/8\). Choose \(0<\zeta<1/2\) and \(\theta\in(0,1/2)\) such that
	\begin{equation}
		\label{eq:choose-theta-zeta-sticky}
		\beta<\frac{\theta\zeta}{2}.
	\end{equation}
	By the definition of the H\"older seminorm and
	\eqref{eq:fixed-time-inverse-bound},
	\[
	\sup_{t\in[0,T]}
	\left|
	W^+_{\gamma_\varepsilon(t)}
	-
	W^+_{\gamma_0(t)}
	\right|
	\le
	[W^+]_{\theta,T}
	\sup_{t\in[0,T]}
	|\gamma_\varepsilon(t)-\gamma_0(t)|^\theta.
	\]
	Since \(\gamma_0(t)\in[0,T]\),
	\[
	\sup_{t\in[0,T]}\Delta_\varepsilon(t)
	\le
	\sup_{r\in[0,T]}
	\left|
	\int_0^r
	\bigl(\rho(\alpha_{s/\varepsilon})-\overline\rho\bigr)\,dL_s
	\right|.
	\]
	Applying~\eqref{eq:local-time-averaging-assumption-max}, uniformly in the
	starting point, gives
	\begin{equation}
		\label{eq:uniform-gamma-L2-rate}
		\sup_{x\in\overline D}
		\Esp_{x,\pi}\left[
		\sup_{t\in[0,T]}\Delta_\varepsilon(t)^2
		\right]
		\le
		C\varepsilon^\zeta.
	\end{equation}
	Consequently, for every \(a\in(0,2]\),
	\begin{equation}
		\label{eq:uniform-gamma-La-rate}
		\sup_{x\in\overline D}
		\Esp_{x,\pi}\left[
		\sup_{t\in[0,T]}\Delta_\varepsilon(t)^a
		\right]
		\le
		C\varepsilon^{a\zeta/2}.
	\end{equation}
	Let \(p\in[1,2]\), choose \(r>1\) close enough to \(1\) that
	\(p\theta r\le2\), and let \(r'\) be its conjugate exponent.
	H\"older's inequality and Proposition~\ref{prop:Holder} give
	\[
	\begin{aligned}
		&\Esp_{x,\pi}\left[
		\sup_{t\in[0,T]}
		\left|
		W^+_{\gamma_\varepsilon(t)}
		-
		W^+_{\gamma_0(t)}
		\right|^p
		\right]^{1/p}\\
		&\qquad\le
		\Esp_x^W\left[[W^+]_{\theta,T}^{pr'}\right]^{1/(pr')}
		\Esp_{x,\pi}\left[
		\sup_{t\in[0,T]}
		\Delta_\varepsilon(t)^{p\theta r}
		\right]^{1/(pr)}.
	\end{aligned}
	\]
	Hence,
	\[
	\sup_{x\in\overline D}
	\Esp_{x,\pi}\left[
	\sup_{t\in[0,T]}
	\left|
	W^+_{\gamma_\varepsilon(t)}
	-
	W^+_{\gamma_0(t)}
	\right|^p
	\right]^{1/p}
	\le
	C\varepsilon^{\theta\zeta/2}.
	\]
	By \eqref{eq:choose-theta-zeta-sticky}, the last term is
	\(O(\varepsilon^\beta)\). This proves \eqref{eq:uniform-process-rate}.
\end{proof}
	
	\subsection{Proof of Theorem~\ref{thm:SE_homogenization_1}}
	\label{sec:proof_Wentzell}
	
\begin{proof}
	[Proof of~\ref{item:marginal}]
	Again, \(C\) may change from line to line but is independent of
	\(\varepsilon\) and uniform in \(x\in\overline D\).
	
	\medskip\noindent\emph{Step 1: the killing functional.}
	Fix \(\beta\in(0,1/8)\) and choose \(\theta\in(0,1/2)\) so that
	\[
	\beta<\frac{\theta}{4}.
	\]
	For deterministic \(t\in[0,T]\), we split
	\[
	\begin{aligned}
		K_\varepsilon(t)-K_0(t)
		&=
		\int_0^{\gamma_\varepsilon(t)}
		\kappa(\alpha_{s/\varepsilon})\,dL_s
		-
		\overline\kappa L_{\gamma_0(t)}
		\\
		&=
		\int_0^{\gamma_0(t)}
		\bigl(\kappa(\alpha_{s/\varepsilon})-\overline\kappa\bigr)\,dL_s
		+
		\int_{\gamma_0(t)}^{\gamma_\varepsilon(t)}
		\kappa(\alpha_{s/\varepsilon})\,dL_s,
	\end{aligned}
	\]
	where the second integral is understood in the oriented sense.
	Hence
	\begin{equation}
		\label{eq:fixed-time-killing-bound}
		|K_\varepsilon(t)-K_0(t)|
		\le
		\left|
		\int_0^{\gamma_0(t)}
		\bigl(\kappa(\alpha_{s/\varepsilon})-\overline\kappa\bigr)\,dL_s
		\right|
		+
		\|\kappa\|_\infty
		|L_{\gamma_\varepsilon(t)}-L_{\gamma_0(t)}|.
	\end{equation}
	
	The first term in \eqref{eq:fixed-time-killing-bound} is an averaged
	local-time integral. Since \(\gamma_0(t)\) depends only on the reflected
	path and is bounded by \(T\), the random-time averaging estimate
	\eqref{eq:local-time-averaging-assumption-nonmax}, which is uniform in the
	starting point, together with Lyapunov's inequality gives
	\begin{equation}
		\label{eq:fixed-time-kappa-average-L1}
		\sup_{x\in\overline D}
		\sup_{0\le t\le T}
		\Esp_{x,\pi}\left[
		\left|
		\int_0^{\gamma_0(t)}
		\bigl(\kappa(\alpha_{s/\varepsilon})-\overline\kappa\bigr)\,dL_s
		\right|
		\right]
		\le
		C\varepsilon^{1/4}.
	\end{equation}
	
	For the second term in \eqref{eq:fixed-time-killing-bound}, the H\"older
	seminorm of \(L\) gives
	\[
	|L_{\gamma_\varepsilon(t)}-L_{\gamma_0(t)}|
	\le
	[L]_{\theta,T}
	|\gamma_\varepsilon(t)-\gamma_0(t)|^\theta
	=
	[L]_{\theta,T}\Delta_\varepsilon(t)^\theta,
	\]
	where
	\[
	\Delta_\varepsilon(t)
	:=
	|\gamma_\varepsilon(t)-\gamma_0(t)|.
	\]
	Choose \(q>1\) with \(\theta q\le2\), and let \(q'\) be conjugate to \(q\).
	H\"older's inequality gives
	\[
	\Esp_{x,\pi}\left[
	[L]_{\theta,T}\Delta_\varepsilon(t)^\theta
	\right]
	\le
	\Esp_x^W\left[[L]_{\theta,T}^{q'}\right]^{1/q'}
	\Esp_{x,\pi}\left[
	\Delta_\varepsilon(t)^{\theta q}
	\right]^{1/q}.
	\]
	By Proposition~\ref{prop:Holder},
	\[
	\sup_{x\in\overline D}
	\Esp_x^W\left[[L]_{\theta,T}^{q'}\right]
	<\infty,
	\]
	while \eqref{eq:fixed-time-gamma-La-rate} controls the second factor:
	\[
	\sup_{x\in\overline D}
	\sup_{0\le t\le T}
	\Esp_{x,\pi}\left[
	\Delta_\varepsilon(t)^{\theta q}
	\right]^{1/q}
	\le
	C\varepsilon^{\theta/4}.
	\]
	Therefore,
	\begin{equation}
		\label{eq:fixed-time-local-time-shift-rate}
		\sup_{x\in\overline D}
		\sup_{0\le t\le T}
		\Esp_{x,\pi}\left[
		[L]_{\theta,T}\Delta_\varepsilon(t)^\theta
		\right]
		\le
		C\varepsilon^{\theta/4}.
	\end{equation}
	Combining \eqref{eq:fixed-time-killing-bound},
	\eqref{eq:fixed-time-kappa-average-L1}, and
	\eqref{eq:fixed-time-local-time-shift-rate} gives
	\[
	\sup_{x\in\overline D}
	\sup_{0\le t\le T}
	\Esp_{x,\pi}\left[
	|K_\varepsilon(t)-K_0(t)|
	\right]
	\le
	C\varepsilon^{1/4}+C\varepsilon^{\theta/4}.
	\]
	Since \(\beta<\theta/4<1/4\), this yields
	\begin{equation}
		\label{eq:fixed-time-killing-rate}
		\sup_{x\in\overline D}
		\sup_{0\le t\le T}
		\Esp_{x,\pi}\left[
		|K_\varepsilon(t)-K_0(t)|
		\right]
		\le
		C\varepsilon^\beta.
	\end{equation}
	
	\medskip
	\noindent\emph{Step 2: the Feynman--Kac weights.}
	Define
	\begin{equation}
		\label{eq:Y-epsilon-Y-zero}
		Y_\varepsilon(t)
		:=
		f\left(W^+_{\gamma_\varepsilon(t)}\right)e^{-K_\varepsilon(t)},
		\qquad
		Y_0(t)
		:=
		f\left(W^+_{\gamma_0(t)}\right)e^{-K_0(t)}.
	\end{equation}
	Inserting the mixed term
	\(f(W^+_{\gamma_0(t)})e^{-K_\varepsilon(t)}\) and splitting gives
	\[
	\begin{aligned}
		Y_\varepsilon(t)-Y_0(t)
		&=
		\left[
		f\left(W^+_{\gamma_\varepsilon(t)}\right)
		-
		f\left(W^+_{\gamma_0(t)}\right)
		\right]e^{-K_\varepsilon(t)}
		\\
		&\quad
		+
		f\left(W^+_{\gamma_0(t)}\right)
		\left[
		e^{-K_\varepsilon(t)}-e^{-K_0(t)}
		\right].
	\end{aligned}
	\]
	Using \(e^{-K_\varepsilon(t)}\le1\), the Lipschitz continuity of \(f\),
	and the fact that \(r\mapsto e^{-r}\) is \(1\)-Lipschitz on
	\([0,\infty)\), we obtain
	\begin{equation}
		\label{eq:Y-difference-bound}
		|Y_\varepsilon(t)-Y_0(t)|
		\le
		\mathrm{Lip}(f)
		\left|
		W^+_{\gamma_\varepsilon(t)}
		-
		W^+_{\gamma_0(t)}
		\right|
		+
		\|f\|_\infty
		|K_\varepsilon(t)-K_0(t)|.
	\end{equation}
	Taking expectations and using
	Proposition~\ref{thm:sticky}\ref{item:sticky_marginal} with \(p=1\),
	together with \eqref{eq:fixed-time-killing-rate}, gives
	\begin{equation}
		\label{eq:fixed-time-Y-rate}
		\sup_{x\in\overline D}
		\sup_{0\le t\le T}
		\Esp_{x,\pi}\left[
		|Y_\varepsilon(t)-Y_0(t)|
		\right]
		\le
		C\varepsilon^\beta.
	\end{equation}
	
	\medskip
	\noindent\emph{Step 3: quenched evolution.}
	For fixed \(\alpha\), \eqref{eq:quenched-semigroup} and
	\eqref{eq:limit-semigroup} give
	\[
	u^\varepsilon_\alpha(t,x)-u^0(t,x)
	=
	\Esp_x^W\left[
	Y_\varepsilon(t)-Y_0(t)
	\right].
	\]
	Hence, for every \(t\in[0,T]\),
	\[
	|u^\varepsilon_\alpha(t,x)-u^0(t,x)|
	\le
	\Esp_x^W\left[
	|Y_\varepsilon(t)-Y_0(t)|
	\right].
	\]
	Average now over \(\alpha\), take the supremum over deterministic \(t\),
	and then over \(x\in\overline D\):
	\[
	\begin{aligned}
		&\sup_{x\in\overline D}
		\sup_{0\le t\le T}
		\Esp^\alpha
		\left[
		|u^\varepsilon_\alpha(t,x)-u^0(t,x)|
		\right]
		\\
		&\qquad\le
		\sup_{x\in\overline D}
		\sup_{0\le t\le T}
		\Esp^\alpha
		\Esp_x^W
		\left[
		|Y_\varepsilon(t)-Y_0(t)|
		\right]
		\\
		&\qquad=
		\sup_{x\in\overline D}
		\sup_{0\le t\le T}
		\Esp_{x,\pi}
		\left[
		|Y_\varepsilon(t)-Y_0(t)|
		\right],
	\end{aligned}
	\]
	which combined with \eqref{eq:fixed-time-Y-rate} gives
	\eqref{eq:semigroup-rate-nonmax}.
\end{proof}
	
\begin{proof}[Proof of~\ref{item:sup}]
	Again, \(C\) may change from line to line but is independent of
	\(\varepsilon\) and uniform in \(x\in\overline D\). Fix \(0<\beta<1/8\) and choose
	\(0<\zeta<1/2\) and \(\theta<1/2\) so that
	\begin{equation}
		\label{eq:choose-theta-zeta-SE}
		\beta<\frac{\theta\zeta}{2}.
	\end{equation}
	
	For fixed \(\alpha\),
	\[
	\sup_{t\in[0,T]}
	|u^\varepsilon_\alpha(t,x)-u^0(t,x)|
	\le
	\Esp_x^W\left[
	\sup_{t\in[0,T]}|Y_\varepsilon(t)-Y_0(t)|
	\right].
	\]
	Averaging over \(\alpha\) and taking the supremum over
	\(x\in\overline D\) gives
	\begin{equation}
		\label{eq:uniform-quenched-joint-bound}
		\sup_{x\in\overline D}
		\Esp^\alpha\left[
		\sup_{t\in[0,T]}|u^\varepsilon_\alpha(t,x)-u^0(t,x)|
		\right]
		\le
		\sup_{x\in\overline D}
		\Esp_{x,\pi}\left[
		\sup_{t\in[0,T]}|Y_\varepsilon(t)-Y_0(t)|
		\right].
	\end{equation}
	By~\eqref{eq:Y-difference-bound},
	\begin{equation}
		\label{eq:quenched_Wentzell_split}
		\begin{aligned}
			\sup_{t\in[0,T]}|Y_\varepsilon(t)-Y_0(t)|
			&\le
			\mathrm{Lip}(f)
			\sup_{t\in[0,T]}
			|W^+_{\gamma_\varepsilon(t)}-W^+_{\gamma_0(t)}|\\
			&\quad+
			\|f\|_\infty
			\sup_{t\in[0,T]}|K_\varepsilon(t)-K_0(t)|.
		\end{aligned}
	\end{equation}
	Proposition~\ref{thm:sticky}\ref{item:sticky_uniform} yields
	\begin{equation}
		\label{eq:uniform-space-error-L1}
		\sup_{x\in\overline D}
		\Esp_{x,\pi}\left[
		\sup_{t\in[0,T]}
		|W^+_{\gamma_\varepsilon(t)}-W^+_{\gamma_0(t)}|
		\right]
		\le C\varepsilon^\beta.
	\end{equation}
	
	For the killing term, the same decomposition as before gives
	\[
	\begin{aligned}
		\sup_{t\in[0,T]}|K_\varepsilon(t)-K_0(t)|
		&\le
		\sup_{r\in[0,T]}
		\left|
		\int_0^r
		\bigl(\kappa(\alpha_{s/\varepsilon})-\overline\kappa\bigr)\,dL_s
		\right|\\
		&\quad+
		\|\kappa\|_\infty [L]_{\theta,T}
		\sup_{t\in[0,T]}\Delta_\varepsilon(t)^\theta.
	\end{aligned}
	\]
	By~\eqref{eq:local-time-averaging-assumption-max} and Lyapunov's inequality,
	\[
	\sup_{x\in\overline D}
	\Esp_{x,\pi}\left[
	\sup_{r\in[0,T]}
	\left|
	\int_0^r
	\bigl(\kappa(\alpha_{s/\varepsilon})-\overline\kappa\bigr)\,dL_s
	\right|
	\right]
	\le C\varepsilon^{\zeta/2}.
	\]
	Choose \(r>1\) such that \(\theta r\le2\), and let \(r'\) be its conjugate
	exponent. By H\"older's inequality,
	\eqref{eq:Kolmogorov_regularity}, and
	\eqref{eq:uniform-gamma-La-rate},
	\[
	\begin{aligned}
		&\Esp_{x,\pi}\left[
		[L]_{\theta,T}
		\sup_{t\in[0,T]}\Delta_\varepsilon(t)^\theta
		\right]\\
		&\qquad\le
		\Esp_x^W\left[[L]_{\theta,T}^{r'}\right]^{1/r'}
		\Esp_{x,\pi}\left[
		\sup_{t\in[0,T]}\Delta_\varepsilon(t)^{\theta r}
		\right]^{1/r}.
	\end{aligned}
	\]
	By Proposition~\ref{prop:Holder},
	\[
	\sup_{x\in\overline D}
	\Esp_x^W\left[[L]_{\theta,T}^{r'}\right]
	<\infty,
	\]
	while \eqref{eq:uniform-gamma-La-rate} gives
	\[
	\sup_{x\in\overline D}
	\Esp_{x,\pi}\left[
	\sup_{t\in[0,T]}\Delta_\varepsilon(t)^{\theta r}
	\right]^{1/r}
	\le
	C\varepsilon^{\theta\zeta/2}.
	\]
	Therefore,
	\[
	\sup_{x\in\overline D}
	\Esp_{x,\pi}\left[
	[L]_{\theta,T}
	\sup_{t\in[0,T]}\Delta_\varepsilon(t)^\theta
	\right]
	\le
	C\varepsilon^{\theta\zeta/2}.
	\]
	Consequently, by \eqref{eq:choose-theta-zeta-SE},
	\[
	\sup_{x\in\overline D}
	\Esp_{x,\pi}\left[
	\sup_{t\in[0,T]}|K_\varepsilon(t)-K_0(t)|
	\right]
	\le
	C\varepsilon^{\zeta/2}+C\varepsilon^{\theta\zeta/2}
	\le C\varepsilon^\beta.
	\]
	Combining the above with
	\eqref{eq:uniform-quenched-joint-bound},
	\eqref{eq:quenched_Wentzell_split}, and
	\eqref{eq:uniform-space-error-L1}
	gives \eqref{eq:semigroup-rate}.
\end{proof}

	\subsection{Sharpness of local-time averaging}
	\label{subsec:sharpness}
	
We first examine the sharpness of the averaging estimate underlying the
homogenization result. We show that the boundary-local-time averaging rate in
Theorem~\ref{thm:abstract-averaging} is already sharp in one dimension, for
reflected Brownian motion on \([0,1]\).

	\begin{proposition}
		\label{prop:boundary-LT-sharpness}
		Let \(D=(0,1)\), and let \(W^+\) be reflected Brownian motion on
		\(\overline D=[0,1]\), with generator
		\(\frac12\frac{d^2}{dx^2}\), started from \(0\).
		Write \(L^{(0)}\) and \(L^{(1)}\) for the regulators at \(0\) and \(1\),
		respectively, so that
		\[
		L=L^{(0)}+L^{(1)}
		\]
		is the total boundary local time.
		
		Let \(S=\{1,2\}\), and let \(\alpha\) be the independent Markov chain
		with generator
		\[
		Q
		=
		\begin{pmatrix}
			-a&a\\
			a&-a
		\end{pmatrix},
		\qquad a>0.
		\]
		Its invariant distribution is
		\[
		\pi=\left(\frac12,\frac12\right).
		\]
		Assume \(\alpha_0\sim\pi\), and set
		\[
		h(1)=-1,
		\qquad
		h(2)=1.
		\]
		Then \(\pi(h)=0\) and
		\[
		\Esp_\pi^\alpha
		\left[
		h(\alpha_t)h(\alpha_s)
		\right]
		=
		e^{-2a|t-s|}.
		\]
		
		For
		\[
		J_\varepsilon^{(0)}(T)
		:=
		\int_0^T
		h(\alpha_{s/\varepsilon})\,dL_s^{(0)}
		= \int_0^T
		\left(h(\alpha_{s/\varepsilon}) - \pi(h) \right)\,dL_s^{(0)},
		\]
		we have, for every \(T>0\),
		\begin{equation}
			\label{eq:boundary-LT-sharpness}
			\Esp_{0,\pi}
			\left|
			J_\varepsilon^{(0)}(T)
			\right|^2
			\sim
			\frac{\Esp_0[L_T^{(0)}]}{\sqrt a}\,
			\varepsilon^{1/2},
			\qquad
			\varepsilon\downarrow0.
		\end{equation}
		In particular, the exponent \(1/2\) in the squared \(L^2\)
		random-time averaging estimate cannot be improved in general.
		
		For the total boundary local time,
		\[
		J_\varepsilon(T)
		:=
		\int_0^T
		h(\alpha_{s/\varepsilon})\,dL_s,
		\]
		there exist constants \(0<c_T\le C_T<\infty\) such that
		\[
		c_T\varepsilon^{1/2}
		\le
		\Esp_{0,\pi}
		\left|
		J_\varepsilon(T)
		\right|^2
		\le
		C_T\varepsilon^{1/2}
		\]
		for all sufficiently small \(\varepsilon\).
		
		In particular, the exponent \(1/2\) in the boundary-local-time averaging estimate \eqref{eq:local-time-averaging-assumption-nonmax} cannot be improved in general.
	\end{proposition}
	
	\begin{proof}
		This is a specialization of
		Theorem~\ref{thm:abstract-averaging} with
		\[
		A=L^{(0)}.
		\]
		Since \(\pi(h)=0\),
		\[
		J_\varepsilon^{(0)}(T)
		=
		\int_0^T
		\bigl(
		h(\alpha_{s/\varepsilon})-\pi(h)
		\bigr)\,dL_s^{(0)}.
		\]
		The theorem gives the upper bound of order
		\(\varepsilon^{1/2}\). We show that this order is attained.
		
		Let \(p^{W^+}(t,x,y)\) be the reflected heat kernel on \([0,1]\).
		With our normalization, \(L^{(0)}\) is the PCAF in Revuz
		correspondence with the "surface" measure \(\frac12\delta_0\). Conditioning on the reflected
		path and using stationarity of \(\alpha\),
		\[
		\Esp_{0,\pi}
		\left|
		J_\varepsilon^{(0)}(T)
		\right|^2
		=
		\Esp_0
		\int_0^T\int_0^T
		e^{-2a|s-r|/\varepsilon}\,
		dL_r^{(0)}\,dL_s^{(0)}.
		\]
		The diagonal carries no
		\(dL^{(0)}\otimes dL^{(0)}\)-mass. By symmetry, the Markov property,
		and the Revuz formula,
		\begin{equation}
			\label{eq:boundary-LT-double-kernel}
			\Esp_{0,\pi}
			\left|
			J_\varepsilon^{(0)}(T)
			\right|^2
			=
			\frac12
			\int_0^T
			p^{W^+}(r,0,0)
			\int_0^{T-r}
			e^{-2au/\varepsilon}
			p^{W^+}(u,0,0)\,du\,dr .
		\end{equation}
		
		The image representation of the reflected heat kernel gives
		\[
		p^{W^+}(u,0,0)
		=
		\frac{2}{\sqrt{2\pi u}}
		\sum_{k\in\mathbb Z}
		\exp\left(-\frac{2k^2}{u}\right).
		\]
		Since the \(k=0\) term dominates as \(u\downarrow0\),
		\begin{equation}
			\label{eq:Neumann-small-time-boundary}
			p^{W^+}(u,0,0)
			\sim
			\sqrt{\frac{2}{\pi u}},
			\qquad u\downarrow0.
		\end{equation}
		and, for every fixed \(T>0\),
		\[
		p^{W^+}(u,0,0)
		\le
		C_Tu^{-1/2},
		\qquad
		0<u\le T.
		\]
		
		For \(R>0\), put
		\[
		I_\varepsilon(R)
		:=
		\int_0^R
		e^{-2au/\varepsilon}
		p^{W^+}(u,0,0)\,du,
		\]
		so that
		\begin{equation}
			\label{eq:boundary-LT-double-kernel_2}
			\Esp_{0,\pi}
			\left|
			J_\varepsilon^{(0)}(T)
			\right|^2
			= 
			\frac12
			\int_0^T
			p^{W^+}(r,0,0)
			I_\varepsilon(T-r) \,dr .
		\end{equation}
		With \(u=\varepsilon v\),
		\[
		\frac{I_\varepsilon(R)}{\sqrt\varepsilon}
		=
		\sqrt\varepsilon
		\int_0^{R/\varepsilon}
		e^{-2av}
		p^{W^+}(\varepsilon v,0,0)\,dv.
		\]
		By \eqref{eq:Neumann-small-time-boundary} and dominated convergence,
		\begin{equation}
			\label{eq:sharpness_proof_dom1}
			\frac{I_\varepsilon(R)}{\sqrt\varepsilon}
			\longrightarrow
			\sqrt{\frac2\pi}
			\int_0^\infty
			e^{-2av}v^{-1/2}\,dv
			=
			\frac1{\sqrt a}.
		\end{equation}
		Moreover, the heat-kernel bound gives
		\begin{equation}
			\label{eq:sharpness_proof_dom2}
			\sup_{0<R\le T}
			\frac{I_\varepsilon(R)}{\sqrt{\varepsilon}}
			\le C_{T,a},
			\qquad 0<\varepsilon\le1.
		\end{equation}
		
		By~\eqref{eq:sharpness_proof_dom1} and \eqref{eq:sharpness_proof_dom2}, we can apply the dominated convergence to~\eqref{eq:boundary-LT-double-kernel_2}, which yields
		\[
		\varepsilon^{-1/2}
		\Esp_{0,\pi}
		\left|
		J_\varepsilon^{(0)}(T)
		\right|^2
		\longrightarrow
		\frac1{2\sqrt a}
		\int_0^T
		p^{W^+}(r,0,0)\,dr.
		\]
		By the Revuz formula,
		\[
		\Esp_0[L_T^{(0)}]
		=
		\frac12
		\int_0^T
		p^{W^+}(r,0,0)\,dr,
		\]
		which proves \eqref{eq:boundary-LT-sharpness}.
		
		For \(L=L^{(0)}+L^{(1)}\), conditioning on \(W^+\) gives
		\[
		\Esp_{0,\pi}|J_\varepsilon(T)|^2
		=
		\Esp_0
		\int_0^T\int_0^T
		e^{-2a|s-r|/\varepsilon}\,dL_r\,dL_s.
		\]
		Since the covariance kernel is nonnegative and
		\[
		dL=dL^{(0)}+dL^{(1)},
		\]
		all four contributions in the expansion are nonnegative. Hence
		\[
		\Esp_{0,\pi}|J_\varepsilon(T)|^2
		\ge
		\Esp_{0,\pi}|J_\varepsilon^{(0)}(T)|^2
		\ge c_T\varepsilon^{1/2}
		\]
		for all sufficiently small \(\varepsilon\), by
		\eqref{eq:boundary-LT-sharpness}.
		The upper bound follows from
		Theorem~\ref{thm:abstract-averaging} with \(A=L\) and
		\(\vartheta=1/2\).
	\end{proof}
	
	\begin{remark}
		\label{rem:boundary-LT-max-sharpness}
		Since
		\[
		\sup_{0\le t\le T}
		\left|
		J_\varepsilon(t)
		\right|^2
		\ge
		\left|
		J_\varepsilon(T)
		\right|^2,
		\]
		Proposition~\ref{prop:boundary-LT-sharpness} also shows that the
		polynomial exponent \(1/2\) in the maximal squared \(L^2\) estimate is
		sharp. It does not determine whether the logarithmic factor in
		Theorem~\ref{thm:abstract-averaging} is necessary.
	\end{remark}

	\subsection{Improved rates under compatibility}
\label{subsec:improved}
	
We show that the observable rate in
Theorem~\ref{thm:SE_homogenization_1} can improve for smoother compatible
test functions.
	
	\begin{theorem}
		\label{thm:improved-smooth-observable}
		Assume that, for some \(c\ge0\),
		\[
		\kappa_i=c\rho_i,
		\qquad i\in S,
		\]
		and let \(f\in C^2(\overline D)\) satisfy
		\[
		\partial_{\mathbf n}f=0
		\qquad\text{on }\partial D.
		\]
		Then, for every \(T>0\),
		\begin{equation}
			\label{eq:smooth-observable-fixed-rate}
			\sup_{x\in\overline D} \sup_{0\le t\le T}
			\Esp^\alpha
			\left|
			u_\alpha^\varepsilon(t,x)-u^0(t,x)
			\right|
			\le
			C\varepsilon^{1/4},
		\end{equation}
		and
		\begin{equation}
			\label{eq:smooth-observable-max-rate}
			\sup_{x\in\overline D} \Esp^\alpha
			\left[
			\sup_{0\le t\le T}
			\left|
			u_\alpha^\varepsilon(t,x)-u^0(t,x)
			\right|
			\right]
			\le
			C\varepsilon^{1/4}
			\bigl(1+|\log\varepsilon|\bigr).
		\end{equation}
		In particular, the second estimate is
		\(O(\varepsilon^\beta)\) for every \(\beta<1/4\).
	\end{theorem}
	
\begin{proof}
	Since \(\kappa_i=c\rho_i\) for every \(i\), and since \(A_\varepsilon(\gamma_\varepsilon(t))=t\), 
	\[
	K_\varepsilon(t)
	=
	c\int_0^{\gamma_\varepsilon(t)}
	\rho(\alpha_{s/\varepsilon})\,dL_s
	=
	c\bigl(t-\gamma_\varepsilon(t)\bigr).
	\]
	Since also \(\overline\kappa=c\overline\rho\),
	\[
	K_0(t)
	=
	c\bigl(t-\gamma_0(t)\bigr).
	\]
	Therefore
	\[
	u_\alpha^\varepsilon(t,x)
	=
	e^{-ct}
	\Esp_x^W
	\left[
	e^{c\gamma_\varepsilon(t)}
	f(W^+_{\gamma_\varepsilon(t)})
	\right],
	\]
	and similarly
	\[
	u^0(t,x)
	=
	e^{-ct}
	\Esp_x^W
	\left[
	e^{c\gamma_0(t)}
	f(W^+_{\gamma_0(t)})
	\right].
	\]
	
	Fix \(\alpha\) and \(t\in[0,T]\). The two inverse clocks are bounded
	stopping times for the reflected Brownian filtration. Define
	\[
	F(s,x):=e^{cs}f(x).
	\]
	Applying It\^o's formula to the process
	\[
	F(s,W_s^+)=e^{cs}f(W_s^+),
	\]
	and using \(f\in C^2(\overline D)\) together with
	\(\partial_{\mathbf n}f=0\), the boundary local-time term vanishes.
	Thus, for every such stopping time \(\tau\),
	\[
	\Esp_x^W
	\left[
	e^{c\tau}f(W_\tau^+)
	\right]
	=
	f(x)
	+
	\Esp_x^W
	\int_0^\tau
	e^{cs}
	\left(
	cf+\frac12\Delta f
	\right)(W_s^+)\,ds.
	\]
	
	Applying this identity with
	\(\tau=\gamma_\varepsilon(t)\) and \(\tau=\gamma_0(t)\),
	subtracting the resulting expressions, and using the representations above,
	we obtain
	\[
	|u_\alpha^\varepsilon(t,x)-u^0(t,x)|
	\le
	C_f\Esp_x^W
	|\gamma_\varepsilon(t)-\gamma_0(t)|.
	\]
	
	Averaging the preceding estimate over the environment and taking the
	supremum over \(x\in\overline D\) gives
	\[
	\begin{aligned}
		&\sup_{x\in\overline D}
		\sup_{0\le t\le T}
		\Esp^\alpha
		\left|
		u_\alpha^\varepsilon(t,x)-u^0(t,x)
		\right|
		\\
		&\qquad\le
		C_f
		\sup_{x\in\overline D}
		\sup_{0\le t\le T}
		\Esp_{x,\pi}
		\left|
		\gamma_\varepsilon(t)-\gamma_0(t)
		\right|
		\\
		&\qquad\le
		C_f
		\sup_{x\in\overline D}
		\sup_{0\le t\le T}
		\left(
		\Esp_{x,\pi}
		\left|
		\gamma_\varepsilon(t)-\gamma_0(t)
		\right|^2
		\right)^{1/2}
		\\
		&\qquad\le
		C\varepsilon^{1/4},
	\end{aligned}
	\]
	where the second inequality is Lyapunov's inequality and the last one follows
	from~\eqref{eq:fixed-time-gamma-L2-rate}.
	
	For the uniform estimate, the same argument gives
	\[
	\sup_{0\le t\le T}
	\left|
	u_\alpha^\varepsilon(t,x)-u^0(t,x)
	\right|
	\le
	C_f
	\Esp_x^W
	\left[
	\sup_{0\le t\le T}
	|\gamma_\varepsilon(t)-\gamma_0(t)|
	\right].
	\]
	The maximal estimate of
	Theorem~\ref{thm:abstract-averaging}, with
	\(\vartheta=1/2\), followed by Lyapunov's inequality, yields
	\[
	\sup_{x\in\overline D}
	\Esp^\alpha
	\left[
	\sup_{0\le t\le T}
	\left|
	u_\alpha^\varepsilon(t,x)-u^0(t,x)
	\right|
	\right]
	\le
	C\varepsilon^{1/4}
	\bigl(1+|\log\varepsilon|\bigr).
	\]
	This completes the proof.
\end{proof}
	
	\begin{remark}
		Proposition~\ref{prop:boundary-LT-sharpness} shows that the
		\(\varepsilon^{1/2}\) second-moment boundary-local-time averaging rate is sharp, but
		this does not imply that the resulting observable rates are optimal.
		Under the proportionality condition of
		Theorem~\ref{thm:improved-smooth-observable}, one obtains the better
		rate \(\varepsilon^{1/4}\) for \(C^2\) observables satisfying the Neumann boundary condition, and
		\(\varepsilon^{1/4}(1+|\log\varepsilon|)\) for the expected uniform error.
	\end{remark}

	\section{Quenched and annealed evolution equations}
	\label{sec:quenched}
For the quenched problem, we fix a realization of the chain.
The resulting boundary condition is then time-dependent, so the spatial
coordinate alone is generally not Markov. In particular, the future
evolution depends not only on the current position but also on the current
phase of the frozen environment. Consequently, the quenched operators
need not form a semigroup on \(C(\overline D)\). We therefore work directly
with the Feynman--Kac representation and derive the corresponding backward
identities.

These equations are not needed for the preceding homogenization proof;
rather, they identify the analytic structure of the stochastic model and
prepare the forward formulation of the next section.
	
	Let
	\[
	\alpha:[0,\infty)\to S:=\{1,\dots,N\}
	\]
	be a fixed c\`adl\`ag, piecewise constant trajectory of the Markov chain.
	
	Define the quenched clock
	\begin{equation}
		\label{eq:quenched-clock}
		A^\alpha(s)
		:=
		s+\int_0^s\rho(\alpha_u)\,dL_u,
		\qquad s\ge0,
	\end{equation}
	and let \(\gamma^\alpha\) be its right-inverse:
	\begin{equation}
		\label{eq:quenched-clock-inverse}
		\gamma^\alpha(t)
		:=
		\inf\{s\ge0:A^\alpha(s)>t\},
		\qquad t\ge0.
	\end{equation}
	The corresponding killing functional is
	\begin{equation}
		\label{eq:quenched-killing-functional}
		K^\alpha(t)
		:=
		\int_0^{\gamma^\alpha(t)}
		\kappa(\alpha_s)\,dL_s.
	\end{equation}
	
	For \(f\in C(\overline D)\), define the quenched Feynman--Kac operators
	\begin{equation}
		\label{eq:quenched-spatial-operator}
		(S_t^\alpha f)(x)
		:=
		\Esp_x\!\left[
		f\!\left(W^+_{\gamma^\alpha(t)}\right)
		e^{-K^\alpha(t)}
		\right],
		\qquad t\ge0,\quad x\in\overline D.
	\end{equation}
	In general $(S_t^\alpha)_{t\ge0}$ is not a semigroup. At physical time $t$
	the relevant phase of the frozen environment is $\gamma^\alpha(t)$, and that
	information is not contained in the spatial position alone.
	
	\begin{lemma}
		[Continuity in initial position]
		\label{lem:reflected-start-continuity}
		On a bounded $C^2$ domain, couple normally reflected Brownian motions $W^x,W^y$ using the same Brownian driver, with regulators $L^x,L^y$ normalized as in Section~\ref{ssec:notation}: 
		\begin{equation}
			W_t^{x}
			=
			x+B_t-\int_0^t \mathbf n(W_s^x)\,dL_s^x,
			\qquad
			W_t^{y}
			=
			y+B_t-\int_0^t \mathbf n(W_s^y)\,dL_s^y,
			\qquad t\ge0.
		\end{equation}
		For every $T>0$, as $y\to x$,
		\[
		\sup_{s\le T}|W_s^y-W_s^x|+\sup_{s\le T}|L_s^y-L_s^x|
		\longrightarrow0\quad\text{in probability}.
		\]
	\end{lemma}
	
	\begin{proof}
		Set $Z_t := W^{x}_t - W^{y}_t $. 
		Under common-noise coupling,
		\[
		d Z_t
		=
		-\mathbf n(W_t^x)\,dL_t^x
		+\mathbf n(W_t^y)\,dL_t^y.
		\]
		Since \(D\) is bounded with \(C^2\) boundary, it satisfies a uniform
		exterior sphere condition; see, e.g.,
		\cite[Chapter~2, Problem~2.11]{gilbarg2001_elliptic_PDEs}.
		Hence, by \cite[Condition~(A) and Remark~1.1]{saisho1987stochastic},
		there exists \(c_D>0\) such that
		\[
		(z-w)\cdot\mathbf n(z)
		\ge -c_D|z-w|^2,
		\qquad z\in\partial D,\quad w\in\overline D.
		\]
		
		Under the common-noise coupling,
		\[
		dZ_t
		=
		-\mathbf n(W_t^x)\,dL_t^x
		+\mathbf n(W_t^y)\,dL_t^y.
		\]
		Since \(Z\) has finite variation,
		\[
		d|Z_t|^2
		=
		-2 Z_t\cdot \mathbf n(W_t^x)\,dL_t^x
		+2 Z_t\cdot \mathbf n(W_t^y)\,dL_t^y.
		\]
		Now \(dL_t^x\) is carried by \(\{W_t^x\in\partial D\}\), so on its support,
		\[
		Z_t\cdot \mathbf n(W_t^x)
		=
		(W_t^x-W_t^y)\cdot \mathbf n(W_t^x)
		\ge -c_D|Z_t|^2.
		\]
		Likewise, since \(dL_t^y\) is carried by \(\{W_t^y\in\partial D\}\),
		\[
		(W_t^y-W_t^x)\cdot \mathbf n(W_t^y)
		\ge -c_D|Z_t|^2,
		\]
		or equivalently
		\[
		Z_t\cdot \mathbf n(W_t^y)
		\le c_D|Z_t|^2.
		\]
		Therefore
		\[
		d|Z_t|^2
		\le
		2c_D|Z_t|^2\,(dL_t^x+dL_t^y).
		\]
		And hence, Gr\"onwall's inequality (see, e.g., \cite[Lemma~2.2]{saisho1987stochastic}) gives
		\[
		\sup_{s\le T}|W_s^x-W_s^y|^2
		\le
		|x-y|^2
		\exp\!\left\{2c_D(L_T^x+L_T^y)\right\}.
		\]
		
		By Lemma~\ref{lem:short-time-growth-local-time},
		\[
		C_T:=\sup_{z\in\overline D}\Esp_z L_T<\infty.
		\]
		Hence, for every \(R>0\), Markov's inequality gives
		\[
		\sup_{x,y\in\overline D}
		\Pr\!\left(L_T^x+L_T^y>R\right)
		\le \frac{2C_T}{R}.
		\]
		Thus \(\{L_T^x+L_T^y:x,y\in\overline D\}\) is tight; see
		\cite[Section~5]{billingsley1968Convergence_of}.
		Combining this with the preceding estimate yields
		\[
		\sup_{s\le T}|W_s^x-W_s^y|\longrightarrow0
		\qquad\text{in probability as }y\to x.
		\]

		For the local time convergence, we use the smooth distance function
		\(r\) described in Appendix~\ref{app:Holder_continuity}: it belongs to
		\(C^2(\overline D)\), agrees with the inward distance in a collar of
		\(\partial D\), and satisfies
		\[
		\nabla r\cdot\mathbf n=-1
		\qquad\text{on }\partial D.
		\]
		It\^o's formula gives
		\[
		L_t^x
		=
		r(W_t^x)-r(x)
		-\int_0^t \nabla r(W_s^x)\cdot dB_s
		-\frac12\int_0^t \Delta r(W_s^x)\,ds.
		\]
		Subtracting the corresponding identity for \(y\), the uniform continuity
		of \(r\) and \(\Delta r\), together with the preceding path convergence,
		shows that the non-martingale terms converge uniformly in probability.
		Moreover, by the BDG inequality,
		\[
		\Esp\!\left[
		\sup_{t\le T}
		\left|
		\int_0^t
		\bigl(\nabla r(W_s^x)-\nabla r(W_s^y)\bigr)\cdot dB_s
		\right|^2
		\right]
		\le
		C\,\Esp\int_0^T
		\left|\nabla r(W_s^x)-\nabla r(W_s^y)\right|^2 ds.
		\]
		The right-hand side tends to zero because \(\nabla r\) is bounded and
		uniformly continuous and
		\(\sup_{s\le T}|W_s^x-W_s^y|\to0\) in probability. Hence
		\[
		\sup_{t\le T}|L_t^x-L_t^y|
		\longrightarrow0
		\qquad\text{in probability}.
		\]
		This completes the proof.
	\end{proof}
	
	\subsection{Quenched backward equation}
	
	For \(s\ge0\), set
	\begin{equation}
		\label{eq:quenched-operational-killing}
		H^\alpha_s
		:=
		\exp\left\{-B^{\alpha}_s\right\}, \quad 
		B^\alpha_s := \int_0^s\kappa(\alpha_u)\,dL_u
	\end{equation}
	so that
	\[
	e^{-K^\alpha(t)}
	=
	H^\alpha_{\gamma^\alpha(t)},
	\qquad t\ge0.
	\]
	
	Let
	\[
	S_\alpha
	:=
	\{i\in S:\alpha_s=i\text{ for some }s\ge0\}
	\]
	be the set of regimes visited by the fixed path \(\alpha\), and define the
	common smooth test class
	\begin{equation}
		\label{eq:quenched-common-domain}
		\mathcal D_\alpha
		:=
		\left\{
		\begin{array}{l}
			f\in C^2(\overline D):\
			\partial_{\mathbf n}f
			+\kappa_i f
			+\dfrac{\rho_i}{2}\Delta f
			=0
			\text{ on }\partial D\\
			\text{for every }i\in S_\alpha
		\end{array}
		\right\}.
	\end{equation}
	
	\begin{theorem}
		\label{thm:quenched_sticky_elastic}
		The family \((S_t^\alpha)_{t\ge0}\) defined by
		\eqref{eq:quenched-spatial-operator} has the following properties.
		
		\begin{enumerate}
			\item (Preservation of continuity, contractivity, and continuity at the origin.) For every \(t\ge0\),
			\[
			S^{\alpha}_t f\in C(\ol D), \quad 
			\|S_t^\alpha f\|_\infty
			\le
			\|f\|_\infty,
			\quad f\in C(\overline D).
			\]
			Moreover,
			\[
			\|S_t^\alpha f-f\|_\infty\longrightarrow0,
			\qquad t\downarrow0.
			\]
			
			\item (Dynkin formula.) For every \(f\in C^2(\overline D)\),
			\begin{equation}
				\label{eq:quenched-backward-dynkin}
				\begin{aligned}
					(S_t^\alpha f)(x)-f(x)
					&=
					\Esp_x\!\left[
					\frac12
					\int_0^{\gamma^\alpha(t)}
					H^\alpha_s
					\Delta f(W_s^+)\,ds
					\right]
					\\
					&\quad
					-
					\Esp_x\!\left[
					\int_0^{\gamma^\alpha(t)}
					H^\alpha_s
					\left(
					\partial_{\mathbf n}f
					+\kappa(\alpha_s)f
					\right)(W_s^+)\,dL_s
					\right].
				\end{aligned}
			\end{equation}
			
			\item (Quenched backward equation.) If \(f\in\mathcal D_\alpha\), then
			\begin{equation}
				\label{eq:quenched-backward-mild}
				S_t^\alpha f
				=
				f+\frac12\int_0^t S_s^\alpha(\Delta f)\,ds,
				\qquad t\ge0.
			\end{equation}
			In particular, the map \(t\mapsto S_t^\alpha f\) is continuously
			differentiable and satisfies
			\begin{equation}
				\label{eq:quenched-backward-derivative}
				\frac{d}{dt}S_t^\alpha f
				=
				\frac12S_t^\alpha(\Delta f),
				\qquad t\ge0.
			\end{equation}
		\end{enumerate}
	\end{theorem}
	
	\begin{proof}
		\medskip\noindent\emph{1. Contractivity and continuity at the origin.}
		Since \(K^\alpha(t)\ge0\), for every \(f\in C(\overline D)\),
		\[
		|(S_t^\alpha f)(x)|
		\le
		\Esp_x\!\left[
		\left|f\!\left(W^+_{\gamma^\alpha(t)}\right)\right|
		e^{-K^\alpha(t)}
		\right]
		\le
		\|f\|_\infty.
		\]
		For continuity at $t=0$, since
		\[
		0\le\gamma^\alpha(t)\le t,
		\]
		we have
		\begin{align}
			|(S_t^\alpha f)(x)-f(x)|
			&\le
			\Esp_x\!\left[
			\left|
			f\!\left(W^+_{\gamma^\alpha(t)}\right)-f(x)
			\right|
			\right]
			\nonumber\\
			&\quad
			+
			\|f\|_\infty
			\Esp_x\!\left[1-e^{-K^\alpha(t)}\right].
			\label{eq:quenched-strong-continuity-bound}
		\end{align}
		The first term tends to zero uniformly in \(x\), by the uniform continuity
		of \(f\) and Proposition~\ref{prop:Holder}, which yields
		\[
		\sup_{x\in\overline D}
		\Esp_x\!\left[
		\sup_{0\le s\le t}|W_s^+-x|^q
		\right]
		\le C t^{q\theta}.
		\]
		
		For the second, \(1-e^{-z}\le z\) gives
		\[
		\Esp_x\!\left[1-e^{-K^\alpha(t)}\right]
		\le
		\|\kappa\|_\infty\Esp_x[L_t].
		\]
		Again by Proposition~\ref{prop:Holder},
		\[
		\sup_{x\in\overline D}\Esp_x[L_t]
		\le C t^\theta,
		\]
		so the right-hand side tends to zero uniformly in \(x\). Thus
		\[
		\|S_t^\alpha f-f\|_\infty\longrightarrow0
		\qquad\text{as }t\downarrow0.
		\]
		
		It remains to show that \(S_t^\alpha\) preserves continuity. Let \(x_n\to x\)
		and use the common-reference coupling from Lemma~\ref{lem:reflected-start-continuity}.
		Then, on every bounded operational-time interval,
		\[
		\sup_{s\le T}|W_s^{x_n}-W_s^x|
		+
		\sup_{s\le T}|L_s^{x_n}-L_s^x|
		\longrightarrow0
		\qquad\text{in probability}.
		\]
		Since the fixed path \(\alpha\) has only finitely many jumps on \([0,T]\),
		the weighted local-time integrals are finite sums of increments of \(L\);
		hence the corresponding clocks and killing functionals converge uniformly
		in probability. As the clocks are continuous and strictly increasing, their
		inverses converge uniformly on compact time intervals as well. Hence, for each
		fixed \(t\),
		\[
		e^{-K^{\alpha,x_n}(t)}
		f\!\left(W_{\gamma^{\alpha,x_n}(t)}^{x_n}\right)
		\longrightarrow
		e^{-K^{\alpha,x}(t)}
		f\!\left(W_{\gamma^{\alpha,x}(t)}^{x}\right)
		\]
		in probability. 
		Since these random variables are uniformly bounded by
		\(\|f\|_\infty\), the convergence also holds in \(L^1\), see, e.g., \cite[Theorem~4.6.3]{durrett2019probability}.  
		Thus,
		\[
		(S_t^\alpha f)(x_n)\to(S_t^\alpha f)(x).
		\]
		
		The same coupling argument, together with Proposition~\ref{prop:Holder} and
		the continuity of the inverse clocks, shows that
		\[
		(t,x)\longmapsto (S_t^\alpha f)(x)
		\]
		is jointly continuous on bounded time intervals. Since
		\([0,T]\times\overline D\) is compact, it follows that
		\(t\mapsto S_t^\alpha f\) is continuous in \(C(\overline D)\).
		
		\medskip\noindent\emph{2. Dynkin formula.}
		For $f\in C^2(\overline D)$, It\^o's formula with the outward normal
		convention gives
		\[
		f(W_s^+)-f(x)
		=
		\int_0^s\nabla f(W_u^+)\cdot dB_u
		+
		\frac12\int_0^s\Delta f(W_u^+)\,du
		-
		\int_0^s\partial_{\mathbf n}f(W_u^+)\,dL_u.
		\]
		Applying the product rule to
		\[
		H^\alpha_s f(W_s^+)
		\]
		and stopping at the bounded stopping time \(\gamma^\alpha(t)\le t\), we obtain
		\begin{align}
			&H^\alpha_{\gamma^\alpha(t)}
			f\!\left(W^+_{\gamma^\alpha(t)}\right)-f(x)
			\nonumber\\
			&=
			\mathcal N_t
			+
			\frac12
			\int_0^{\gamma^\alpha(t)}
			H^\alpha_s\Delta f(W_s^+)\,ds
			\nonumber\\
			&\quad
			-
			\int_0^{\gamma^\alpha(t)}
			H^\alpha_s
			\left(
			\partial_{\mathbf n}f
			+\kappa(\alpha_s)f
			\right)(W_s^+)\,dL_s,
			\label{eq:quenched-product-formula}
		\end{align}
		where $\mathcal N_t$ is a zero-mean martingale term. Since
		\[
		H^\alpha_{\gamma^\alpha(t)}
		=
		e^{-K^\alpha(t)},
		\]
		taking expectations gives \eqref{eq:quenched-backward-dynkin}.
		
		\medskip
		\noindent\emph{3. Quenched backward equation.}
		Let \(f\in\mathcal D_\alpha\). 
		Since \(dL_s\) is supported on \(\{W_s^+\in\partial D\}\), the boundary
		condition defining \(\mathcal D_\alpha\) gives, \(dL_s\)-a.e.,
		\[
		-
		\left(
		\partial_{\mathbf n}f+\kappa(\alpha_s)f
		\right)(W_s^+)
		=
		\frac{\rho(\alpha_s)}2
		\Delta f(W_s^+).
		\]
		Substitution into \eqref{eq:quenched-backward-dynkin} and the definition of $ A^{\alpha}$ yields
		\begin{align}
			(S_t^\alpha f)(x)-f(x)
			&=
			\frac12
			\Esp_x\!\left[
			\int_0^{\gamma^\alpha(t)}
			H^\alpha_s\Delta f(W_s^+)
			\,dA^\alpha(s)
			\right].
			\label{eq:quenched-clock-integral}
		\end{align}
		Since \(A^\alpha\) is continuous and strictly increasing and
		\[
		A^\alpha(\gamma^\alpha(t))=t,
		\]
		the change of variables \(u=A^\alpha(s)\) gives
		\begin{align}
			(S_t^\alpha f)(x)-f(x)
			&=
			\frac12
			\int_0^t
			\Esp_x\!\left[
			e^{-K^\alpha(u)}
			\Delta f\!\left(W^+_{\gamma^\alpha(u)}\right)
			\right]du
			\nonumber\\
			&=
			\frac12
			\int_0^t
			(S_u^\alpha\Delta f)(x)\,du.
		\end{align}
		This is \eqref{eq:quenched-backward-mild}. 
		
		By Step~1, the map
		\(u\mapsto S_u^\alpha(\Delta f)\) is continuous in \(C(\overline D)\).
		Hence, by the fundamental theorem of calculus applied to
		\eqref{eq:quenched-backward-mild},
		\[
		\frac{d}{dt}S_t^\alpha f
		=
		\frac12 S_t^\alpha(\Delta f).
		\]
		This completes the proof. 
	\end{proof}
	
	\begin{remark}
		\label{rem:quenched-no-semigroup}
		Equation~\eqref{eq:quenched-backward-dynkin} holds for every
		$f\in C^2(\overline D)$, without imposing the boundary conditions
		defining $\mathcal D_\alpha$.
		The integral identity \eqref{eq:quenched-backward-mild}
		is established for $f\in\mathcal D_\alpha$.
		This class need not be dense in $C(\overline D)$.
		
		Indeed, if two visited regimes $i,j$ satisfy
		$\rho_i=\rho_j$ and $\kappa_i\neq\kappa_j$, then subtracting
		the two boundary conditions gives
		\[
		(\kappa_i-\kappa_j)f=0
		\qquad\text{on }\partial D.
		\]
		Thus every $f\in\mathcal D_\alpha$ vanishes on $\partial D$,
		and $\|1-f\|_\infty\ge1$.
		Hence $\mathcal D_\alpha$ is not dense in $C(\overline D)$
		in the uniform norm.
		
On the other hand,
\[
C_c^\infty(D)\subset\mathcal D_\alpha,
\]
so \(\mathcal D_\alpha\) is always nontrivial and infinite-dimensional.
Viewed as a class of functions on \(D\), it is dense in \(L^p(D,dx)\) for
every \(1\le p<\infty\), since \(C_c^\infty(D)\) is dense in these spaces.
This density is with respect to the interior Lebesgue measure only; it does
not extend to measures that charge the boundary \(\partial D\), which are
precisely the measures appearing in the forward formulation.
		
		We use \eqref{eq:quenched-backward-mild} as a test-function
		identity for the evolution defined by the stochastic
		representation. 
	\end{remark}
	
	\subsection{Annealed backward equation}
	\label{sec:annealed}
	
	For the annealed problem we average over both sources of randomness. The state
	is therefore the joint process $(X_t,I_t)$ of
	Subsection~\ref{subsec:operational-time-construction}; unlike the quenched
	spatial evolution, this process is Markov.
	
	For \(\Phi\in C(\overline D\times S)\), write
	\[
	\Phi_i(x):=\Phi(x,i),
	\]
	and define
	\begin{equation}
		\label{eq:annealed-interior-operator}
		(\mathcal L\Phi)_i(x)
		:=
		\frac12\Delta\Phi_i(x)
		+
		\sum_{j\in S}q_{ij}\Phi_j(x),
		\qquad x\in D.
	\end{equation}
	
	Let \((\mathcal S_t)_{t\ge0}\) be the sub-Markov transition operators defined
	in \eqref{eq:construction-semigroup}, namely
	\begin{equation}
		\label{eq:annealed-semigroup-definition}
		(\mathcal S_t\Phi)(x,i)
		:=
		\Esp_{x,i}\left[
		\Phi(X_t,I_t)\mathbf 1_{\{t<\zeta\}}
		\right].
	\end{equation}
	
	\begin{theorem}[Annealed sticky-elastic switching semigroup]
		\label{thm:annealed-sticky-elastic-semigroup}
		The family \((\mathcal S_t)_{t\ge0}\) is a strongly continuous contraction
		semigroup on \(C(\overline D\times S)\).
		
		Let
		\begin{equation}
			\label{eq:D0-definition}
			\mathcal D_0
			:=
			\left\{
			\begin{array}{l}
				\Phi\in C(\overline D\times S):
				\Phi_i\in C^2(\overline D)\text{ for every }i\in S,\text{ and}\\
				\partial_{\mathbf n}\Phi_i
				+\kappa_i\Phi_i
				+\rho_i
				\left(
				\dfrac12\Delta\Phi_i
				+\displaystyle\sum_{j\in S}q_{ij}\Phi_j
				\right)
				=0
				\quad\text{on }\partial D,\ i\in S
			\end{array}
			\right\}.
		\end{equation}
		Then \(\mathcal D_0\subset D(\mathcal G)\), where \(\mathcal G\) is the
		generator of \((\mathcal S_t)_{t\ge0}\), and for every
		\(\Phi\in\mathcal D_0\),
		\begin{equation}
			\label{eq:switching-generator-abstract}
			(\mathcal G\Phi)(x,i)
			=
			(\mathcal L\Phi)_i(x)
			\qquad x\in D.
		\end{equation}
		
		We only claim the inclusion $\mathcal D_0\subset D(\mathcal G)$; no core
		property or converse characterization of the generator domain is needed.
		
		For every \(\Phi\in C(\overline D\times S)\), the function
		\begin{equation}
			\label{eq:u-semigroup-definition}
			u(t,x,i):=(\mathcal S_t\Phi)(x,i)
		\end{equation}
		is the unique mild solution of the abstract Cauchy problem
		\begin{equation}
			\label{eq:abstract-cauchy-problem}
			\frac{d}{dt}u(t)=\mathcal Gu(t),
			\qquad
			u(0)=\Phi.
		\end{equation}
		If, additionally, $u$ is continuously differentiable in time with values in $C(\overline D\times S)$ and $u(t)\in\mathcal D_0$, then $u$ solves the coupled
		backward system
		\begin{equation}
			\label{eq:coupled-backward-system}
			\partial_tu(t,x,i)
			=
			\frac12\Delta_xu(t,x,i)
			+
			\sum_{j\in S}q_{ij}u(t,x,j),
			\qquad x\in D,
		\end{equation}
		with sticky-elastic boundary condition
		\begin{equation}
			\label{eq:coupled-sticky-elastic-bc}
			\partial_{\mathbf n}u(t,\xi,i)
			+\kappa_i u(t,\xi,i)
			+\rho_i
			\left(
			\frac12\Delta_xu(t,\xi,i)
			+\sum_{j\in S}q_{ij}u(t,\xi,j)
			\right)
			=0,
			\qquad \xi\in\partial D,
		\end{equation}
		and initial datum
		\begin{equation}
			\label{eq:coupled-initial-condition}
			u(0,x,i)=\Phi(x,i).
		\end{equation}
		Using the interior equation, the same boundary condition reads
		\begin{equation}
			\label{eq:dynamic-wentzell-bc}
			\partial_{\mathbf n}u(t,\xi,i)
			+\kappa_i u(t,\xi,i)
			+\rho_i\partial_tu(t,\xi,i)
			=0.
		\end{equation}
	\end{theorem}

	\begin{proof}
		\medskip\noindent\emph{Step 1: contraction and semigroup property.}
		Both follow from Proposition~\ref{prop:operational-time-process}.

		\medskip
		\noindent\emph{Step 2: preservation of continuity.}
		Fix \(t>0\), \(i\in S\), and let \(x_m\to x\). Couple the reflected motions as in
		Lemma~\ref{lem:reflected-start-continuity}, using the same independent chain
		started from \(i\).
		
		As in the proof of Theorem~\ref{thm:quenched_sticky_elastic}, the reflected
		paths and local times converge uniformly in probability on bounded operational
		intervals. Since the chain has only finitely many jumps on such intervals, the
		corresponding clocks and killing functionals converge uniformly in probability,
		and so do their inverse clocks on compact physical-time intervals.
		
		It remains only to treat the chain evaluated at the inverse clock. For \(i\neq j\),
		let
		\[
		N_{ij}(s)
		:=
		\#\{u\le s:\alpha_{u-}=i,\ \alpha_u=j\}.
		\]
		Then the physical-time counting process \(N_{ij}(\gamma(t))\) has compensator
		\[
		\int_0^t q_{ij}\mathbf1_{\{I_{r-}=i\}}\,d\gamma(r).
		\]
		Since \(\gamma\) is \(1\)-Lipschitz, this compensator is absolutely continuous
		with respect to Lebesgue measure. Hence \(I\) has no jump at a fixed physical
		time. Therefore, almost surely, \(\gamma^x(t)\) is not a jump time of the common
		chain, and
		\[
		\alpha_{\gamma^{x_m}(t)}
		\longrightarrow
		\alpha_{\gamma^x(t)}
		\qquad\text{in probability}.
		\]
		
		Together with the convergences recalled above, this implies that the killed
		observables converge in probability. Since they are uniformly bounded by
		\(\|\Phi\|_\infty\), the convergence also holds in \(L^1\). Therefore
		\[
		(\mathcal S_t\Phi)(x_m,i)\to(\mathcal S_t\Phi)(x,i).
		\]
		
		\noindent\emph{Step 3: strong continuity.}
		Let
		\[
		q_*:=\max_{i\in S}(-q_{ii}).
		\]
		Because $\gamma(t)\le t$, the chain changes state before
		$\gamma(t)$ with probability at most $q_*t$. The reflected motion is
		uniformly stochastically continuous on the compact state space (recall the proof of Theorem~\ref{thm:quenched_sticky_elastic}). Hence uniform
		continuity of $\Phi$ gives
		\[
		\sup_{(x,i)}
		\Esp_{x,i}\!\left[
		\left|
		\Phi(W^+_{\gamma(t)},\alpha_{\gamma(t)})
		-\Phi(x,i)
		\right|
		\right]
		\longrightarrow0.
		\]
		The killing term satisfies
		\[
		\Esp_{x,i}[1-e^{-K(t)}]
		\le
		\|\kappa\|_\infty\Esp_x[L_t],
		\]
		which also tends to zero uniformly in $x$. Thus
		\[
		\|\mathcal S_t\Phi-\Phi\|_\infty\longrightarrow0,
		\qquad t\downarrow0.
		\]
		This proves strong continuity.
		
		\medskip
		\noindent\emph{Step 4: identification of the generator.}
		Let \(\Phi\in\mathcal D_0\), and set
		\[
		B(s):=\int_0^s\kappa(\alpha_u)\,dL_u.
		\]
		Applying the generalized It\^o formula for switching diffusions to
		\(\Phi(W_s^+,\alpha_s)\), together with the reflection term, gives
		\[
		d\Phi(W_s^+,\alpha_s)
		=
		d\mathcal M_s
		+
		(\mathcal L\Phi)_{\alpha_s}(W_s^+)\,ds
		-
		\partial_{\mathbf n}\Phi_{\alpha_s}(W_s^+)\,dL_s,
		\]
		where \(\mathcal M\) is a local martingale. Here the compensated jumps of
		\(\alpha\) produce the term
		\[
		\sum_{j\in S}q_{\alpha_s j}\Phi_j(W_s^+)\,ds
		\]
		appearing in \(\mathcal L\); see, e.g.,
		\cite[Section~2.2]{ngyuen2025_book}.
		
		Since
		\[
		d(e^{-B(s)})
		=
		-e^{-B(s)}\kappa(\alpha_s)\,dL_s,
		\]
		the product rule yields
		\begin{align}
			&e^{-B(\gamma(t))}
			\Phi(W^+_{\gamma(t)},\alpha_{\gamma(t)})
			-\Phi(x,i)
			\nonumber\\
			&=
			\mathcal N_t
			+
			\int_0^{\gamma(t)}
			e^{-B(s)}
			(\mathcal L\Phi)_{\alpha_s}(W_s^+)\,ds
			\nonumber\\
			&\quad
			-
			\int_0^{\gamma(t)}
			e^{-B(s)}
			\left(
			\partial_{\mathbf n}\Phi_{\alpha_s}
			+\kappa(\alpha_s)\Phi_{\alpha_s}
			\right)(W_s^+)\,dL_s,
			\label{eq:annealed-ito-formula}
		\end{align}
		where \(\mathcal N_t\) has zero expectation; indeed,
		\(\gamma(t)\le t\) is a bounded stopping time and all coefficients involved are
		bounded on \(\overline D\times S\).
		
		For \(\Phi\in\mathcal D_0\),
		\[
		-
		\left(
		\partial_{\mathbf n}\Phi_i+\kappa_i\Phi_i
		\right)
		=
		\rho_i(\mathcal L\Phi)_i
		\qquad\text{on }\partial D.
		\]
		Since \(dL_s\) is supported on \(\{W_s^+\in\partial D\}\), taking expectations
		in \eqref{eq:annealed-ito-formula} and using
		\[
		dA(s)=ds+\rho(\alpha_s)\,dL_s
		\]
		gives
		\begin{equation}
			\label{eq:annealed-dynkin-formula}
			(\mathcal S_t\Phi)(x,i)-\Phi(x,i)
			=
			\Esp_{x,i}\!\left[
			\int_0^{\gamma(t)}
			e^{-B(s)}
			(\mathcal L\Phi)_{\alpha_s}(W_s^+)\,dA(s)
			\right].
		\end{equation}
		
		Since \(A\) is continuous and strictly increasing, with inverse \(\gamma\),
		the change of variables \(r=A(s)\) gives
		\[
		\mathcal S_t\Phi-\Phi
		=
		\int_0^t\mathcal S_r(\mathcal L\Phi)\,dr.
		\]
		As \(\mathcal L\Phi\in C(\overline D\times S)\), strong continuity implies
		\[
		\left\|
		\frac{\mathcal S_t\Phi-\Phi}{t}-\mathcal L\Phi
		\right\|_\infty
		\le
		\sup_{0\le r\le t}
		\|\mathcal S_r(\mathcal L\Phi)-\mathcal L\Phi\|_\infty
		\longrightarrow0.
		\]
		Thus
		\[
		\Phi\in D(\mathcal G),
		\qquad
		\mathcal G\Phi=\mathcal L\Phi.
		\]
		
		\medskip
		\noindent\emph{Step 5: the backward equation.}
		The semigroup now gives the abstract Cauchy problem directly:
		$u(t)=\mathcal S_t\Phi$ is its unique mild solution.
		
		If $u$ is continuously differentiable in time with values in $C(\overline D\times S)$ and $u(t)\in\mathcal D_0$, the generator formula gives
		\eqref{eq:coupled-backward-system}, while membership in $\mathcal D_0$ gives
		\eqref{eq:coupled-sticky-elastic-bc}. The initial condition follows from
		strong continuity. Finally, substituting
		\[
		\partial_tu_i
		=
		\frac12\Delta u_i+\sum_{j\in S}q_{ij}u_j
		\]
		into \eqref{eq:coupled-sticky-elastic-bc} gives
		\eqref{eq:dynamic-wentzell-bc}.
	\end{proof}

	\section{Forward formulation and homogenization}
	\label{sec:forward-homogenization}
	
The backward formulation describes the evolution of observables and is the
natural setting for the Feynman--Kac representation used in the preceding
sections. From a physical point of view, however, the forward formulation is
often more natural, since physical distributions evolve forward in time. It
describes the evolution of an initial distribution of particles under the
switching boundary dynamics. In particular, when killing is present, the
total mass of the forward measure represents the surviving mass at a given
physical time and, for a probability initial measure, the corresponding
survival probability. It is therefore useful to establish quantitative
homogenization results for the quenched forward evolution.
	
	We first derive the quenched forward formulation by duality and then show that
	the homogenization estimates transfer to the forward measures. In the general
	case we obtain quantitative convergence in the bounded-Lipschitz distance,
	while under the proportionality and compatibility assumptions of
	Theorem~\ref{thm:improved-smooth-observable} we recover the improved
	\(1/4\)-scale for smooth observables and, in particular, for the total
	surviving mass.
	
	Let us introduce some further notation for this part. 
	For a compact space \(F\), let \(\mathcal M(F)\) denote the space of finite
	signed Radon measures on \(F\), equipped with the total variation norm
	\(\|\cdot\|_{\mathrm{TV}}\).
	
	We use \emph{narrow convergence} for the weak-\(*\) topology induced by
	\(C(F)\): thus, for \(\mu_n,\mu\in\mathcal M(F)\),
	\[
	\mu_n\to\mu
	\quad\text{narrowly}
	\]
	if
	\[
	\int_F f\,d\mu_n
	\longrightarrow
	\int_F f\,d\mu
	\qquad
	\text{for every }f\in C(F).
	\]
	Accordingly, a map \(t\mapsto\mu_t\in\mathcal M(F)\) is said to be
	\emph{narrowly continuous} if
	\[
	t\longmapsto\int_F f\,d\mu_t
	\]
	is continuous for every \(f\in C(F)\).
	
	\subsection{Quenched forward formulation}

	Let \(\mu_0\in\mathcal M(\overline D)\) be an initial finite measure.
	When \(\mu_0\) is nonnegative, it may be interpreted as the initial spatial
	distribution of particles. We define the corresponding quenched forward
	measure by duality:
	\begin{equation}
		\label{eq:quenched-forward-duality}
		\mu_t^\alpha
		:=
		(S_t^\alpha)^*\mu_0,
	\end{equation}
	that is,
	\[
	\int_{\overline D}f(x)\,\mu_t^\alpha(dx)
	=
	\int_{\overline D}(S_t^\alpha f)(x)\,\mu_0(dx),
	\qquad f\in C(\overline D).
	\]
	
	Integrating the backward identities against the initial measure
	gives the forward equation and identifies the mass lost through
	killing. For the latter it suffices to take $f=1 $.
	
	\begin{theorem}[Quenched forward equation]
		\label{thm:forward_quenched}
		Let \(\mu_0\in\mathcal M(\overline D)\), and define
		\((\mu_t^\alpha)_{t\ge0}\) by
		\eqref{eq:quenched-forward-duality}. Then the following properties hold.
		
		\begin{enumerate}
			\item For every \(t\ge0\),
			\begin{equation}
				\label{eq:forward-tv-contraction}
				\|\mu_t^\alpha\|_{\mathrm{TV}}
				\le
				\|\mu_0\|_{\mathrm{TV}}.
			\end{equation}
			If \(\kappa\equiv0\), then total mass is conserved.
			
			\item The map \(t\mapsto\mu_t^\alpha\) is narrowly continuous.
			
			\item For every \(f\in C^2(\overline D)\),
			\begin{equation}
				\label{eq:quenched-forward-dynkin}
				\begin{aligned}
					&\int_{\overline D}f(y)\,\mu_t^\alpha(dy)
					-
					\int_{\overline D}f(y)\,\mu_0(dy)
					\\
					&=
					\int_{\overline D}
					\Esp_x\!\left[
					\frac12
					\int_0^{\gamma^\alpha(t)}
					H^\alpha_s
					\Delta f(W_s^+)\,ds
					\right]\mu_0(dx)
					\\
					&\quad
					-
					\int_{\overline D}
					\Esp_x\!\left[
					\int_0^{\gamma^\alpha(t)}
					H^\alpha_s
					\left(
					\partial_{\mathbf n}f
					+\kappa(\alpha_s)f
					\right)(W_s^+)\,dL_s
					\right]\mu_0(dx).
				\end{aligned}
			\end{equation}
			
			\item If \(f\in\mathcal D_\alpha\), then
			\begin{equation}
				\label{eq:weak-forward-equation}
				\int_{\overline D}f\,d\mu_t^\alpha
				-
				\int_{\overline D}f\,d\mu_0
				=
				\frac12
				\int_0^t
				\int_{\overline D}\Delta f(y)\,
				\mu_s^\alpha(dy)\,ds.
			\end{equation}
			Equivalently,
			\begin{equation}
				\label{eq:weak-forward-derivative}
				\frac{d}{dt}
				\int_{\overline D}f\,d\mu_t^\alpha
				=
				\frac12
				\int_{\overline D}\Delta f\,d\mu_t^\alpha
			\end{equation}
			for every \(t\ge0\).
			
			\item More generally, let
			\(\varphi\in C^1([0,T];C(\overline D))\) satisfy
			\[
			\varphi(t,\cdot)\in\mathcal D_\alpha,
			\qquad
			\Delta\varphi\in C([0,T]\times\overline D).
			\]
			Then
			\begin{equation}
				\label{eq:weak-forward-time-dependent}
				\begin{aligned}
					&\int_{\overline D}\varphi(T,x)\,\mu_T^\alpha(dx)
					-
					\int_{\overline D}\varphi(0,x)\,\mu_0(dx)
					\\
					&\qquad
					=
					\int_0^T
					\int_{\overline D}
					\left(
					\partial_t\varphi(t,x)
					+\frac12\Delta\varphi(t,x)
					\right)
					\mu_t^\alpha(dx)\,dt.
				\end{aligned}
			\end{equation}
		\end{enumerate}
	\end{theorem}
	
	The short duality proof is deferred to Appendix~\ref{app:forward-proofs}.
	
	\begin{remark}
		\label{rem:forward-density}
		The family \((\mu_t^\alpha)_{t\ge0}\) is the quenched spatial law associated
		with the frozen environment \(\alpha\).
		Equation~\eqref{eq:quenched-forward-dynkin} is its
		exact forward weak formulation. Equation~\eqref{eq:weak-forward-equation}
		is a closed relation only on the common test class
		\(\mathcal D_\alpha\).
		
		In general, without adjoining the operational phase as an additional state
		variable, one cannot expect a closed Markov forward equation or a uniqueness
		theorem based only on the spatial variable. This does not affect the
		homogenization result, which uses directly the quenched Feynman--Kac
		representation \eqref{eq:quenched-spatial-operator}.
	\end{remark}
	
	\subsection{Forward homogenization}
	\label{subsec:forward-homogenization}
	
	We now transfer the quenched homogenization estimates to the corresponding
	forward measure evolutions. Let \(\mu_0\) be a finite nonnegative Radon
	measure on \(\overline D\). For the accelerated environment
	\[
	\alpha^\varepsilon_s:=\alpha_{s/\varepsilon},
	\]
	define the quenched forward measure
	\(\mu_t^{\varepsilon,\alpha}\) by duality:
	\begin{equation}
		\label{eq:forward-epsilon-duality}
		\int_{\overline D} f(y)\,
		\mu_t^{\varepsilon,\alpha}(dy)
		:=
		\int_{\overline D}
		u_\alpha^\varepsilon(t,x;f)\,
		\mu_0(dx),
		\qquad
		f\in C(\overline D),
	\end{equation}
	where \(u_\alpha^\varepsilon(t,x;f)\) denotes the quenched
	Feynman--Kac evolution with initial datum \(f\).
	Likewise, define the homogenized forward measure \(\mu_t^0\) by
	\begin{equation}
		\label{eq:forward-zero-duality}
		\int_{\overline D} f(y)\,\mu_t^0(dy)
		:=
		\int_{\overline D}
		u^0(t,x;f)\,\mu_0(dx).
	\end{equation}
	
	For a bounded Lipschitz function \(f\), set
	\[
	\|f\|_{\mathrm{BL}}
	:=
	\|f\|_\infty+\operatorname{Lip}(f),
	\]
	and, for finite nonnegative measures \(\mu,\nu\) on \(\overline D\), define
	the bounded-Lipschitz distance
	\begin{equation}
		\label{eq:BL-distance}
		d_{\mathrm{BL}}(\mu,\nu)
		:=
		\sup_{\|f\|_{\mathrm{BL}}\le1}
		\left|
		\int_{\overline D} f\,d(\mu-\nu)
		\right|.
	\end{equation}
	
Since \(\overline D\) is compact, the bounded-Lipschitz metric metrizes
narrow convergence on probability measures; see
\cite[Theorem~11.3.3]{dudley2002real}. The same conclusion extends
immediately to finite nonnegative measures with uniformly bounded total mass.
	
	\begin{corollary}[Quenched forward homogenization]
		\label{cor:quenched-forward-homogenization}
		Let \(\mu_0\) be a finite nonnegative Radon measure on
		\(\overline D\). For every \(T>0\) and every \(0<\beta<1/8\), there
		exists \(C=C(T,\beta)>0\), independent of \(\varepsilon\) and
		\(\mu_0\), such that
		\begin{equation}
			\label{eq:forward-BL-fixed}
			\sup_{0\le t\le T}
			\Esp^\alpha
			\left[
			d_{\mathrm{BL}}
			\bigl(
			\mu_t^{\varepsilon,\alpha},
			\mu_t^0
			\bigr)
			\right]
			\le
			C\,\mu_0(\overline D)\,
			\varepsilon^\beta,
		\end{equation}
		and
		\begin{equation}
			\label{eq:forward-BL-uniform}
			\Esp^\alpha
			\left[
			\sup_{0\le t\le T}
			d_{\mathrm{BL}}
			\bigl(
			\mu_t^{\varepsilon,\alpha},
			\mu_t^0
			\bigr)
			\right]
			\le
			C\,\mu_0(\overline D)\,
			\varepsilon^\beta.
		\end{equation}
		In particular,
		\[
		\sup_{0\le t\le T}
		d_{\mathrm{BL}}
		\bigl(
		\mu_t^{\varepsilon,\alpha},
		\mu_t^0
		\bigr)
		\longrightarrow0
		\qquad
		\text{in }\Prb^\alpha\text{-probability}.
		\]
		Hence the quenched forward measures converge narrowly to the homogenized
		forward evolution, uniformly on bounded time intervals in probability.
	\end{corollary}
	
	The result follows by duality from
	Theorem~\ref{thm:SE_homogenization_1}. Indeed, the pathwise estimate
	\eqref{eq:Y-difference-bound} is uniform over
	\(\|f\|_{\mathrm{BL}}\le1\), so that taking the supremum over this class
	before averaging introduces no additional loss. The proof is given in
	Appendix~\ref{app:forward-homogenization}.
	
	The improved rate for smooth compatible observables also transfers directly
	to the forward evolution.
	
	\begin{corollary}[Improved forward rate]
		\label{cor:improved-forward-homogenization}
		Assume that, for some \(c\ge0\),
		\[
		\kappa_i=c\rho_i,
		\qquad i\in S,
		\]
		and let \(f\in C^2(\overline D)\) satisfy
		\[
		\partial_{\mathbf n}f=0
		\qquad\text{on }\partial D.
		\]
		Then, for every \(T>0\),
		\begin{equation}
			\label{eq:forward-improved-fixed}
			\sup_{0\le t\le T}
			\Esp^\alpha
			\left|
			\int_{\overline D} f\,d
			\bigl(
			\mu_t^{\varepsilon,\alpha}-\mu_t^0
			\bigr)
			\right|
			\le
			C_f\,\mu_0(\overline D)\,
			\varepsilon^{1/4},
		\end{equation}
		and
		\begin{equation}
			\label{eq:forward-improved-uniform}
			\Esp^\alpha
			\left[
			\sup_{0\le t\le T}
			\left|
			\int_{\overline D} f\,d
			\bigl(
			\mu_t^{\varepsilon,\alpha}-\mu_t^0
			\bigr)
			\right|
			\right]
			\le
			C_f\,\mu_0(\overline D)\,
			\varepsilon^{1/4}
			\bigl(1+|\log\varepsilon|\bigr).
		\end{equation}
	\end{corollary}
	
	The proof is an immediate duality consequence of
	Theorem~\ref{thm:improved-smooth-observable}; see
	Appendix~\ref{app:forward-homogenization}.
	
	Taking \(f\equiv1\) gives a particularly natural physical quantity. For a
	nonnegative initial measure \(\mu_0\), set
	\[
	M_\varepsilon^\alpha(t)
	:=
	\mu_t^{\varepsilon,\alpha}(\overline D),
	\qquad
	M_0(t)
	:=
	\mu_t^0(\overline D).
	\]
	These quantities represent the total surviving mass at physical time \(t\).
	If \(\mu_0\) is a probability measure, they may equivalently be interpreted
	as the survival probabilities of particles whose initial positions are
	distributed according to \(\mu_0\).
	Since \(1\in C^2(\overline D)\) and
	\(\partial_{\mathbf n}1=0\), under the proportionality condition
	\(\kappa_i=c\rho_i\) the preceding corollary gives
	\[
	\sup_{0\le t\le T}
	\Esp^\alpha
	\left|
	M_\varepsilon^\alpha(t)-M_0(t)
	\right|
	\le
	C\,\mu_0(\overline D)\,\varepsilon^{1/4},
	\]
	and
	\[
	\Esp^\alpha
	\left[
	\sup_{0\le t\le T}
	\left|
	M_\varepsilon^\alpha(t)-M_0(t)
	\right|
	\right]
	\le
	C\,\mu_0(\overline D)\,
	\varepsilon^{1/4}
	\bigl(1+|\log\varepsilon|\bigr).
	\]
	
	\section{Numerical benchmarks for the homogenization estimates}
	\label{sec:numerical}
	
	We now illustrate numerically the finite-scale behavior predicted by the
	homogenization results. We compare the finite-switching model with its
	homogenized limit and examine the two errors appearing in
	Theorem~\ref{thm:SE_homogenization_1}. The numerical scheme and its
	consistency for fixed \(\varepsilon\) are described in
	Appendix~\ref{app:numerical-scheme}.
	
	We work on the interval \([-R,R]\). 
	In operational time, the reference process evolves according to the SDE
	\[
	dW_s^+
	=
	\sqrt{2D_{\mathrm{Glu}}}\,dB_s
	+dL_s^- -dL_s^+,
	\qquad
	L_s=L_s^-+L_s^+,
	\]
	with \(W_0^+=0\). The boundary environment is a two-state Markov chain with
	\[
	Q=
	\begin{pmatrix}
		-0.7&0.7\\
		0.3&-0.3
	\end{pmatrix}
	\,\mathrm{ms}^{-1},
	\qquad
	\pi=(0.3,0.7),
	\qquad
	\alpha_0\sim\pi,
	\]
	independent of the reflected Brownian motion.
	
	For \(\varepsilon>0\), let
	\[
	A_\varepsilon(s)
	=
	s+\int_0^s
	\rho(\alpha_{r/\varepsilon})\,dL_r,
	\qquad
	\gamma_\varepsilon=A_\varepsilon^{-1},
	\]
	and
	\[
	K_\varepsilon(t)
	=
	\int_0^{\gamma_\varepsilon(t)}
	\kappa(\alpha_{r/\varepsilon})\,dL_r.
	\]
	For an observable \(f\), write
	\[
	Y_\varepsilon^f(t)
	=
	f(W^+_{\gamma_\varepsilon(t)})
	e^{-K_\varepsilon(t)}.
	\]
	The corresponding homogenized quantities are
	\[
	A_0(s)=s+\overline\rho L_s,
	\qquad
	\gamma_0=A_0^{-1},
	\qquad
	K_0(t)=\overline\kappa L_{\gamma_0(t)},
	\]
	and
	\[
	Y_0^f(t)
	=
	f(W^+_{\gamma_0(t)})e^{-K_0(t)}.
	\]
	
	The homogenized backward equation is
	\begin{equation}
		\label{eq:numerical-homogenized-pde}
		\begin{cases}
			\partial_tu^0
			=
			D_{\mathrm{Glu}}\partial_{xx}u^0,
			&-R<x<R,\\[1mm]
			\partial_{\mathbf n}u^0
			+\overline\kappa u^0
			+\overline\rho\,\partial_tu^0
			=0,
			&x=\pm R,\\[1mm]
			u^0(0,x)=f(x).
		\end{cases}
	\end{equation}
	As before, the initial datum is understood in the mild sense.

	\subsection{Numerical parameters}
	\label{subsec:glutamate-interpretation}
	
	The numerical values are chosen on a scale motivated by extracellular
	glutamate clearance mechanism; this interpretation is used only to set representative
	length and time scales. We take
	\[
	D_{\mathrm{Glu}}
	=
	0.46\,\mu\mathrm m^2/\mathrm{ms},
	\qquad
	R=0.5\,\mu\mathrm m,
	\]
	with the diffusion coefficient taken from
	\cite{ZhengEtAl2017Diffusion}. Under our normalization of boundary local
	time,
	\[
	[\rho_i]=\mathrm{ms}/\mu\mathrm m,
	\qquad
	[\kappa_i]=\mu\mathrm m^{-1}.
	\]
	
	For the constant-coefficient model and the survival observable \(f=1\),
	write
	\[
	M_0(t)=u^0(t,0;1).
	\]
	Its mean lifetime is
	\begin{equation}
		\label{eq:numerical-mean-lifetime}
		\int_0^\infty M_0(t)\,dt
		=
		\frac{R^2}{2D_{\mathrm{Glu}}}
		+
		\frac{\overline\rho+R/D_{\mathrm{Glu}}}
		{\overline\kappa}.
	\end{equation}
	The principal decay rate is
	\[
	\lambda_1=D_{\mathrm{Glu}}k_1^2,
	\qquad
	k_1\tan(k_1R)
	=
	\overline\kappa
	-
	\overline\rho D_{\mathrm{Glu}}k_1^2.
	\]
	Using the control centroid \(4.1\pm0.5\,\mathrm{ms}\) from
	\cite{ScimemiEtAl2013} and a \(4\,\mathrm{ms}\) tail time motivated by
	\cite[Fig.~2D]{ArmbrusterDullaDiamond2020} gives
	\[
	\overline\rho
	=
	0.0735035\,\mathrm{ms}/\mu\mathrm m,
	\qquad
	\overline\kappa
	=
	0.303130\,\mu\mathrm m^{-1}.
	\]
	These values are used only as nominal parameters: the two time scales come
	from different settings and the inversion is ill-conditioned near the
	chosen values.
	
	We take the microscopic coefficients
	\[
	\rho_0=\kappa_0=0,
	\qquad
	\rho_1=\frac{\overline\rho}{0.7},
	\qquad
	\kappa_1=\frac{\overline\kappa}{0.7}.
	\]
	Their invariant averages are \(\overline\rho\) and \(\overline\kappa\), and
	\begin{equation}
		\label{eq:numerical-proportional-coefficients}
		\kappa_i
		=
		\frac{\overline\kappa}{\overline\rho}\rho_i,
		\qquad i\in\{0,1\}.
	\end{equation}

\subsection{Quantities compared}
\label{subsec:numerical-errors}

The annealed finite-switching observable is
\[
\overline u^\varepsilon(t,0;f)
=
\Esp^{W,\alpha}[Y_\varepsilon^f(t)],
\]
whereas
\[
u^0(t,0;f)=\Esp^W[Y_0^f(t)].
\]
This comparison concerns only one-time observables and does not require a
coupling.

To examine the quenched estimates, we use the common-reference coupling from the
proof of Theorem~\ref{thm:SE_homogenization_1}. On a physical observation
grid \(\mathcal T\subset[0,T]\), define
\begin{align}
	\mathcal E_{\varepsilon,\mathcal T}^{\mathrm{marg}}(f)
	&=
	\max_{t\in\mathcal T}
	\Esp^{W,\alpha}
	\left|
	Y_\varepsilon^f(t)-Y_0^f(t)
	\right|,
	\label{eq:numerical-marginal-error}\\
	\mathcal E_{\varepsilon,\mathcal T}^{\mathrm{sup}}(f)
	&=
	\Esp^{W,\alpha}
	\left[
	\max_{t\in\mathcal T}
	\left|
	Y_\varepsilon^f(t)-Y_0^f(t)
	\right|
	\right].
	\label{eq:numerical-uniform-error}
\end{align}
By conditional Jensen's inequality, these quantities dominate the
corresponding quenched errors on the observation grid. Indeed,
\begin{align}
	\max_{t\in\mathcal T}
	\Esp^\alpha
	\left|
	u_\alpha^\varepsilon(t,0;f)-u^0(t,0;f)
	\right|
	&\le
	\mathcal E_{\varepsilon,\mathcal T}^{\mathrm{marg}}(f),
	\label{eq:coupled-dominates-quenched-marginal}
	\\
	\Esp^\alpha
	\left[
	\max_{t\in\mathcal T}
	\left|
	u_\alpha^\varepsilon(t,0;f)-u^0(t,0;f)
	\right|
	\right]
	&\le
	\mathcal E_{\varepsilon,\mathcal T}^{\mathrm{sup}}(f).
	\label{eq:coupled-dominates-quenched-uniform}
\end{align}
Thus the coupled quantities are stronger pathwise diagnostics, rather than
the quantities defining the quenched homogenization errors themselves.

For a general bounded Lipschitz observable,
Theorem~\ref{thm:SE_homogenization_1} gives every rate
\(O(\varepsilon^\beta)\), \(0<\beta<1/8\), for the two quenched
errors on the left-hand sides above.

For the survival observable \(f=1\), however,
Theorem~\ref{thm:improved-smooth-observable} applies because
\(1\in C^2(\overline D)\), \(\partial_{\mathbf n}1=0\), and
\eqref{eq:numerical-proportional-coefficients} gives
\[
\kappa_i
=
\frac{\overline\kappa}{\overline\rho}\rho_i.
\]
Hence the corresponding quenched errors satisfy the sharper
\(O(\varepsilon^{1/4})\) fixed-time bound and
\(O(\varepsilon^{1/4}(1+|\log\varepsilon|))\) uniform-in-time bound.

The second quenched experiment below considers instead a generic bounded
Lipschitz observable and non-proportional microscopic coefficients, so that
only Theorem~\ref{thm:SE_homogenization_1} applies. This yields both fixed-time
and uniform-in-time quenched error bounds of order
\(O(\varepsilon^\beta)\) for every \(0<\beta<1/8\).

\subsection{Finite-switching and homogenized survival curves}

We first consider the survival observable \(f=1\), with the proportional
microscopic coefficients specified in
\eqref{eq:numerical-proportional-coefficients}. For
\[
\varepsilon\in\{1,1/2,1/4,1/8,1/16\},
\]
we simulate \(3500\) independent trajectories up to
\(T=12\,\mathrm{ms}\), using operational step
\[
\Delta s=2.5\times10^{-4}\,\mathrm{ms}
\]
and physical output spacing \(0.1\,\mathrm{ms}\).
The environment is initialized from \(\pi\) and advanced over each
operational step using the exact transition matrix
\(e^{Q\Delta s/\varepsilon}\). The reflected process, local time, clock,
and killing functional are discretized as in
Appendix~\ref{app:numerical-scheme}.

The homogenized equation
\eqref{eq:numerical-homogenized-pde} is solved independently by finite
differences on \(301\) spatial nodes with BDF time integration.
Figure~\ref{fig:calibrated-switching-survival} compares the resulting
annealed survival curves. Pointwise Monte Carlo intervals are computed as
the sample mean plus or minus \(1.96\) sample standard errors.

\begin{figure}[tb]
	\centering
	\includegraphics[width=0.7\textwidth]
	{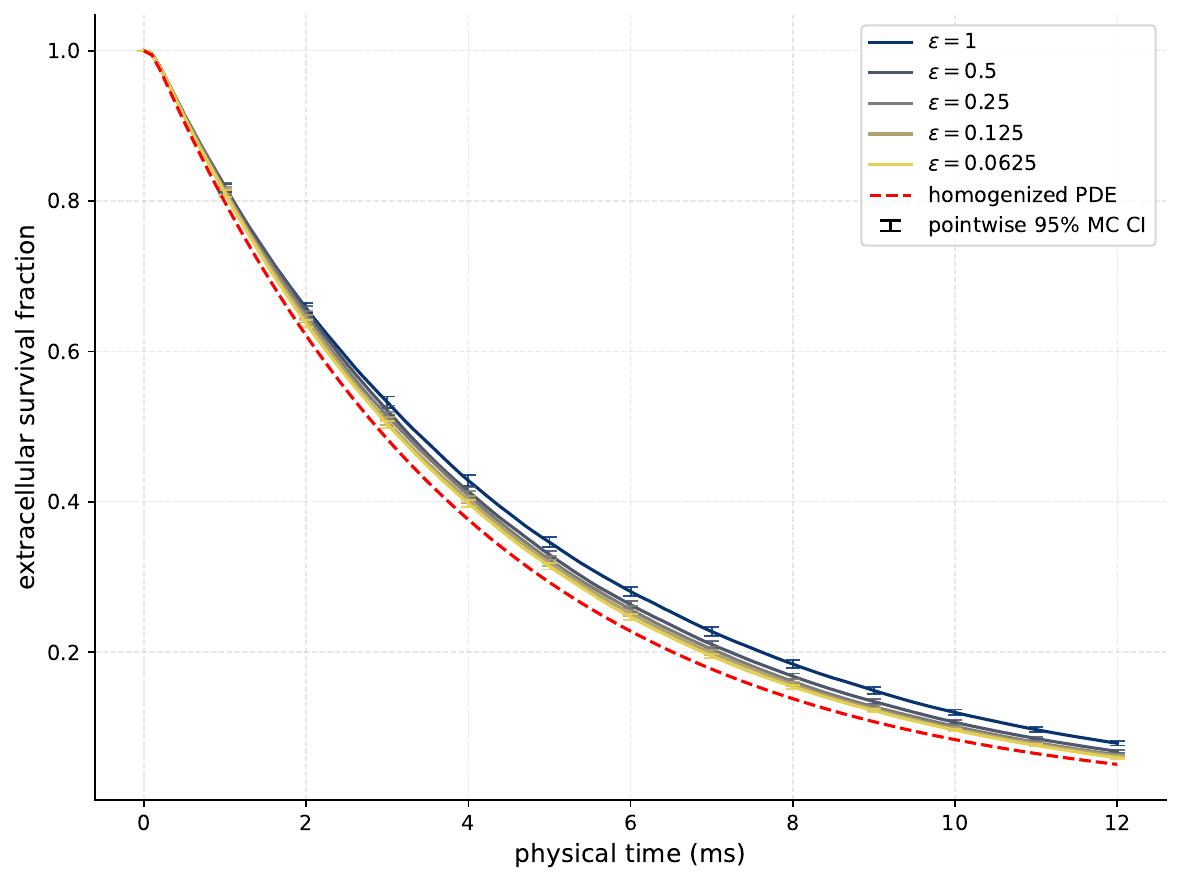}
	\caption{
		Annealed survival curves for the finite-switching model with proportional
		microscopic coefficients and for the homogenized Wentzell equation.
		Error bars show pointwise \(95\%\) Monte Carlo intervals. In the
		glutamate interpretation of
		Subsection~\ref{subsec:glutamate-interpretation}, survival is the
		fraction of molecules not yet irreversibly taken up.
	}
	\label{fig:calibrated-switching-survival}
\end{figure}

The finite-switching curves approach the homogenized evolution as the
environmental time scale becomes faster.

\subsection{Finite-scale behavior of the quenched errors}
\label{subsec:quenched-errors}

We also estimate the quenched errors themselves. For each fixed environment
\(\alpha\), the conditional difference
\[
u_\alpha^\varepsilon(t,0;f)-u^0(t,0;f)
=
\Esp^W
\left[
Y_\varepsilon^f(t)-Y_0^f(t)
\,\middle|\,\alpha
\right]
\]
is estimated by averaging over independent reflected Brownian trajectories.

\medskip
\noindent\emph{Survival observable and proportional coefficients.}
We first take \(f=1\) with the proportional microscopic coefficients
in~\eqref{eq:numerical-proportional-coefficients}. For each
\[
\varepsilon\in\{1/2,1/4,1/8,1/16,1/32\},
\]
we use \(640\) independent stationary environment paths and \(96\)
Brownian trajectories for each environment. The operational step is
\(2.5\times10^{-4}\,\mathrm{ms}\), and the output spacing is
\(0.1\,\mathrm{ms}\). Within each Brownian trajectory, the switching and
homogenized models use the same reflected increments and local-time
increments.

Uncertainty is assessed by a cluster bootstrap over the environment
paths, keeping together all Brownian trajectories associated with each
environment. We use \(5000\) bootstrap replicates; the same resamples are
used to obtain intervals for the error diagnostics and for the fitted
log--log slopes.

Over the range shown in Figure~\ref{fig:quenched-rate-test-structured}, the fitted
slopes are
\[
0.48
\qquad\text{and}\qquad
0.444,
\]
with \(95\%\) bootstrap intervals
\[
[0.451, 0.507]
\qquad\text{and}\qquad
[0.425, 0.463],
\]
for the marginal and uniform-in-time errors, respectively.

\begin{figure}[htbp]
	\centering
	
	\begin{subfigure}[t]{0.65\textwidth}
		\centering
		\includegraphics[width=\textwidth]{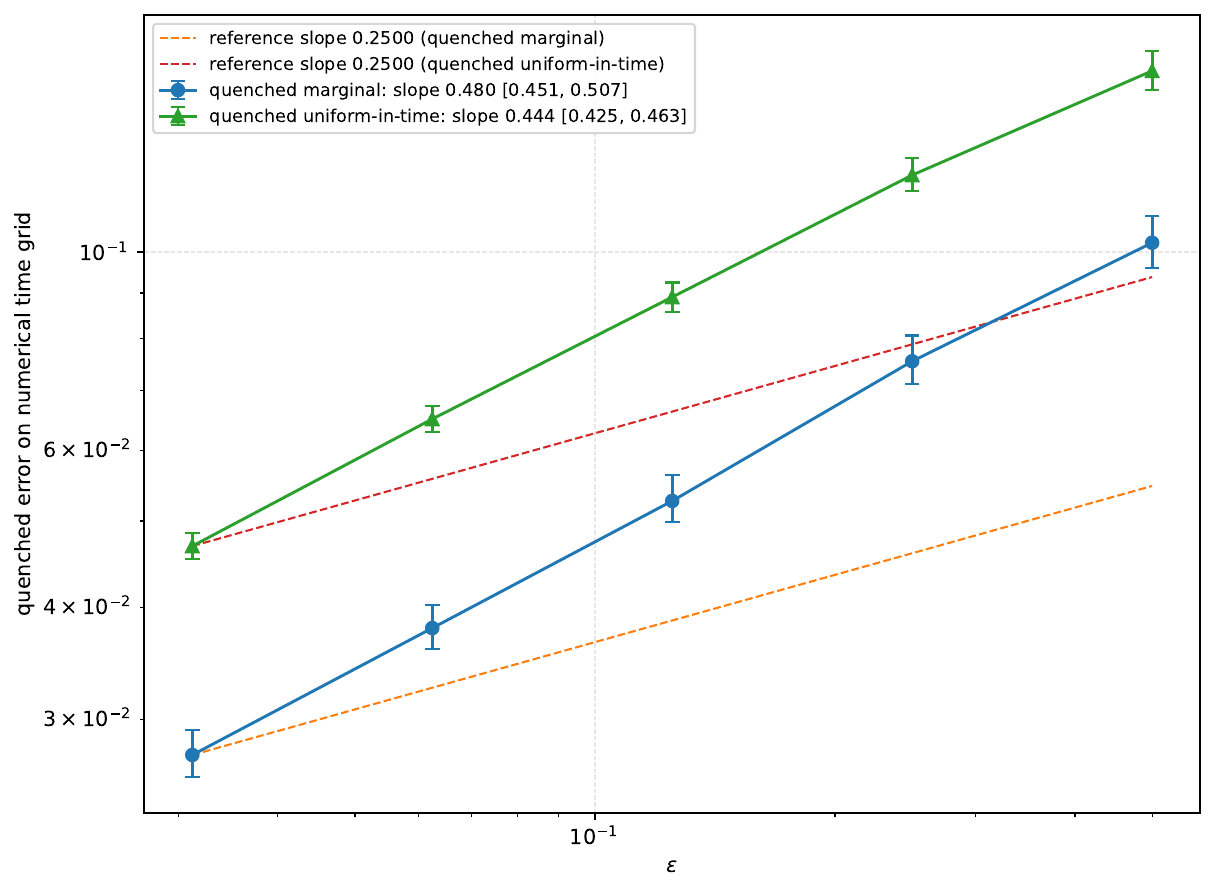}
		\caption{
			Survival observable \(f=1\) with proportional microscopic coefficients.
			The dashed reference lines have polynomial slope \(1/4\), corresponding
			to Theorem~\ref{thm:improved-smooth-observable}; the
			uniform-in-time bound also contains a logarithmic factor.
		}
		\label{fig:quenched-rate-test-structured}
	\end{subfigure}
	\\
	\begin{subfigure}[t]{0.65\textwidth}
		\centering
		\includegraphics[width=\textwidth]{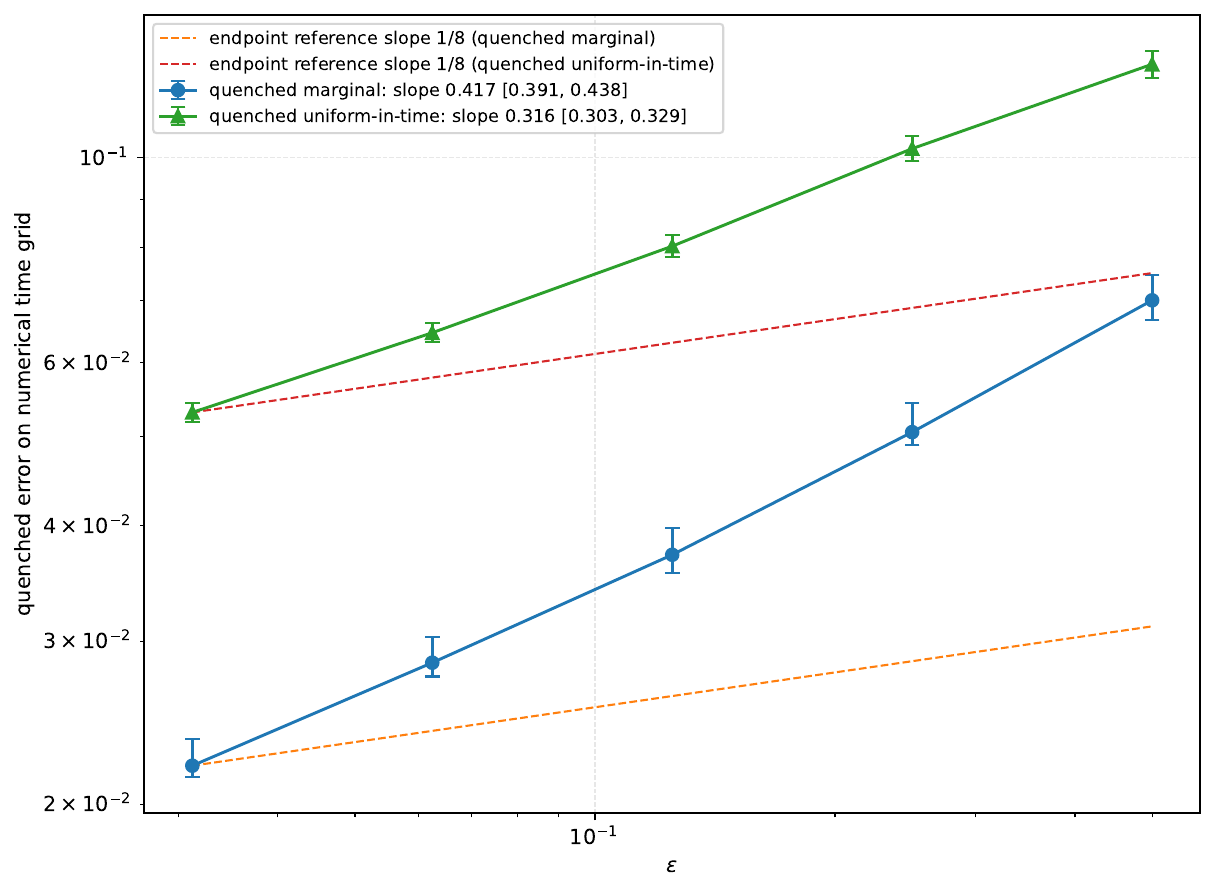}
		\caption{
			Generic Lipschitz observable
			\(f_{\mathrm{gen}}(x)=1+|x|/R\) with non-proportional microscopic
			coefficients. The dashed reference lines have slope \(1/8\), shown as
			the endpoint polynomial threshold associated with
			Theorem~\ref{thm:SE_homogenization_1}; the theorem itself gives every
			exponent strictly below \(1/8\).
		}
		\label{fig:quenched-rate-test-generic}
	\end{subfigure}
	
	\caption{
		Direct numerical estimates of the quenched marginal and
		uniform-in-time errors in two regimes. Panel~\textup{(a)} shows the
		structured case covered by
		Theorem~\ref{thm:improved-smooth-observable}; panel~\textup{(b)}
		shows a genuinely generic bounded Lipschitz case, for which only
		Theorem~\ref{thm:SE_homogenization_1} applies. Error bars show
		\(95\%\) bootstrap uncertainty intervals. In both panels, the fitted
		slopes describe only the finite range of switching scales shown and are
		not estimates of sharp asymptotic exponents.
	}
	\label{fig:quenched-rate-comparison}
\end{figure}

Over the simulated range, the observed decay is faster than the polynomial
scale appearing in
\eqref{eq:smooth-observable-fixed-rate}--%
\eqref{eq:smooth-observable-max-rate}. This is not in conflict with
the theory, which gives upper bounds rather than asymptotic equivalences.
The five switching scales cover only a limited range, and the logarithmic
factor in the uniform estimate cannot be separated numerically from an
effective power law over that range.

\medskip
\noindent\emph{Generic Lipschitz observable and non-proportional coefficients.}
To complement the preceding experiment, we next consider a setting in which
the assumptions leading to the improved rate of
Theorem~\ref{thm:improved-smooth-observable} are absent. We take
\[
f_{\mathrm{gen}}(x):=1+\frac{|x|}{R},
\qquad x\in[-R,R].
\]
This observable is bounded and Lipschitz, but it is not \(C^2\), and it does
not satisfy the Neumann compatibility condition; indeed,
\[
\partial_{\mathbf n}f_{\mathrm{gen}}(\pm R)=\frac1R.
\]

We retain the same invariant averages
\(\overline\rho,\overline\kappa\), but replace the microscopic coefficients by
\[
\rho_0^{\mathrm{gen}}=0,
\qquad
\rho_1^{\mathrm{gen}}=\frac{\overline\rho}{0.7},
\]
and
\[
\kappa_0^{\mathrm{gen}}=\frac12\overline\kappa,
\qquad
\kappa_1^{\mathrm{gen}}=\frac{17}{14}\overline\kappa.
\]
Since \(\pi=(0.3,0.7)\),
\[
\sum_{i=0}^1 \pi_i \rho_i^{\mathrm{gen}}=\overline\rho,
\qquad
\sum_{i=0}^1 \pi_i \kappa_i^{\mathrm{gen}}=\overline\kappa,
\]
so the homogenized coefficients are unchanged. On the other hand, the
microscopic killing and stickiness coefficients are not proportional, since
\(\rho_0^{\mathrm{gen}}=0\) while \(\kappa_0^{\mathrm{gen}}>0\).

Thus Theorem~\ref{thm:improved-smooth-observable} does not apply, and
the relevant theoretical comparison is the general estimate of
Theorem~\ref{thm:SE_homogenization_1}, which gives both fixed-time and
uniform-in-time quenched error bounds of order
\(O(\varepsilon^\beta)\) for every \(0<\beta<1/8\).

We use the same switching scales, number of stationary environment paths,
number of Brownian trajectories per environment, operational step, output
spacing, and cluster-bootstrap procedure as in the preceding quenched
experiment. Over the range shown in
Figure~\ref{fig:quenched-rate-test-generic}, the fitted slopes are
\[
0.417
\qquad\text{and}\qquad
0.316,
\]
with \(95\%\) bootstrap intervals
\[
[0.391,0.438]
\qquad\text{and}\qquad
[0.303,0.329],
\]
for the marginal and uniform-in-time errors, respectively.

The observed finite-range slopes need not coincide with \(1/8\).
Theorem~\ref{thm:SE_homogenization_1} provides upper bounds rather than
asymptotic equivalences, and does not assert that the threshold \(1/8\) is
optimal. The purpose of this experiment is instead to test the homogenization
behavior in a genuinely generic regime, without the smoothness, Neumann
compatibility, or proportionality assumptions responsible for the improved
rate above.

\medskip
\noindent\emph{Mesh sensitivity.}
To separate the operational-time discretization error from the
\(\varepsilon\)-dependence studied above, we repeated both quenched rate
experiments, for every reported value of \(\varepsilon\), with operational
steps \(h\), \(h/2\), and \(h/4\), where
\(h=2.5\times10^{-4}\,\mathrm{ms}\).
The three discretizations were coupled using the same exact continuous-time
environment path and nested Brownian increments. The fitted log--log slopes
were stable under mesh refinement, and the discrepancies between the two
finest meshes were small relative to the Monte Carlo uncertainty.
Thus, over the range of \(\varepsilon\) considered here, the observed slopes
are not driven by the operational-time mesh. This mesh-sensitivity study also
provides an empirical check of the approximation scheme described in
Appendix~\ref{app:numerical-scheme}, complementing the fixed-\(\varepsilon\)
consistency result proved there.

As an additional implementation check, we also repeated the simulation with
both regimes equal to the homogenized coefficients. In that case the coupled
discrepancy vanishes, as expected.
	
	\subsection{Reproducibility.}
	The code used for the numerical experiments is publicly available at
	\[
	\text{\url{https://gitlab.com/alexganagn/glutamate-clearance}}.
	\]
	The repository contains the routines used to generate the survival curves
	and the direct quenched-error experiment, together with the corresponding
	random seeds, simulation parameters, and automated tests.
	The survival-curve simulations use seeds \(700\)--\(704\), while the direct
	quenched experiment uses base seed \(91021\).
	Appendix~\ref{app:numerical-scheme} justifies the simulation approach and
	describes the discretization of reflection, boundary local time, the random
	clock, and the killing functional.

	\appendix
	
	\section{H\"older regularity of the process and its boundary local time}
	\label{app:Holder_continuity}
	
	For $\theta<1/2$ and $T>0$, recall the seminorms
	\begin{equation}
		\label{eq:holder-norms-app}
		[W^+]_{\theta,T}
		:=
		\sup_{0\le r<s\le T}
		\frac{|W^+_s-W^+_r|}{|s-r|^\theta},
		\qquad
		[L]_{\theta,T}
		:=
		\sup_{0\le r<s\le T}
		\frac{|L_s-L_r|}{|s-r|^\theta}.
	\end{equation}
	
	\begin{proposition}
		\label{prop:Holder}
		For every \(\theta\in(0,1/2)\), \(q\ge1\), and \(T>0\), there exists a
		constant \(C(D,\theta,q,T)>0\) such that
		\begin{equation}
			\label{eq:holder-moment-assumption}
			\sup_{x\in\overline D} \Esp_x\left[
			[W^+]_{\theta,T}^q + [L]_{\theta,T}^q
			\right]
			\le C(D,\theta,q,T).
		\end{equation}
	\end{proposition}
	
	Fix a smooth distance function \(r\in C^2(\overline D)\) which agrees with
	\(\dist(\cdot,\partial D)\) in a collar
	\[
	D_{r_0}:=\{x\in D:\dist(x,\partial D)<r_0\}
	\]
	of the boundary, for some \(r_0>0\); see
	\cite[Lemma~14.16]{gilbarg2001_elliptic_PDEs} and
	\cite[Corollary~1.2]{lieberman1985_regularized_distance}.
	Moreover,
	\[
	r>0 \quad\text{in }D,
	\qquad
	r=0 \quad\text{on }\partial D.
	\]
	Since \(\overline D\) is compact and \(r\in C^2(\overline D)\),
	\(\|\nabla r\|_\infty\) and \(\|\Delta r\|_\infty\) are finite.
	
	\begin{proof}
		\emph{Step 1: bound for \(W^{+}\).}
		We first prove that, for \(\theta<1/2\) and \(q\ge2\), there exists a
		constant \(C_{D,T,\theta,q}>0\), independent of the starting position \(x\),
		such that
		\begin{equation}
			\label{eq:Holder_bound_Wp}
			\sup_{x\in\overline D}
			\Esp_x\left[[W^+]_{\theta,T}^q\right]
			\le C_{D,T,\theta,q}<\infty.
		\end{equation}
		
		The reflected heat kernel satisfies the Gaussian bound
		\cite[Theorem~3.1 and Remark~3.11]{bass1991_some_potential_theory}:
		\begin{equation}
			\label{eq:neumann-heat-kernel-upper-bound}
			p^{W^+}(u,y,z)
			\le
			C^{(1)}_{D,T} u^{-d/2}
			\exp\left(
			-\frac{|z-y|^2}{C^{(2)}_D u}
			\right),
			\qquad
			0<u\le T,
			\quad y,z\in\overline D.
		\end{equation}
		
		The Markov property and
		\eqref{eq:neumann-heat-kernel-upper-bound} give, for
		\(0\le s<t\le T\) and \(a\ge2\),
		\begin{align}
			\Esp_x\left[
			|W_t^+-W_s^+|^a
			\right]
			&= \Esp_x\left[ \Esp \sqbraces{ 
				|W_t^+-W_s^+|^a \mid \mc F_{s} }
			\right]
			\\
			&=
			\Esp_x\left[
			\int_D
			|z-W_s^+|^a
			p^{W^+}(t-s,W_s^+,z)\,dz
			\right]
			\nonumber\\
			&\le
			C^{(1)}_{D,T}
			\Esp_x\left[
			\int_{\mathbb R^d}
			|\xi-W_s^+|^a
			(t-s)^{-d/2}
			\exp\left(
			-\frac{|\xi-W_s^+|^2}{C^{(2)}_D(t-s)}
			\right)\,d\xi
			\right]
			\nonumber\\
			&=
			C^{(1)}_{D,T}
			\Esp_x\left[
			\int_{\mathbb R^d}
			|\zeta|^a
			(t-s)^{-d/2}
			\exp\left(
			-\frac{|\zeta|^2}{C^{(2)}_D(t-s)}
			\right)\,d\zeta
			\right]
			\nonumber\\
			&=
			C^{(1)}_{D,T}
			\int_{\mathbb R^d}
			|\zeta|^a
			(t-s)^{-d/2}
			\exp\left(
			-\frac{|\zeta|^2}{C^{(2)}_D(t-s)}
			\right)\,d\zeta
			\nonumber\\
			&=
			C_{a,D,T}|t-s|^{a/2},
			\label{eq:reflected-bm-increment-moment}
		\end{align}
		where in the last step we used the change of variables
		\(\zeta=\sqrt{t-s}\,\eta\) and absorbed the resulting Gaussian moment
		into the constant \(C_{a,D,T}\).
		
		Choose now \(a>q\) so large that
		\[
		\theta < \frac12 - \frac1a.
		\]
		Kolmogorov's theorem in moment form
		\cite[Theorem~A.10]{friz2010_book}, with constants uniform in the starting
		point, gives
		\[
		\Esp_x\left[[W^+]_{\theta,T}^q\right]
		\le C_{D,T,\theta,q}<\infty,
		\]
		which proves \eqref{eq:Holder_bound_Wp}.
		
		We turn to the regularity of the regulator \(L\).
		
		\emph{Step 2: smooth distance local time.}
		Let \(R_t=r(W_t^+)\). We first relate \(L\) to the right local time of \(R\).
		With our normalization,
		\begin{equation}
			\label{eq:sloctime_limit}
			L_t = \lim_{\varepsilon \to 0} \frac{1}{2\varepsilon}
			\int_{0}^{t}
			\indic{\dist(W^+_s,\, \partial D)< \varepsilon} \vd s. 
		\end{equation}
		The reflected SDE, with outward normal, is
		\cite{bass2005_uniqueness_for_reflecting}
		\begin{equation}
			\rd W^+_t = \rd W_t - \mathbf n(W^+_t) \vd L_t,
		\end{equation}
		where \(W\) is a \(d\)-dimensional Brownian motion and \(\mathbf n\) is the
		outward unit normal vector on \(\partial D\).
		It\^o's formula gives
		\begin{align}
			\rd R_t
			&= \nabla r(W^+_t) \cdot \vd W^+_t
			+ \frac{1}{2} \sum_{i,j=1}^{d}
			\partial_i \partial_j r(W^+_t)
			\vd \langle (W^+)^i, (W^+)^j\rangle_t \\
			&= \nabla r(W^+_t) \cdot \vd W_t
			- \nabla r(W^+_t) \cdot \mbf n(W^+_t) \vd L_t
			+ \frac{1}{2} \Delta r(W^+_t) \vd t,
		\end{align}
		where we used that the finite-variation reflection term does not contribute
		to quadratic variation, and hence
		\[
		\rd\langle (W^+)^i,(W^+)^j\rangle_t
		=
		\delta_{ij}\,\rd t.
		\]
		On the boundary, \(r\) is the inward distance, hence
		\[
		\nabla r(x) = -\mbf n(x),
		\qquad
		\nabla r(x)\cdot \mbf n(x) = -1.
		\]
		Since \(dL_t\) is supported on
		\(\{W^+_t\in\partial D\}=\{R_t=0\}\), we obtain
		\begin{equation}
			R_t = R_0
			+ \int_{0}^{t} \nabla r(W^+_u) \cdot \vd W_u
			+ L_t
			+ \frac{1}{2} \int_{0}^{t} \Delta r(W^+_u) \vd u.
		\end{equation}
		Its quadratic variation is
		\begin{equation}
			\langle R \rangle_t
			=
			\sum_{i=1}^{d} \int_{0}^{t}
			\sqbraces{ \partial_{i} r(W^+_u) }^{2} \vd u.
		\end{equation}
		
		Since \(R\ge0\), Tanaka's formula
		\cite[Theorem~VI.1.2]{revuz1999Continuous_martingales} gives
		\begin{equation}
			R_t - R_0
			=
			\int_{0}^{t} \indicB{R_u>0} \vd R_u
			+ \frac{1}{2} \loct{R}{0}{t},
		\end{equation}
		where \(\loct{R}{0}{}\) is the right local time of \(R\).
		Comparison with the preceding decomposition yields
		\begin{equation}
			\int_{0}^{t}
			\indicB{R_u = 0} \nabla r(W^+_u) \cdot \vd W_u
			+ L_t
			+ \frac{1}{2} \int_{0}^{t}
			\indicB{R_u = 0} \Delta r(W^+_u) \vd u
			=
			\frac{1}{2} \loct{R}{0}{t}.
		\end{equation}
		
		Since \(\partial D\) has zero Lebesgue measure and \(W_u^+\) admits a
		transition density for every \(u>0\),
		\[
		\Esp_x\left[
		\int_0^T \indicB{W_u^+\in\partial D}\,\vd u
		\right]
		=
		\int_0^T
		\Pr_x(W_u^+\in\partial D)\,\vd u
		=0.
		\]
		Hence
		\[
		\int_0^T \indicB{R_u=0}\,\vd u=0
		\qquad\text{a.s.}
		\]
		The boundary-supported martingale therefore has zero quadratic variation
		and hence vanishes
		\cite[Proposition~IV.1.13]{revuz1999Continuous_martingales}:
		\begin{equation}
			\left\langle
			\int_{0}^{\cdot}
			\indicB{R_u = 0} \nabla r(W^+_u) \cdot \vd W_u
			\right\rangle_t
			=
			\int_0^t \indicB{R_u=0}\,
			|\nabla r(W^+_u)|^2\,\vd u
			=0.
		\end{equation}
		Likewise,
		\[
		\int_0^t \indicB{R_u=0}\Delta r(W^+_u)\,\vd u=0.
		\]
		Thus, with our normalization,
		\begin{equation}
			\label{eq:loctime_relation}
			L_t=\frac12\loct{R}{0}{t}.
		\end{equation}
		This factor \(1/2\) is precisely the convention fixed in
		Section~\ref{ssec:notation}; compare
		\cite[Section~VI.1]{revuz1999Continuous_martingales}.
		
		\medskip\noindent\emph{Step 3: bound for \(L\).}
		By Tanaka's formula applied to the nonnegative continuous semimartingale \(R\),
		and using the identity~\eqref{eq:loctime_relation}, for \(0\le s<t\le T\),
		\begin{align}
			L_t - L_s
			&=
			\braces{R_t - R_s}
			-
			\int_s^t \indicB{R_u>0}\,\vd R_u \\
			&=
			\braces{R_t - R_s}
			-
			\braces{M_t-M_s}
			-
			\braces{A_t-A_s},
		\end{align}
		where
		\[
		M_t := \int_{0}^{t}
		\indicB{R_u > 0} \nabla r(W^+_u) \cdot \vd W_u,
		\qquad
		A_t := \frac{1}{2} \int_{0}^{t}
		\indicB{R_u > 0} \Delta r(W^+_u) \vd u.
		\]
		Here we used that \(dL_u\) is supported on \(\{R_u=0\}\), so the reflection
		term vanishes after multiplication by \(\indicB{R_u>0}\).
		
		By convexity,
		\begin{equation}
			\Esp_x \sqbraces{|L_t - L_s|^{q}}
			\le
			3^{q-1}
			\braces{
				\Esp_x \sqbraces{|R_t - R_s|^{q}}
				+ \Esp_x \sqbraces{|M_t - M_s|^{q}}
				+ \Esp_x \sqbraces{|A_t - A_s|^{q}}
			}.
		\end{equation}
		
		Since \(r\) is Lipschitz on \(\overline D\),
		\begin{equation}
			|R_t - R_s|
			=
			|r(W^+_t) - r(W^+_s)|
			\le
			\mathrm{Lip}(r)\,|W^{+}_t - W^{+}_s|.
		\end{equation}
		Therefore, by Step~1,
		\begin{equation}
			\Esp_x \sqbraces{|R_t - R_s|^{q}}
			\le
			\mathrm{Lip}(r)^q C_{q,D,T} |t-s|^{q/2}.
		\end{equation}
		
		For the martingale term, BDG gives
		(see, e.g., \cite[Section~IV.4]{revuz1999Continuous_martingales})
		\begin{align}
			\Esp_x \sqbraces{|M_t - M_s|^{q}}
			&\le C_{q}
			\Esp_x \sqbraces{
				|\qv{M}_t - \qv{M}_s|^{q/2}} \\
			&= C_{q} \Esp_x \sqbraces{
				\abs{
					\int_{s}^{t}
					\indicB{R_u > 0}
					|\nabla r(W^+_u)|^{2}
					\vd u}^{q/2}} \\
			&\le C_q \xnorm{\nabla r}{\infty}^{q} |t-s|^{q/2}.
		\end{align}
		
		For the finite-variation term,
		\begin{align}
			\Esp_x \sqbraces{|A_t - A_s|^{q}}
			&\le
			\frac{1}{2^q}
			\xnorm{\Delta r}{\infty}^{q}
			|t-s|^{q} \\
			&\le
			\frac{1}{2^q}
			\xnorm{\Delta r}{\infty}^{q}
			T^{q/2}
			|t-s|^{q/2}.
		\end{align}
		
		Combining the three bounds gives
		\begin{equation}
			\Esp_x \sqbraces{ |L_t - L_s|^{q} }
			\le
			C_{q,D,T}|t-s|^{q/2},
			\qquad 0\le s<t\le T.
		\end{equation}
		The same argument applies with \(q\) replaced by any \(a\ge2\). Hence
		\begin{equation}
			\label{eq:local-time-increment-moment}
			\sup_{x\in\overline D}
			\Esp_x\sqbraces{|L_t-L_s|^a}
			\le
			C_{a,D,T}|t-s|^{a/2},
			\qquad 0\le s<t\le T.
		\end{equation}
		
		Now fix \(\theta<1/2\) and \(q\ge2\), and choose \(a>q\) large enough so that
		\[
		\theta < \frac12-\frac1a.
		\]
		Kolmogorov's theorem in moment form then gives
		\[
		\sup_{x\in\overline D}
		\Esp_x\left[[L]_{\theta,T}^q\right]
		\le C_{D,T,\theta,q}<\infty.
		\]
		
		It remains to treat \(1\le q<2\). By Jensen's inequality,
		\[
		\Esp_x[Z^q]
		=
		\Esp_x[(Z^2)^{q/2}]
		\le
		\Esp_x[Z^2]^{q/2}.
		\]
		Applying this with \(Z=[W^+]_{\theta,T}\) and
		\(Z=[L]_{\theta,T}\) extends the preceding \(q=2\) estimate to all
		\(1\le q<2\).
		This completes the proof.
	\end{proof}
	
	\section{Boundary local time and verification of the abstract assumptions}
	\label{app:averaging}
	
	In this section, we verify that the reflected Brownian boundary local time satisfies the
	abstract short-time hypothesis~\eqref{eq:abstract-A-first} for $\vartheta = 1/2 $. 
	Once this is done, Proposition~\ref{prop:LTA}
	is an immediate specialization of Section~\ref{sec:abstract-averaging}.
	
	\subsection{Preliminary estimates}
	
	The needed geometric input is the familiar square-root growth of boundary
	local time.
	
	\begin{lemma}[Short-time growth of boundary local time]
		\label{lem:short-time-growth-local-time}
		Assume the notations of Section~\ref{ssec:notation}.
		Then, for every \(T>0\), there exists a constant
		\(C_{D,T}<\infty\) such that
		\begin{equation}
			\sup_{x\in\overline D}\Esp_x[L_t]
			\le C_{D,T} t^{1/2},
			\qquad 0<t\le T.
		\end{equation}
		Consequently, for every \(c>0\), there exists
		\(C_{D,T,c}<\infty\) such that, for every
		\(\varepsilon\in(0,1]\),
		\[
		\sup_{\xi\in\partial D}
		\Esp_\xi\left[
		\int_0^T e^{-cu/\varepsilon}\,dL_u
		\right]
		\le C_{D,T,c}\varepsilon^{1/2}.
		\]
	\end{lemma}
	
	\begin{proof}
		Let \(p^{W^+}(t,x,y)\) be the reflected heat kernel. 
		With the Skorokhod-regulator normalization fixed in
		Section~\ref{ssec:notation}, the boundary local time
		\(L\) is the PCAF in Revuz correspondence with
		\(\frac12\sigma\); see, for example,
		\cite[Theorem~3.5]{Akhlil2012}.
		Therefore, for every nonnegative Borel
		function \(g\), the Revuz formula gives
		\begin{equation}
			\label{eq:Revuz_formula_II}
			\Esp_x\left[
			\int_0^t g(W_s^+)\,dL_s
			\right]
			=
			\frac12
			\int_0^t
			\int_{\partial D}
			g(\xi)\,
			p^{W^+}(s,x,\xi)\,
			d\sigma(\xi)\,ds ,
		\end{equation}
		see also \cite[proof of Theorem~5.3]{bass1991_some_potential_theory}. 
		Here, the heat kernel is taken with respect to Lebesgue measure, and \(\sigma\) denotes surface measure.
		
		Taking \(g\equiv1\) in \eqref{eq:Revuz_formula_II} gives
		\begin{equation}
			\label{eq:proof_Esp_ltime}
			\Esp_x[L_t]
			=
			\frac12
			\int_0^t\int_{\partial D}
			p^{W^+}(s,x,\xi)\,d\sigma(\xi)\,ds.
		\end{equation}
		
		The Gaussian heat-kernel bound
		\cite[Theorem~3.1 and Remark~3.11]{bass1991_some_potential_theory} gives, for every time-horizon \(T>0\),
		\begin{equation}
			\label{eq:proof_Gaussian_surface_bound}
			\int_{\partial D}p^{W^+}(s,x,\xi)\,d\sigma(\xi)
			\le
			C^{(1)}_{D,T}s^{-d/2}
			\int_{\partial D}
			\exp\left(-\frac{|x-\xi|^2}{C^{(2)}_{D}s}\right)\,d\sigma(\xi),
			\qquad 0<s\le T.
		\end{equation}
		The surface integral satisfies
		\[
		\sup_{x\in\overline D}
		\int_{\partial D}
		\exp\left(-\frac{|x-\xi|^2}{C^{(2)}_D s}\right)\,d\sigma(\xi)
		\le
		C_Ds^{(d-1)/2},
		\qquad s>0.
		\]
		Indeed, since \(\partial D\) is a compact \(C^2\) hypersurface, it admits a finite
		covering by \(C^1\) graph charts with uniformly bounded Jacobians. Hence, by
		the area formula,
		\begin{equation}
			\label{eq:surface_estimate}
			\sigma(\partial D\cap B(y,r))
			\le C_D r^{d-1},
			\qquad
			y\in\partial D,\quad 0<r\le r_0;
		\end{equation}
		see, e.g., \cite[\S 3.3.4.B]{evans2015measure}.
		If \(x\in\overline D\) and
		\(B(x,r)\cap\partial D\neq\emptyset\), choose
		\(y\in B(x,r)\cap\partial D\). Then
		\[
		B(x,r)\cap\partial D
		\subset B(y,2r)\cap\partial D.
		\]
		Thus \eqref{eq:surface_estimate} gives the same bound for
		\(x\in\overline D\) whenever \(2r\le r_0\). For larger \(r\), the bound
		follows after increasing \(C_D\), since \(\sigma(\partial D)<\infty\).
		Hence
		\[
		\sigma(\partial D\cap B(x,r))
		\le C_Dr^{d-1},
		\qquad x\in\overline D,\quad r>0.
		\]
		
		We decompose \(\partial D\) into the disjoint dyadic annuli 
		\[
		A_0:=\partial D\cap B(x,2\sqrt{s}),\qquad
		A_n:=\partial D\cap
		\bigl(B(x,2^{n+1}\sqrt{s})\setminus B(x,2^n\sqrt{s})\bigr),
		\quad n\ge1.
		\]
		Then
		\[
		\begin{aligned}
			\int_{\partial D}
			\exp\left(-\frac{|x-\xi|^2}{C^{(2)}_D s}\right)\,d\sigma(\xi)
			&\le
			\sigma(A_0)
			+
			\sum_{n\ge1}
			\exp\left(-\frac{2^{2n}}{C^{(2)}_D}\right)\sigma(A_n)\\
			&\le
			C_Ds^{(d-1)/2}
			+
			C_Ds^{(d-1)/2}
			\sum_{n\ge1}
			2^{(n+1)(d-1)}
			\exp\left(-\frac{2^{2n}}{C^{(2)}_D}\right)\\
			&\le
			C_Ds^{(d-1)/2}.
		\end{aligned}
		\]
		Substituting this into~\eqref{eq:proof_Gaussian_surface_bound} yields
		\begin{equation}
			\label{eq:boundary-kernel-integral}
			\sup_{x\in\overline D}
			\int_{\partial D}p^{W^+}(s,x,\xi)\,d\sigma(\xi)
			\le
			\widetilde C_{D,T}s^{-1/2},
			\qquad 0<s\le T.
		\end{equation}
		By~\eqref{eq:proof_Esp_ltime} and
		\eqref{eq:boundary-kernel-integral},
		\[
		\Esp_x[L_t]
		\le
		\frac{\widetilde C_{D,T}}{2}
		\int_0^t s^{-1/2}\,ds
		=
		\widetilde C_{D,T}t^{1/2},
		\qquad 0<t\le T.
		\]
		This proves the first estimate. 
		
		For the discounted bound, use the same Revuz
		formula with weight \(e^{-cu/\varepsilon}\):
		\[
		\Esp_\xi\left[
		\int_0^T e^{-cu/\varepsilon}\,dL_u
		\right]
		=
		\frac12
		\int_0^T
		e^{-cu/\varepsilon}
		\int_{\partial D}p^{W^+}(u,\xi,\eta)\,d\sigma(\eta)\,du.
		\]
		Therefore, by \eqref{eq:boundary-kernel-integral},
		\[
		\Esp_\xi\left[
		\int_0^T e^{-cu/\varepsilon}\,dL_u
		\right]
		\le
		\frac{\widetilde C_{D,T}}2
		\int_0^T e^{-cu/\varepsilon}u^{-1/2}\,du.
		\]
		By the change of variables \(v=cu/\varepsilon\),
		\[
		\int_0^T e^{-cu/\varepsilon}u^{-1/2}\,du
		=
		\left(\frac{\varepsilon}{c}\right)^{1/2}
		\int_0^{cT/\varepsilon} e^{-v}v^{-1/2}\,dv
		\le
		\left(\frac{\varepsilon}{c}\right)^{1/2}
		\int_0^\infty e^{-v}v^{-1/2}\,dv.
		\]
		Since
		\[
		\int_0^\infty e^{-v}v^{-1/2}\,dv
		=
		\Gamma(1/2)<\infty,
		\]
		setting
		\[
		C_{D,T,c}
		:=
		\frac{\widetilde C_{D,T}}2\,
		c^{-1/2}\Gamma(1/2),
		\]
		we obtain
		\[
		\Esp_\xi\left[
		\int_0^T e^{-cu/\varepsilon}\,dL_u
		\right]
		\le
		C_{D,T,c}\varepsilon^{1/2}.
		\]
		The estimate is uniform in \(\xi\).
		This completes the proof. 
	\end{proof}
	
\begin{proposition}[Boundary-local-time averaging]
	\label{prop:LTA}
	Under the assumptions of Section~\ref{ssec:notation}, let
	\(h:S\to\mathbb R\) be bounded and centered with respect to \(\pi\).
	Then, for every \(T>0\), the following estimates hold uniformly in
	\(x\in\overline D\) and \(i\in S\):
	\begin{enuroman}
		\item\label{item:LTA_marginal}
		For every \(\varepsilon\in(0,1]\) and every \([0,T]\)-valued random
		time \(\tau\) measurable with respect to the reflected Brownian path,
		\begin{equation}
			\label{eq:random-time-local-time-averaging}
			\sup_{\substack{x\in\overline D\\ i\in S}}
			\Esp_{x,i}\left[
			\left|
			\int_0^\tau h(\alpha_{s/\varepsilon})\,dL_s
			\right|^2
			\right]
			\le C_T(h)\varepsilon^{1/2};
		\end{equation}
		
		\item\label{item:LTA_sup}
		For every \(\varepsilon\in(0,1]\),
		\begin{equation}
			\label{eq:maximal-local-time-averaging-log}
			\sup_{\substack{x\in\overline D\\ i\in S}}
			\Esp_{x,i}\left[
			\sup_{0\le t\le T}
			\left|
			\int_0^t h(\alpha_{s/\varepsilon})\,dL_s
			\right|^2
			\right]
			\le
			C_{T,h}\varepsilon^{1/2}
			(1+|\log\varepsilon|)^2.
		\end{equation}
		Consequently, for every \(0<\zeta<1/2\), there exists
		\(C_{T,h,\zeta}<\infty\) such that
		\begin{equation}
			\label{eq:maximal-local-time-averaging}
			\sup_{\substack{x\in\overline D\\ i\in S}}
			\Esp_{x,i}\left[
			\sup_{0\le t\le T}
			\left|
			\int_0^t h(\alpha_{s/\varepsilon})\,dL_s
			\right|^2
			\right]
			\le C_{T,h,\zeta}\varepsilon^\zeta.
		\end{equation}
	\end{enuroman}
\end{proposition}
	
\begin{proof}
	By Lemma~\ref{lem:short-time-growth-local-time},
	\[
	\sup_{x\in\overline D}\Esp_x[L_t]
	\le C_{D,T}t^{1/2},
	\qquad 0<t\le T.
	\]
	Hence the hypothesis~\eqref{eq:abstract-A-first} holds with
	\(A=L\) and \(\vartheta=1/2\).
	
	The random-time estimate
	\eqref{eq:random-time-local-time-averaging}
	is therefore
	Theorem~\ref{thm:abstract-averaging}%
	\ref{item:abstract-random-time}, uniformly in
	\(x\in\overline D\) and \(i\in S\).
	Likewise, Theorem~\ref{thm:abstract-averaging}%
	\ref{item:abstract-max} gives
	\[
	\sup_{\substack{x\in\overline D\\ i\in S}}
	\Esp_{x,i}\left[
	\sup_{0\le t\le T}
	\left|
	\int_0^t h(\alpha_{s/\varepsilon})\,dL_s
	\right|^2
	\right]
	\le
	C_{T,h}\varepsilon^{1/2}
	(1+|\log\varepsilon|)^2.
	\]
	In particular, Corollary~\ref{cor:abstract-all-moments}
	yields
	\eqref{eq:maximal-local-time-averaging}
	for every \(0<\zeta<1/2\).
\end{proof}
	
	\section{Proofs of the forward formulations}
	\label{app:forward-proofs}
	
	\begin{proof}[Proof of Theorem~\ref{thm:forward_quenched}]
		Since \(S_t^\alpha\) is a contraction on \(C(\overline D)\), duality gives
		\[
		\begin{aligned}
			\|\mu_t^\alpha\|_{\mathrm{TV}}
			&=
			\sup_{\substack{f\in C(\overline D)\\ \|f\|_\infty\le1}}
			\left|
			\int_{\overline D}f\,d\mu_t^\alpha
			\right|
			\\
			&=
			\sup_{\substack{f\in C(\overline D)\\ \|f\|_\infty\le1}}
			\left|
			\int_{\overline D}S_t^\alpha f\,d\mu_0
			\right|
			\\
			&\le
			\|\mu_0\|_{\mathrm{TV}},
		\end{aligned}
		\]
		where the first identity is the norm characterization in the
		Riesz--Markov--Kakutani theorem; see, e.g.,
		\cite[Theorem~6.19]{rudin1987real}.
		If \(\kappa\equiv0\), then \(S_t^\alpha1=1\), and therefore
		\[
		\mu_t^\alpha(\overline D)
		=
		\mu_0(\overline D).
		\]

		For every $f\in C(\overline D)$ and every fixed $x$,
		continuity of the time-changed spatial path and of the killing
		functional, together with bounded convergence, gives
		$S_t^\alpha f(x)\to S_s^\alpha f(x)$ as $t\to s$.
		Since $|S_t^\alpha f(x)|\le\|f\|_\infty$, dominated convergence
		with respect to the finite measure $|\mu_0|$ yields
		\[
		\int_{\overline D} f\,d\mu_t^\alpha
		=
		\int_{\overline D}S_t^\alpha f\,d\mu_0
		\longrightarrow
		\int_{\overline D}S_s^\alpha f\,d\mu_0
		=
		\int_{\overline D}f\,d\mu_s^\alpha.
		\]
		This proves narrow continuity.

		Integrating \eqref{eq:quenched-backward-dynkin} with respect to
		\(\mu_0(dx)\) gives \eqref{eq:quenched-forward-dynkin}.
		
		If \(f\in\mathcal D_\alpha\), then
		\eqref{eq:quenched-backward-mild} gives
		\[
		\begin{aligned}
			\int_{\overline D}f\,d\mu_t^\alpha
			-
			\int_{\overline D}f\,d\mu_0
			&=
			\int_{\overline D}
			(S_t^\alpha f-f)\,d\mu_0
			\\
			&=
			\frac12
			\int_0^t
			\int_{\overline D}
			S_s^\alpha(\Delta f)(x)\,\mu_0(dx)\,ds
			\\
			&=
			\frac12
			\int_0^t
			\int_{\overline D}
			\Delta f(y)\,\mu_s^\alpha(dy)\,ds.
		\end{aligned}
		\]
		This proves \eqref{eq:weak-forward-equation} and
		\eqref{eq:weak-forward-derivative}. 
		
		For \eqref{eq:weak-forward-time-dependent}, fix $\varphi$ as in the statement
		and set $g(t):=\int_{\overline D}\varphi(t,x)\,\mu_t^\alpha(dx)$. Writing, for
		$h>0$,
		\[
		\frac{g(t+h)-g(t)}{h}
		=
		\int_{\overline D}\frac{\varphi(t+h,\cdot)-\varphi(t,\cdot)}{h}\,d\mu_{t+h}^\alpha
		+
		\frac{1}{h}\left(
		\int_{\overline D}\varphi(t,\cdot)\,d\mu_{t+h}^\alpha
		-
		\int_{\overline D}\varphi(t,\cdot)\,d\mu_t^\alpha
		\right),
		\]
		the first term converges to $\int_{\overline D}\partial_t\varphi(t,\cdot)\,d\mu_t^\alpha$
		by $\varphi\in C^1([0,T];C(\overline D))$, the uniform boundedness
		$\|\mu_t^\alpha\|_{\mathrm{TV}}\le\|\mu_0\|_{\mathrm{TV}}$, and narrow
		continuity; the second term converges to
		$\tfrac12\int_{\overline D}\Delta\varphi(t,\cdot)\,d\mu_t^\alpha$ by
		\eqref{eq:weak-forward-derivative}, since $\varphi(t,\cdot)\in\mathcal D_\alpha$.
		Hence $g$ is continuously differentiable with
		$g'(t)=\int_{\overline D}(\partial_t\varphi+\tfrac12\Delta\varphi)\,d\mu_t^\alpha$,
		and integrating from $0$ to $T$ gives \eqref{eq:weak-forward-time-dependent}.
	\end{proof}

\section{Proofs of the forward homogenization results}
\label{app:forward-homogenization}

\begin{proof}[Proof of
	Corollary~\ref{cor:quenched-forward-homogenization}]
	Fix \(T>0\) and \(0<\beta<1/8\), and set
	\(m_0:=\mu_0(\overline D)\).
	We write \(\Esp_{x,\pi}\) for joint expectation over the reflected
	Brownian motion started from \(x\) and the environment started from its
	stationary distribution \(\pi\).
	
	Under the common-reference coupling used in
	Section~\ref{sec:homogenization}, in which the switching and homogenized
	models are built from the same reflected Brownian motion, put
	\[
	R_\varepsilon(t)
	:=
	\left|
	W^+_{\gamma_\varepsilon(t)}
	-
	W^+_{\gamma_0(t)}
	\right|
	+
	\left|
	K_\varepsilon(t)-K_0(t)
	\right|.
	\]
	The uniform spatial estimate
	\eqref{eq:uniform-process-rate}, with \(p=1\), and the maximal
	killing estimate obtained in the proof of
	Theorem~\ref{thm:SE_homogenization_1} give
	\begin{equation}
		\label{eq:forward-proof-coupling-bound}
		\sup_{x\in\overline D}
		\Esp_{x,\pi}
		\left[
		\sup_{0\le t\le T}R_\varepsilon(t)
		\right]
		\le C\varepsilon^\beta.
	\end{equation}
	The constant is uniform in \(x\): the averaging estimates follow from
	Theorem~\ref{thm:abstract-averaging} and
	Lemma~\ref{lem:short-time-growth-local-time}, uniformly in the starting
	point, while Proposition~\ref{prop:Holder} provides the required uniform
	H\"older-moment bounds. Thus the spatial and killing estimates used in
	\eqref{eq:forward-proof-coupling-bound} have constants depending only on
	\(D,T,Q,\rho,\kappa,\beta\), and not on \(x,\varepsilon\), or \(\mu_0\).
	
	For a bounded Lipschitz function \(f\), the duality definitions
	\eqref{eq:forward-epsilon-duality} and
	\eqref{eq:forward-zero-duality} give
	\[
	\int_{\overline D}
	f\,d\bigl(\mu_t^{\varepsilon,\alpha}-\mu_t^0\bigr)
	=
	\int_{\overline D}
	\bigl(
	u_\alpha^\varepsilon(t,x;f)-u^0(t,x;f)
	\bigr)\,\mu_0(dx).
	\]
	If \(\|f\|_{\mathrm{BL}}\le1\), then
	\(\|f\|_\infty\le1\) and \(\operatorname{Lip}(f)\le1\).
	Consequently, \eqref{eq:Y-difference-bound} implies, for every fixed
	environment,
	\[
	\left|
	u_\alpha^\varepsilon(t,x;f)-u^0(t,x;f)
	\right|
	\le
	\Esp_x^W\!\left[R_\varepsilon(t)\right].
	\]
	Since \(\mu_0\) is nonnegative, it follows that
	\[
	\left|
	\int_{\overline D}
	f\,d\bigl(\mu_t^{\varepsilon,\alpha}-\mu_t^0\bigr)
	\right|
	\le
	\int_{\overline D}
	\Esp_x^W\!\left[R_\varepsilon(t)\right]\,\mu_0(dx).
	\]
	The right-hand side does not depend on \(f\). Taking the supremum over
	\(\|f\|_{\mathrm{BL}}\le1\) therefore yields
	\begin{equation}
		\label{eq:forward-proof-BL-bound}
		d_{\mathrm{BL}}
		\bigl(\mu_t^{\varepsilon,\alpha},\mu_t^0\bigr)
		\le
		\int_{\overline D}
		\Esp_x^W\!\left[R_\varepsilon(t)\right]\,\mu_0(dx).
	\end{equation}
	
The suprema used here are measurable. Indeed, the bounded-Lipschitz unit
ball is uniformly bounded and equicontinuous. By the Arzel\`a--Ascoli
theorem, its closure in \(C(\overline D)\) is compact; see
\cite[Theorem~4.43]{folland1999real}. Moreover, the unit ball is closed in
the uniform norm. Indeed, if \(f_n\to f\) uniformly and
\(\|f_n\|_{\mathrm{BL}}\le1\), then
\[
\|f\|_\infty=\lim_{n\to\infty}\|f_n\|_\infty
\]
and
\[
\operatorname{Lip}(f)
\le
\liminf_{n\to\infty}\operatorname{Lip}(f_n).
\]
Consequently,
\[
\|f\|_{\mathrm{BL}}
=
\|f\|_\infty+\operatorname{Lip}(f)
\le
\liminf_{n\to\infty}
\bigl(
\|f_n\|_\infty+\operatorname{Lip}(f_n)
\bigr)
\le1.
\]
Thus the bounded-Lipschitz unit ball is compact, hence separable, in the
uniform norm, and therefore admits a countable dense subset.

Moreover, for every \(f\in C(\overline D)\), the map
\[
t\longmapsto
\int_{\overline D} f\,d\mu_t^{\varepsilon,\alpha}
\]
is continuous, and similarly for \(\mu_t^0\). Hence the suprema over test
functions and over time may be taken over countable dense subsets.
	
	Taking the supremum in time in
	\eqref{eq:forward-proof-BL-bound}, then averaging over the environment
	and applying Tonelli's theorem, gives
	\[
	\begin{aligned}
		&\Esp^\alpha
		\left[
		\sup_{0\le t\le T}
		d_{\mathrm{BL}}
		\bigl(\mu_t^{\varepsilon,\alpha},\mu_t^0\bigr)
		\right]
		\\
		&\qquad\le
		\int_{\overline D}
		\Esp^\alpha\Esp_x^W
		\left[
		\sup_{0\le t\le T}R_\varepsilon(t)
		\right]\,\mu_0(dx)
		\\
		&\qquad=
		\int_{\overline D}
		\Esp_{x,\pi}
		\left[
		\sup_{0\le t\le T}R_\varepsilon(t)
		\right]\,\mu_0(dx)
		\\
		&\qquad\le
		C m_0\varepsilon^\beta,
	\end{aligned}
	\]
	where the last inequality follows from
	\eqref{eq:forward-proof-coupling-bound}.
	This proves \eqref{eq:forward-BL-uniform}.
	The fixed-time estimate \eqref{eq:forward-BL-fixed} follows from
	\[
	\sup_{0\le t\le T}
	\Esp^\alpha
	d_{\mathrm{BL}}
	\bigl(\mu_t^{\varepsilon,\alpha},\mu_t^0\bigr)
	\le
	\Esp^\alpha
	\left[
	\sup_{0\le t\le T}
	d_{\mathrm{BL}}
	\bigl(\mu_t^{\varepsilon,\alpha},\mu_t^0\bigr)
	\right].
	\]
	
	Finally, for every \(\eta>0\), Markov's inequality gives
	\[
	\Prb^\alpha
	\left(
	\sup_{0\le t\le T}
	d_{\mathrm{BL}}
	\bigl(\mu_t^{\varepsilon,\alpha},\mu_t^0\bigr)
	>\eta
	\right)
	\le
	\frac{C m_0}{\eta}\varepsilon^\beta
	\longrightarrow0.
	\]
	Both forward measures have mass at most \(m_0\).
	Since \(\overline D\) is compact, the bounded-Lipschitz distance
	metrizes narrow convergence on this set of measures.
	This proves the asserted uniform-in-time convergence in probability.
\end{proof}

\begin{proof}[Proof of
	Corollary~\ref{cor:improved-forward-homogenization}]
	Fix \(T>0\), and set
	\[
	a_f
	:=
	\left\|cf+\frac12\Delta f\right\|_\infty.
	\]
	As in the proof of
	Theorem~\ref{thm:improved-smooth-observable}, the proportionality condition
	\(\kappa_i=c\rho_i\) yields
	\[
	K_\varepsilon(t)=c\bigl(t-\gamma_\varepsilon(t)\bigr),
	\qquad
	K_0(t)=c\bigl(t-\gamma_0(t)\bigr).
	\]
	For a fixed environment, apply It\^o's formula to
	\(F(s,x)=e^{cs}f(x)\) at the bounded stopping times
	\(\gamma_\varepsilon(t)\) and \(\gamma_0(t)\).
	The Neumann condition removes the boundary local-time term.
	Taking expectations and subtracting the two identities yields
	\[
	u_\alpha^\varepsilon(t,x;f)-u^0(t,x;f)
	=
	\Esp_x^W
	\left[
	\int_{\gamma_0(t)}^{\gamma_\varepsilon(t)}
	e^{-c(t-s)}
	\left(cf+\frac12\Delta f\right)(W_s^+)\,ds
	\right],
	\]
	where the integral is understood in the oriented sense.
	
	Both inverse clocks are at most \(t\), and \(c\ge0\).
	Thus \(e^{-c(t-s)}\le1\) on the interval between them, and
	\begin{equation}
		\label{eq:forward-proof-smooth-clock-bound}
		\left|
		u_\alpha^\varepsilon(t,x;f)-u^0(t,x;f)
		\right|
		\le
		a_f\Esp_x^W
		\left|
		\gamma_\varepsilon(t)-\gamma_0(t)
		\right|.
	\end{equation}
	Using duality and the nonnegativity of \(\mu_0\), we obtain
	\begin{equation}
		\label{eq:forward-proof-smooth-duality-bound}
		\left|
		\int_{\overline D}
		f\,d\bigl(\mu_t^{\varepsilon,\alpha}-\mu_t^0\bigr)
		\right|
		\le
		a_f\int_{\overline D}
		\Esp_x^W
		\left|
		\gamma_\varepsilon(t)-\gamma_0(t)
		\right|\,\mu_0(dx).
	\end{equation}
	
	The fixed-time clock estimate
	\eqref{eq:fixed-time-gamma-L2-rate} holds uniformly in the starting
	point, by
	Theorem~\ref{thm:abstract-averaging} and
	Lemma~\ref{lem:short-time-growth-local-time}.
	Consequently, Tonelli's theorem and Lyapunov's inequality give
	\[
	\begin{aligned}
		&\sup_{0\le t\le T}
		\Esp^\alpha
		\left|
		\int_{\overline D}
		f\,d\bigl(\mu_t^{\varepsilon,\alpha}-\mu_t^0\bigr)
		\right|
		\\
		&\qquad\le
		a_f\int_{\overline D}
		\sup_{0\le t\le T}
		\Esp_{x,\pi}
		\left|
		\gamma_\varepsilon(t)-\gamma_0(t)
		\right|\,\mu_0(dx)
		\\
		&\qquad\le
		a_f\int_{\overline D}
		\sup_{0\le t\le T}
		\left(
		\Esp_{x,\pi}
		\left|
		\gamma_\varepsilon(t)-\gamma_0(t)
		\right|^2
		\right)^{1/2}\,\mu_0(dx)
		\\
		&\qquad\le
		C a_f\,\mu_0(\overline D)\,\varepsilon^{1/4}.
	\end{aligned}
	\]
	This proves \eqref{eq:forward-improved-fixed}.
	
	For the uniform estimate, the inverse-clock comparison
	\eqref{eq:fixed-time-inverse-bound} gives
	\[
	\sup_{0\le t\le T}
	\left|
	\gamma_\varepsilon(t)-\gamma_0(t)
	\right|
	\le
	\sup_{0\le r\le T}
	\left|
	\int_0^r
	\bigl(\rho(\alpha_{s/\varepsilon})-\overline\rho\bigr)\,dL_s
	\right|.
	\]
	Applying Theorem~\ref{thm:abstract-averaging} with
	\(A=L\), \(h=\rho-\overline\rho\), and \(\vartheta=1/2\), we obtain
	\[
	\sup_{x\in\overline D}
	\Esp_{x,\pi}
	\left[
	\sup_{0\le t\le T}
	\left|
	\gamma_\varepsilon(t)-\gamma_0(t)
	\right|^2
	\right]
	\le
	C\varepsilon^{1/2}
	\bigl(1+|\log\varepsilon|\bigr)^2.
	\]
	Taking the supremum in time in
	\eqref{eq:forward-proof-smooth-duality-bound}, then using again Tonelli's
	theorem and Lyapunov's inequality, gives
	\[
	\begin{aligned}
		&\Esp^\alpha
		\left[
		\sup_{0\le t\le T}
		\left|
		\int_{\overline D}
		f\,d\bigl(\mu_t^{\varepsilon,\alpha}-\mu_t^0\bigr)
		\right|
		\right]
		\\
		&\qquad\le
		a_f\int_{\overline D}
		\Esp_{x,\pi}
		\left[
		\sup_{0\le t\le T}
		\left|
		\gamma_\varepsilon(t)-\gamma_0(t)
		\right|
		\right]\,\mu_0(dx)
		\\
		&\qquad\le
		a_f\int_{\overline D}
		\left(
		\Esp_{x,\pi}
		\left[
		\sup_{0\le t\le T}
		\left|
		\gamma_\varepsilon(t)-\gamma_0(t)
		\right|^2
		\right]
		\right)^{1/2}\,\mu_0(dx)
		\\
		&\qquad\le
		C a_f\,\mu_0(\overline D)\,
		\varepsilon^{1/4}\bigl(1+|\log\varepsilon|\bigr).
	\end{aligned}
	\]
	Absorbing \(a_f\) into the constant \(C_f\) proves
	\eqref{eq:forward-improved-uniform}.
	Taking \(f\equiv1\) gives the stated bounds for the total surviving mass.
\end{proof}
	
	\section{Numerical approximation of reflection, time change, and killing}
	\label{app:numerical-scheme}
	
	In order to approximate by Monte Carlo the quenched Feynman--Kac
	evolution in \eqref{eq:quenched-semigroup}, its annealed averages, and
	the homogenized evolution in \eqref{eq:limit-semigroup}, one must
	simulate jointly the reflected Brownian motion and its boundary local
	time, the switching environment, the corresponding clocks and their
	inverses, and the associated Feynman--Kac functionals.
	We define such a numerical scheme below and prove its consistency.
	Quantitative discretization rates are beyond the scope of the present
	paper and are left for future work.
	
	Several approximation schemes for sticky and reflected diffusions are
	available in the literature; see, for example,
	\cite{amir1991Sticky_Brownian,anagnostakis2023General_diffusion,ankirchner2020A_functional,bou-rabee2020Sticky_Brownian,jiang2026_a_modified_Euler,meier2019Markov_Chain,sharma2025_weak_approximation}. Here we require, in addition to an approximation of the
	reflected path and its boundary local time, a coupled approximation of the
	local-time clock, its inverse, and the associated killing functional.
	
	\subsection{Approximation of reflection and boundary local time.}
	\label{subsec:numerical-reflection}
	
	Let \(Z\) be Brownian motion with generator \(D\partial_{xx}\),
	normally reflected on \([-R,R]\). We write its Skorokhod decomposition as
	\begin{equation}
		\label{eq:numerical-skorokhod-decomposition}
		Z_s
		=
		x+\sqrt{2D}\,W_s+L_s^- -L_s^+,
	\end{equation}
	where \(L^-\) and \(L^+\) are the boundary regulators at
	\(-R\) and \(R\), respectively. Throughout this section we use the
	normalization fixed by \eqref{eq:numerical-skorokhod-decomposition} and set
	\begin{equation}
		\label{eq:numerical-total-local-time}
		L:=L^-+L^+.
	\end{equation}
	
	Let \(s_k=kh\) be an operational-time mesh. Starting from
	\(Z_0^h=x\) and \(L_0^{\pm,h}=0\), we consider the recursive
	projection scheme
	\begin{equation}
		\label{eq:numerical-projection-scheme}
		\begin{aligned}
			\widetilde Z_{k+1}^h
			&=
			Z_k^h+\sqrt{2D}\,
			(W_{s_{k+1}}-W_{s_k}),
			\\
			\Delta L_{k+1}^{-,h}
			&=
			(-R-\widetilde Z_{k+1}^h)^+,
			&
			\Delta L_{k+1}^{+,h}
			&=
			(\widetilde Z_{k+1}^h-R)^+,
			\\
			Z_{k+1}^h
			&=
			\widetilde Z_{k+1}^h
			+\Delta L_{k+1}^{-,h}
			-\Delta L_{k+1}^{+,h}.
		\end{aligned}
	\end{equation}
	Thus \(Z_{k+1}^h\) is simply the projection of the Gaussian proposal
	\(\widetilde Z_{k+1}^h\) onto \([-R,R]\). 
	The accumulated corrections at the two endpoints are
	\begin{equation}
		\label{eq:numerical-discrete-regulators}
		L_{s_k}^{\pm,h}
		=
		\sum_{j=1}^k\Delta L_j^{\pm,h}.
	\end{equation}
	These will serve as approximations of the boundary regulators
	\(L^-\) and \(L^+\), respectively.
	We extend \(Z^h\) and \(L^{\pm,h}\) as piecewise constant processes on
	each interval \([s_k,s_{k+1})\).
	
	The recursion \eqref{eq:numerical-projection-scheme} is the
	one-dimensional, constant-coefficient version of the recursive projection
	scheme studied in \cite{slominski2001euler}; see in particular
	equation~(3.1) and Theorems~3.1 and~3.2 therein.
	
	It is also useful to view the scheme directly through the interval
	Skorokhod map. Define
	\begin{equation}
		\label{eq:numerical-driving-paths}
		\xi_s^h
		=
		x+\sqrt{2D}\,W_{h\lfloor s/h\rfloor},
		\qquad
		\xi_s
		=
		x+\sqrt{2D}\,W_s.
	\end{equation}
	Let \(\Gamma_{[-R,R]}\) denote the two-sided Skorokhod map on
	\([-R,R]\). After translating \([-R,R]\) to \([0,2R]\), we have
	\[
	\Gamma_{[-R,R]}(\psi)
	=
	\Gamma_{0,2R}(\psi+R)-R,
	\]
	where the explicit representation of
	\cite[Remark~1.5]{kruk2007explicit} is
	\begin{equation}
		\label{eq:explicit-interval-skorokhod-map}
		\Gamma_{0,a}(\psi)(t)
		=
		\psi(t)
		-
		\left[
		\left(
		(\psi(0)-a)^+
		\wedge
		\inf_{0\le u\le t}\psi(u)
		\right)
		\vee
		\sup_{0\le s\le t}
		\left(
		(\psi(s)-a)
		\wedge
		\inf_{s\le u\le t}\psi(u)
		\right)
		\right].
	\end{equation}
	Thus
	\[
	Z^h=\Gamma_{[-R,R]}(\xi^h),
	\qquad
	Z=\Gamma_{[-R,R]}(\xi).
	\]
	
	Fix \(S>0\). By the uniform-norm Lipschitz estimate of
	\cite[Corollary~1.6]{kruk2007explicit}, there exists a constant
	\(C_\Gamma>0\), independent of \(S\), such that
	\begin{equation}
		\label{eq:numerical-skorokhod-lipschitz}
		\sup_{0\le s\le S}|Z_s^h-Z_s|
		\le
		C_\Gamma
		\sup_{0\le s\le S}|\xi_s^h-\xi_s|.
	\end{equation}
	Since Brownian paths are uniformly continuous on compact time intervals,
	\begin{equation}
		\label{eq:numerical-driver-convergence}
		\sup_{0\le s\le S}
		|\xi_s^h-\xi_s|
		\longrightarrow0
		\qquad\text{almost surely}.
	\end{equation}
	Hence \eqref{eq:numerical-skorokhod-lipschitz} yields
	\begin{equation}
		\label{eq:numerical-reflected-convergence}
		\sup_{0\le s\le S}
		|Z_s^h-Z_s|
		\longrightarrow0
		\qquad\text{almost surely}.
	\end{equation}
	
	The convergence of the net constraining term now follows directly from
	the identity between the driving and reflected paths. Namely, write
	\begin{equation}
		\label{eq:numerical-net-regulators}
		\eta^h
		:=
		Z^h-\xi^h
		=
		L^{-,h}-L^{+,h},
		\qquad
		\eta
		:=
		Z-\xi
		=
		L^- -L^+.
	\end{equation}
	Combining
	\eqref{eq:numerical-driver-convergence} and
	\eqref{eq:numerical-reflected-convergence}, we obtain
	\begin{equation}
		\label{eq:numerical-net-regulator-convergence}
		\delta_h
		:=
		\sup_{0\le s\le S}
		|\eta_s^h-\eta_s|
		\longrightarrow0
		\qquad\text{almost surely}.
	\end{equation}
	
	For the approximation of the clock below, it remains to show convergence
	of the total boundary local time
	\[
	L^h:=L^{-,h}+L^{+,h}.
	\]
	This does not follow directly from the uniform convergence of the net
	constraining terms
	\[
	\eta^h=L^{-,h}-L^{+,h},
	\qquad
	\eta=L^- -L^+,
	\]
	but it follows from a short localization argument.
	
	Fix a sample path for which
	\eqref{eq:numerical-reflected-convergence} and
	\eqref{eq:numerical-net-regulator-convergence} hold. By continuity of
	\(Z\), choose a finite partition
	\[
	0=t_0<t_1<\cdots<t_m=S
	\]
	such that
	\[
	\sup_{s,t\in[t_{j-1},t_j]}|Z_s-Z_t|
	<
	\frac R2,
	\qquad j=1,\ldots,m.
	\]
	Since the two endpoints are \(2R\) apart, this implies that on each
	interval \(I=[t_{j-1},t_j]\), at least one endpoint remains at distance
	greater than \(3R/4\) from \(Z\). Since \(Z^h\to Z\) uniformly, for all
	sufficiently small \(h\), the paths \(Z\) and \(Z^h\) avoid the same
	endpoint on each such interval.
	
	Let
	\[
	\delta_h
	:=
	\sup_{0\le s\le S}|\eta_s^h-\eta_s|.
	\]
	Suppose first that both paths avoid \(R\) on
	\([a,b]=[t_{j-1},t_j]\). Then the upper regulators are constant there,
	and hence, for \(a\le t\le b\),
	\[
	L_t^h-L_a^h
	=
	\eta_t^h-\eta_a^h,
	\qquad
	L_t-L_a
	=
	\eta_t-\eta_a.
	\]
	Consequently,
	\[
	\sup_{a\le t\le b}
	\left|
	(L_t^h-L_a^h)-(L_t-L_a)
	\right|
	\le 2\delta_h.
	\]
	If instead both paths avoid \(-R\), the lower regulators are constant and
	\[
	L_t^h-L_a^h
	=
	-(\eta_t^h-\eta_a^h),
	\qquad
	L_t-L_a
	=
	-(\eta_t-\eta_a),
	\]
	so the same estimate holds.
	
	Since \(L_0^h=L_0=0\), iteration over the finitely many intervals gives
	\[
	\sup_{0\le s\le S}|L_s^h-L_s|
	\le
	2m\,\delta_h.
	\]
	Since \(\delta_h\to0\) almost surely, we conclude that
	\begin{equation}
		\label{eq:numerical-local-time-convergence}
		\sup_{0\le s\le S}|L_s^h-L_s|
		\longrightarrow0
		\qquad\text{almost surely}.
	\end{equation}

	\subsection{Approximation of the clock.}
	\label{subsec:numerical-clock}
	
	Fix \(\varepsilon>0\), let
	\[
	\alpha_s^\varepsilon:=\alpha_{s/\varepsilon},
	\qquad s\ge0,
	\]
	and consider the operational-time mesh
	\[
	s_k=kh,
	\qquad k\ge0,
	\]
	with step size \(h>0\).
	Using the discrete total boundary increment
	\begin{equation}
		\label{eq:numerical-total-boundary-increment}
		\Delta L_{k+1}^h
		:=
		\Delta L_{k+1}^{-,h}
		+
		\Delta L_{k+1}^{+,h},
	\end{equation}
	we construct the discrete clock and killing functional. 
	For this purpose,
	it is convenient to introduce the clock $A^{\varepsilon} $ and the operational-time killing functional \(B^\varepsilon\):
	\begin{equation}
		\label{eq:numerical-continuous-clock-killing}
		A^\varepsilon(s)
		=
		s+
		\int_0^s
		\rho_{\alpha_r^\varepsilon}\,dL_r,
		\qquad
		B^\varepsilon(s)
		=
		\int_0^s
		\kappa_{\alpha_r^\varepsilon}\,dL_r, \quad s\ge 0.
	\end{equation}
	We approximate these by the recursively defined discrete-time quantities
	\begin{equation}
		\label{eq:numerical-discrete-clock-killing}
		\begin{aligned}
			A_{k+1}^{\varepsilon,h}
			&=
			A_k^{\varepsilon,h}
			+h
			+\rho_{\alpha_{s_k}^\varepsilon}
			\Delta L_{k+1}^h,
			\\
			B_{k+1}^{\varepsilon,h}
			&=
			B_k^{\varepsilon,h}
			+\kappa_{\alpha_{s_k}^\varepsilon}
			\Delta L_{k+1}^h,
		\end{aligned}
		\qquad k\ge 0,
	\end{equation}
	with \(A_0^{\varepsilon,h}=B_0^{\varepsilon,h}=0\).
	The coefficient multiplying a boundary increment is evaluated at the
	left endpoint \(s_k\) of the operational-time step. 
	The homogenized scheme is obtained by replacing the state-dependent coefficients $(\rho_{\alpha_{s_k}^\varepsilon},\kappa_{\alpha_{s_k}^\varepsilon}) $ by the constants \((\bar\rho,\bar\kappa)\).
	
	We use piecewise-constant interpolation between consecutive operational mesh points:
	\[
	A^{\varepsilon,h}(s)
	=
	A_k^{\varepsilon,h},
	\qquad
	B^{\varepsilon,h}(s)
	=
	B_k^{\varepsilon,h},
	\qquad
	s\in[s_k,s_{k+1}).
	\]
	
	Since \(B^\varepsilon\) is defined in operational time, the corresponding
	physical-time killing functional is
	\begin{equation}
		\label{eq:numerical-physical-killing}
		K_\varepsilon(t)
		=
		B^\varepsilon(\gamma^\varepsilon(t)),
		\qquad
		\gamma^\varepsilon(t)
		=
		\inf\{s\ge0:A^\varepsilon(s)>t\}.
	\end{equation}
	At a physical time \(t\), we approximate \(\gamma^\varepsilon(t)\) by
	the first operational mesh point for which the discrete clock reaches
	level \(t\):
	\begin{equation}
		\label{eq:numerical-first-crossing}
		n_h^\varepsilon(t)
		=
		\min\{k:A_k^{\varepsilon,h}\ge t\}.
	\end{equation}
	The resulting numerical approximation of the Feynman--Kac observable and its continuous analogue are 
	\begin{equation}
		\label{eq:numerical-FK-observable}
		f\!\left(
		Z^h_{s_{n_h^\varepsilon(t)}}
		\right)
		\exp\!\left\{
		-B_{n_h^\varepsilon(t)}^{\varepsilon,h}
		\right\}, \quad
		f\!\left(Z_{\gamma^\varepsilon(t)}\right)
		\exp\!\left\{-B^\varepsilon(\gamma^\varepsilon(t))\right\}.
	\end{equation}
	
	We now verify consistency for fixed \(\varepsilon\).
	Let \(S>0\) be a bounded operational-time horizon. Almost surely,
	\(\alpha^\varepsilon\) has only finitely many jumps on \([0,S]\), and
	by \eqref{eq:numerical-local-time-convergence},
	\[
	\sup_{0\le s\le S}|L_s^h-L_s|
	\longrightarrow0
	\qquad\text{almost surely}.
	\]
	
	It remains to account for the jumps of the integrands in
	\eqref{eq:numerical-continuous-clock-killing}. Let
	\(\tau_1,\ldots,\tau_N\) be the jump times of \(\alpha^\varepsilon\) in
	\([0,S]\). Since \(L\) is continuous, for every \(\eta>0\) we may choose
	pairwise disjoint intervals \((a_j,b_j)\) containing \(\tau_j\) such that
	\[
	\sum_{j=1}^N
	\bigl(L_{b_j}-L_{a_j}\bigr)
	<
	\eta.
	\]
	For sufficiently small \(h\), every mesh interval containing a jump time
	is contained in one of these intervals. On each component of their
	complement, the chain is constant, so the sums in
	\eqref{eq:numerical-discrete-clock-killing} are constant multiples of
	increments of \(L^h\). These converge uniformly to the corresponding
	increments of \(L\).
	
	On the intervals containing the jumps, the contribution is small because
	the coefficients \(\rho_i\) and \(\kappa_i\) are bounded and, for each
	fixed choice of the intervals,
	\[
	\sum_{j=1}^N
	\bigl(L_{b_j}^h-L_{a_j}^h\bigr)
	\longrightarrow
	\sum_{j=1}^N
	\bigl(L_{b_j}-L_{a_j}\bigr).
	\]
	Since \(\eta\) is arbitrary, we obtain
	\begin{equation}
		\label{eq:numerical-clock-killing-convergence}
		\sup_{0\le s\le S}
		|A^{\varepsilon,h}(s)-A^\varepsilon(s)|
		\longrightarrow0,
		\qquad
		\sup_{0\le s\le S}
		|B^{\varepsilon,h}(s)-B^\varepsilon(s)|
		\longrightarrow0
	\end{equation}
	almost surely.
	
	Since \(s_k=kh\), the operational time corresponding to the first
	crossing index \(n_h^\varepsilon(t)\) is
	\begin{equation}
		\label{eq:numerical-discrete-inverse}
		\gamma_h^\varepsilon(t)
		:=
		s_{n_h^\varepsilon(t)}
		=
		h\,n_h^\varepsilon(t).
	\end{equation}
	
	Recall that
	\[
	A^\varepsilon(s)-A^\varepsilon(r)\ge s-r,
	\qquad 0\le r\le s,
	\]
	so that its inverse \(\gamma^\varepsilon\) is \(1\)-Lipschitz.
	Moreover, \(A_k^{\varepsilon,h}\ge kh\), and hence
	\[
	\gamma_h^\varepsilon(t)\le t+h.
	\]
	Set
	\[
	\delta_h
	:=
	\sup_{0\le s\le T+h}
	|A^{\varepsilon,h}(s)-A^\varepsilon(s)|.
	\]
	If \(n=n_h^\varepsilon(t)\ge1\), then
	\[
	A_{n-1}^{\varepsilon,h}<t\le A_n^{\varepsilon,h}.
	\]
	Using the \(1\)-Lipschitz property of \(\gamma^\varepsilon\) at the two
	mesh points \(s_{n-1}\) and \(s_n\), it follows that
	\[
	|\gamma_h^\varepsilon(t)-\gamma^\varepsilon(t)|
	\le
	\delta_h+h.
	\]
	The case \(t=0\) is immediate. Therefore,
	\[
	\sup_{0\le t\le T}
	|\gamma_h^\varepsilon(t)-\gamma^\varepsilon(t)|
	\le
	\sup_{0\le s\le T+h}
	|A^{\varepsilon,h}(s)-A^\varepsilon(s)|
	+h.
	\]
	By \eqref{eq:numerical-clock-killing-convergence},
	\begin{equation}
		\label{eq:numerical-inverse-convergence}
		\sup_{0\le t\le T}
		|\gamma_h^\varepsilon(t)-\gamma^\varepsilon(t)|
		\longrightarrow0
		\qquad\text{almost surely},
	\end{equation}
	for every \(T>0\).
	
	Finally, with the piecewise-constant interpolation above,
	\[
	Z^h_{s_{n_h^\varepsilon(t)}}
	=
	Z^h_{\gamma_h^\varepsilon(t)},
	\qquad
	B_{n_h^\varepsilon(t)}^{\varepsilon,h}
	=
	B^{\varepsilon,h}(\gamma_h^\varepsilon(t)).
	\]
	Since \(Z\) and \(B^\varepsilon\) are continuous,
	\eqref{eq:numerical-reflected-convergence},
	\eqref{eq:numerical-clock-killing-convergence}, and
	\eqref{eq:numerical-inverse-convergence} imply, for every \(T>0\) and
	every bounded continuous \(f\),
	\begin{equation}
		\label{eq:numerical-observable-convergence}
		\begin{aligned}
			\sup_{0\le t\le T}
			\Bigg|
			&
			f\!\left(
			Z^h_{s_{n_h^\varepsilon(t)}}
			\right)
			\exp\!\left\{
			-B_{n_h^\varepsilon(t)}^{\varepsilon,h}
			\right\}
			\\
			&\qquad
			-
			f\!\left(
			Z_{\gamma^\varepsilon(t)}
			\right)
			e^{-B^\varepsilon(\gamma^\varepsilon(t))}
			\Bigg|
			\longrightarrow0
			\qquad\text{almost surely}.
		\end{aligned}
	\end{equation}
	
	There is no difficulty in choosing the operational-time horizon:
	we know that 
	\(\gamma^\varepsilon(t)\le t\), while
	\eqref{eq:numerical-first-crossing} gives
	\(\gamma_h^\varepsilon(t)\le t+h\).
	
	The additive-functional time change used for sticky reflection and the
	multiplicative-functional representation of elastic killing are standard;
	for the corresponding interval construction see
	\cite{kostrykin2010intervals}. 
	The convergence statements above concern
	\(h\downarrow0\) for fixed \(\varepsilon>0\).
	The dependence of the reported \(\varepsilon\)-slopes on the operational
	mesh is assessed separately by the mesh-refinement experiment described in
	Subsection~\ref{subsec:quenched-errors}.
	
	\section*{Acknowledgments}
	
	Alexis Anagnostakis was supported by the Centro de Modelamiento Matemático (CMM) BASAL fund FB210005 for Centers of Excellence from ANID-Chile. 
	Fausto Colantoni thanks Ulm University and the INdAM-GNAMPA group for their support.
	
	\section*{Declaration of AI use}
	
	Generative AI tools were used to improve the exposition, to suggest
	refinements in the presentation of some arguments, and to assist with the
	numerical implementation. The mathematical results, arguments, and proofs
	were developed by the authors. All AI-assisted output was critically reviewed
	and independently verified by the authors, who take full responsibility for
	the content of the manuscript.


\end{document}